\documentclass[11pt, a4paper, leqno]{amsart}
\usepackage[T3,T1]{fontenc}
\DeclareSymbolFont{tipa}{T3}{cmr}{m}{n}
\DeclareMathAccent{\invbreve}{\mathalpha}{tipa}{16}
\usepackage{esint} 
\usepackage{graphicx} 
\usepackage[dvipsnames]{xcolor}
\usepackage{varwidth}
\usepackage[colorlinks=true,linkcolor=Red,citecolor=Green]{hyperref}

\usepackage{amsthm}
\usepackage{amsmath}
\usepackage{amsfonts}
\usepackage{amssymb}
\usepackage{amscd}
\usepackage{mathrsfs}
\usepackage{enumerate}
\usepackage{bm}
\usepackage{dsfont}
\usepackage{yhmath}

\usepackage[utf8]{inputenc}
\usepackage[english]{babel}
\usepackage{anysize}

\usepackage{lipsum}

\newcommand{\R}{\mathbb{R}}

\newcommand{\Op}{\operatorname{Op}}

\newcommand{\supp}{\operatorname{supp}}
\newcommand{\T}{\mathbb{T}} 
\newcommand{\sgn}{\operatorname{sgn}}
\newcommand{\w}{\mathbf{w}}

\newtheorem{prop}{Proposition}
\newtheorem{lemma}{Lemma}

\newtheorem{definition}{Definition}

\newtheorem{teor}{Theorem}

\theoremstyle{remark}
\newtheorem{example}{Example}

\definecolor{darkgreen}{RGB}{0,100,10}

\theoremstyle{remark}
\newtheorem{remark}{Remark}

\title[Quantum limits of the Martinet sub-Laplacian]{Quantum limits of the Martinet sub-Laplacian}
\author{Víctor Arnaiz}

\address{Institut de Mathématiques de Bordeaux}

\email{victor.arnaiz-solorzano@math.u-bordeaux.fr}

\begin{document}

\begin{abstract}
In this article we study the semiclassical asymptotics of the Martinet sub-Laplacian on the flat toroidal cylinder $M = \R \times \mathbb{T}^2$. We describe the asymptotic distribution of sequences of eigenfunctions oscillating at  different scales prefixed by Rothschild-Stein estimates via the introduction of adapted two-microlocal semiclassical measures. We obtain concentration and invariance properties of these measures in terms of effective dynamics governed by harmonic or an-harmonic oscillators depending on the regime, and we show additional regularity properties with respect to critical points of the eigenvalues of the Montgomery family of quartic oscillators.
\end{abstract}

\maketitle

\section{Introduction}

Let $\mathbb{T}^2 := \R^2/2\pi \mathbb{Z}^2$ be the two dimensional flat torus. In this work we are interested in the study of \textit{quantum limits} of the \textit{Martinet sub-Laplacian} $\Delta_{\mathcal{M}}$ defined on the flat toroidal cylinder $M := \R\times  \mathbb{T}^2$ by the formally symmetric  operator
\begin{equation}
\label{e:Martinet}
\Delta_{\mathcal{M}} =  -\mathcal{X}_1^2 - \mathcal{X}_2^2,
\end{equation}
where, in canonical coordinates $(x,y,z) \in M$, the vector fields $\mathcal{X}_1$ and $\mathcal{X}_2$ are written by
\begin{equation}
\label{e:vector_fields}
\mathcal{X}_1 = \partial_x, \quad \mathcal{X}_2 = \partial_y + x^2 \partial_z.
\end{equation}
One can understand this operator as a \textit{sub-elliptic} version of the standard elliptic Laplacian $\Delta = -\partial_x^2 -\partial_y^2 - \partial_z^2$ with a degeneracy at the point $x = 0$. It can also be interpreted as a magnetic Laplacian \cite{Montgomery95} with magnetic potential given by $(A_x,A_y) = (0,-x^2)$, providing a basic example of magnetic field $B = (\partial_x A_y - \partial_y A_x)$ with zero locus at $x = 0$.

From the spectral point of view, we consider the operator $\Delta_{\mathcal{M}}$ as un-bounded operator on the Hilbert space $L^2_0(M) \subset L^2(M)$ defined by: 
\begin{align}
\label{e:0_L^2}
L_0^2(M) := \big \{ \psi \in L^2(M) \, : \, \mathcal{F}_z \psi(\cdot, \cdot ,0) = 0 \big \},
 \end{align} 
where $\mathcal{F}_z \psi(\cdot, \cdot, 0)$ stands for the $0$-th Fourier coefficient in the variable $z$. That is, we restrict ourselves to consider functions with zero-mean in the $z$-variable. We define the operator $\Delta_{\mathcal{M}}$ with domain 
$$
\mathcal{D}(\Delta_{\mathcal{M}}) = \{ \psi \in L^2_0(M) \, : \, \Delta_{\mathcal{M}} \psi \in L^2_0(M) \},
$$  
where $\Delta_{\mathcal{M}} \psi$ is considered in distributional sense. Hence the operator $(\mathcal{D}(\Delta_{\mathcal{M}}), \Delta_\mathcal{M})$ is self-adjoint and has compact resolvent (see Appendix \ref{a:spectral_properties} for a brief proof of these facts). In particular, the spectrum of $\Delta_{\mathcal{M}}$ is discrete and satisfies
$$
\operatorname{Sp}_{L^2_0(M)} ( \Delta_{\mathcal{M}}) = \{ 0 < \lambda^2_0 \leq \lambda^2_1 \leq \cdots \leq \lambda^2_j \leq \cdots \to + \infty \}.
$$
On the other hand, regarding the sub-ellipticity properties of $\Delta_{\mathcal{M}}$, we first notice that the horizontal distribution $\mathscr{D} := \operatorname{span}\{ \mathcal{X}_1, \mathcal{X}_2 \}$
satisfies the Hörmander condition \cite{Hor67} with step two, that is: 
$$
\operatorname{Lie}^{(2)}(\mathscr{D}) := \mathscr{D} + [\mathscr{D}, \mathscr{D}] + [\mathscr{D}, [\mathscr{D}, \mathscr{D}]] = TM.
$$ 
More precisely, the vector fields $\mathcal{X}_1$, $\mathcal{X}_2$, and $\mathcal{X}_3 = [\mathcal{X}_1,\mathcal{X}_2] = 2x \partial_z$ span the tangent space $T M$ at any point $q \in M$ except at the singular set 
\begin{equation}
\label{e:critical_set}
\mathscr{S} := \{ q =  (x,y,z) \in M \, : \, x = 0 \},
\end{equation}
where a further commutator $\mathcal{X}_4 = [\mathcal{X}_1,[\mathcal{X}_1,\mathcal{X}_2]] = 2 \partial_z$ is required.  Then, by Hörmander's Theorem \cite[Th. 1.1]{Hor67}, the operator  $\Delta_{\mathcal{M}}$ is hypoelliptic.  Moreover  $\Delta_{\mathcal{M}}$ satisfies  sharp hypoelliptic estimates established by Rotchschild and Stein \cite{RothschildStein76} (see \eqref{e:sub_elliptic_estimates} below).
\medskip

Let $(\psi_j)_{j \in \mathbb{N}}$ be a sequence of solutions to the equation
\begin{equation}
\label{e:non_semiclassical_eigenvalue_problem}
\Delta_{\mathcal{M}} \, \psi_j = \lambda^2_j \, \psi_j, \quad j \in \mathbb{N}.
 \end{equation}
 We are interested in studying the accumulation points for the weak-$\star$ topology of the sequence of probability densities $\vert \psi_j \vert^2 dx dy dz$ as $j \to +\infty$. By sequentially compactness, there exists a sub-sequence $(\psi_j)$ and a positive Radon measure $\nu \in \mathcal{M}_+(M)$ such that, for each $b \in \mathcal{C}_c(M)$,
$$
\lim_{j \to +\infty} \int_{M} b \, \vert \psi_j \vert^2 \, dxdydz = \int_{M} b \, d\nu.
$$ 
These measures $\nu$ obtained as weak-$\star$ limits of sequences of $L^2$-densities of eigenfunctions are usually called \textit{quantum limits}. Notice that $\nu$ is not in general a probability measure due to the non-compacity of the manifold $M$, since part of the mass of the sequence $(\psi_h)$ could scape to infinity in the $x$ variable. We will however restrict our study to the measure $\nu$, which captures the most relevant accumulation features of the sequence $(\psi_h)$ regarding sub-ellipticity of $\Delta_{\mathcal{M}}$. 

The study of quantum limits for sub-Riemannian Laplacians starts with the work of Colin-de-Verdière, Hillairet, and Trèlat \cite{Colin_de_Verdiere18} (see also \cite{Colin_de_Verdiere22}),  where the authors study the quantum limits of sub-Riemannian Laplacians in the particular 3D contact case, and state the first quantum ergodicity theorem in this context. In the 3D contact case, one considers a contact manifold $(M,\alpha)$ of dimension $3$. The contact structure gives a natural decomposition $TM = \mathscr{D} \oplus L_{\mathcal{Z}}$, where $\mathscr{D} = \ker \alpha$ is the contact distribution associated with the contact one-form $\alpha$, and $L_{\mathcal{Z}}$ is the line bundle spanned by the Reeb vector field $\mathcal{Z}$. Any sub-Laplacian $\Delta_{\operatorname{sR}} = \mathcal{X}_1^* \mathcal{X}_1 + \mathcal{X}_2^* \mathcal{X}_2$ satisfying $\mathscr{D} = \operatorname{span} \{ \mathcal{X}_1, \mathcal{X}_2 \}$ satisfies then the Hörmander condition of step one at any point. In \cite{Colin_de_Verdiere18}, the authors show that the non-compact part of the measure $\nu_\infty$ (coming roughly from high oscillations of the sequence $(\psi_j)$ in the direction of $\mathcal{Z} = [\mathcal{X}_1, \mathcal{X}_2] \, \operatorname{mod} \, \mathscr{D}$) is in this case invariant by the Reeb flow generated by $\mathcal{Z}$, which replaces the geodesic vector field as effective dynamical system arising in the analogous case of the Laplace-Beltrami operator on a Riemannian manifold (see for instance \cite[Th. 5.4]{Zworski12}). More recently, \cite{Ar_Riv24}, \cite{Riv23}, this study has been generalized to consider hypoelliptic perturbations of 3D contact sub-Laplacians. In these works, new two-microlocal semiclassical measures adapted to the different semiclassical scales predetermined by Rothschild-Stein estimates are introduced. This study refines the microlocal description of \cite{Colin_de_Verdiere18} and allows to capture the different effective propagation phenomena (reflectig the deformation due to perturbation) at each two-microlocal scale. In the framework of Heisenberg-type or more generally nilpotent Lie groups, the works \cite{FermanianFischer21,Fermanian19, Fermanian20} give analogous results on semiclassical asymptotics and defect measures of sub-Laplacians in this more algebraic context. See also \cite{Fermanian_Letrouit21} for applications to control and observability of the Schrödinger equations on groups of Heisenberg type.  The Baouendi-Grushin operator has also received much attention in the recent years. In \cite{ArnaizSun22}, the introduction of two-microlocal semiclassical measures allows to fully describe the concentration and invariance properties of quantum limits (or more generally quasimodes) in the sub-elliptic regime, which is a key ingredient in the description of propagation of singularities. This has direct applications to the study of observability and stabilization (see also \cite{BurqSun22}, \cite{Sun20} for previous works on observability and stabilization of Schrödinger equations governed by Baouendi-Grushin operators, and \cite{Letrouit23} for the study of lack of observability of sub-elliptic wave-equations). 

Comparing to the 3D contact case, the Martinet sub-Laplacian presents two main new issues. On the one-hand, $\Delta_{\mathcal{M}}$ is a step-two sub-Laplacian, which produces a much richer structure on the regions of phase space  accesible by eigenfunctions of the system. On the other hand, the distribution $\mathscr{D}$ is not regular, meaning that it has different step on $M\setminus \mathscr{S}$ and on the singular set $\mathscr{S}$. This generates an additional geometric problem when coupling these two different regimes.  In this work, we introduce new adapted two-microlocal semiclassical measures which allow to describe in full generality the concentration and invariance properties of quantum limits of the Martinet sub-Laplacian. Associated to the region $M \setminus \mathscr{S}$, the corresponding measures satisfy \textit{high-oscillation} invariance properties ruled by the harmonic oscillator, similarly as in the 3D contact case \cite{Ar_Riv24}, and a drift invariance given in terms of the Reeb vector field associated to the contact distribution, as stated in \cite{Colin_de_Verdiere18}; while in the singular region $\mathscr{S}$, the high-oscillation invariance is governed by a quartic oscillator, similar to the case of the Engel group \cite{Benedetto24}, while the drift invariance appears as an \textit{abnormal} propagation at larger scale than that of the Reeb vector field. This abnormal drift can be viewed as the effect of a non-commuting term with the quartic oscillator (which would nevertheless  vanish in the normal-form issued from the usual harmonic oscillator). This exotic term is the responsable to the abnormal propagation of singularities described in \cite{CdV_Letrouit22}. Moreover, at some critical points, this  drift invariance vanishes (singularities can propagate and positive or negative speeds), leading to obstructions to dispersion for evolution equations governed by the Martinet sub-Laplacian. In this work we establish new weak invariance properties by a further two-microlocalization near these critical points in phase space, similarly as the result of \cite[Th. 1.14]{Benedetto24} in the case of the Engel group.
\medskip

\subsection{The non-compact part of the measure $\nu$} Let us now fix some semiclassical notations. We set:
$$
h_j := \frac{1}{\lambda_j}, \quad \psi_{h_j} := \psi_j, \quad j \in \mathbb{N},
$$ 
and let us drop the index $j$, assuming in the sequel that passing to the limit $h \to 0^+$ is always done through a suitable subsequence of the original sequence $(h_j)_{j \in \mathbb{N}}$. The eigenvalue problem \eqref{e:non_semiclassical_eigenvalue_problem} then reads:
\begin{equation}
\label{e:semiclassical_eigenvalue_problem}
\big( h^2 \Delta_{\mathcal{M}} - 1 \big) \psi_h = 0, \quad \Vert \psi_h \Vert_{L^2(M)} = 1.
\end{equation}

In the study of quantum limits it is convenient to lift the involved wave functions to phase space in order keep  information of position and momentum at same time. This can be done by considering the  associated sequence of Wigner distributions of $(\psi_h)$. Let $a \in \mathcal{C}_c^\infty(T^*M)$, the Wigner distribution $W_h$ is defined by the map
\begin{equation}
\label{e:Wigner_distribution_compact}
W_h \, : \, a \longmapsto W_h(a) := \big \langle \Op^{\w}_h(a) \psi_h, \psi_h \big \rangle_{L^2(M)},
\end{equation}
where $\Op_h^{\w}(a)$ stands for the semiclassical Weyl quantization of the symbol $a$ (see Appendix \ref{a:appendix}). By sequentially compactness in the space of distributions and some positivity arguments (for instance using the sharp G$\ring{\text{a}}$rding inequality), one can show the existence of a sub-sequence $(\psi_h)$, which we do not relabel, and a positive Radon measure $\mu \in \mathcal{M}_+(T^*M)$ such that, for all $a \in \mathcal{C}_c^\infty(T^*M)$,
\begin{equation}
\label{e:semiclassical_measure}
\lim_{h \to 0} \big \langle \Op^{\w}_h(a) \psi_h, \psi_h \big \rangle_{L^2(M)} = \int_{T^*M} a \, d\mu.
\end{equation}
The measure $\mu$ is called the \textit{semiclassical measure} associated to the sub-sequence $(\psi_h)$. It verifies that
$$
\int_{T^*M} d\mu \leq \int_{M} d\nu,
$$
and we remark that in general this inequality is strict. Moreover, the semiclassical measure $\mu$ satisfies the following two properties. Let $H_\mathcal{M}$ be the symbol of the operator $h^2 \Delta_{\mathcal{M}}$, given in coordinates $(x,y,z,\xi,\eta,\zeta) \in T^*M$ by
$$
H_{\mathcal{M}}(x,y,z,\xi,\eta,\zeta) = \xi^2 + (\eta + x^2 \zeta)^2,
$$
and let $\phi_t^{H_{\mathcal{M}}}:T^*M \to T^*M$ be the Hamiltonian flow generated by $H_{\mathcal{M}}$, then:
\begin{itemize}
\item $\operatorname{supp} \mu \subset H_{\mathcal{M}}^{-1}(1) \subset T^*M$.
\medskip

\item $(\phi_t^{H_{\mathcal{M}}})_* \mu = \mu$ for all $t \in \R$.
\end{itemize}
\smallskip

 Let $\pi : T^*M \to M$ be the canonical projection map, we define $\nu_\infty =\nu - \pi_* \mu $. Due to the lack of compacity of the set $H_\mathcal{M}^{-1}(1)$, the measure $\nu_\infty$ is in general non null.  The study of the measure $\nu_\infty$ is the main goal of this article. 

\begin{remark}
We emphasize that the measure $\nu$ is not generally a probability measure due to the non-compacity of the manifold $M$, but it would be if, instead of on $M$, we were to define the Martinet sub-Laplacian on a compact manifold.  This could be done for instance considering the operator
$$
-\partial_x^2 - (\partial_y + V(x) \partial_z)^2, \quad V(x) = \sin^2\left( \frac{x}{2} \right), \quad (x,y,z) \in \T^3.
$$
Our results stated below remain valid in this more general case with some modifications. Notice that the function $V(x)$ has non-degenerate critical values at $x \equiv 0 \, (\text{mod} \, \pi)$, which replaces the singular set $\mathscr{S}$ in this case. We prefer however to restrict the exposition to $\Delta_{\mathcal{M}}$ on $M$ for the sake of clarity, and since $\nu$ contains the most relevant information regarding sub-ellipticity  of $\Delta_{\mathcal{M}}$.
\end{remark}

One of the difficulties of this study is the different allowed scales of oscillation of the sequence $(\psi_h)$ in the directions $\mathcal{X}_j$ for $j=1,\ldots 4$. These scales are prescribed by the Rothschild-Stein theorem \cite{RothschildStein76}.  Let us introduce the adapted Sobolev spaces $\mathscr{H}_{\mathcal{M}}^k$ defined with respect to the norms:
$$
\Vert \psi \Vert_{\mathscr{H}_{\mathcal{M}}^k} := \sum_{s = 0}^k  \sum_{i_1, \ldots, i_s =1  }^2 \left( \int_{M}  \big \vert \mathcal{X}_{i_1} \cdots \mathcal{X}_{i_s} \psi \big \vert^2  dx dy dz \right)^{1/2}.
$$ 
By Rothschild-Stein theorem \cite[Th. 18]{RothschildStein76}, for any $n \in \mathbb{N}$, and any $f \in \mathcal{C}_c^\infty(M)$, 
\begin{equation}
\label{e:sub_elliptic_estimates}
\Vert f \Vert_{H^{n/3}(M)} \lesssim \Vert f \Vert_{\mathscr{H}^{2n}_{\mathcal{M}}} \lesssim \Vert \Delta_{\mathcal{M}}^n f \Vert_{L^2(M)} + \Vert f \Vert_{L^2(M)},
\end{equation}
where the implicit constants in these estimates depend only on $M$ and $n$. Estimate \eqref{e:sub_elliptic_estimates} implies the following semiclassical sub-elliptic estimates for the sequence of solutions $(\psi_h)$ of \eqref{e:semiclassical_eigenvalue_problem}:
\begin{align}
\label{e:R-S_1}
\Vert h \, \mathcal{X}_j \, \psi_h \Vert_{L^2(M)} & \lesssim 1, \quad j = 1, 2, \\[0.2cm]
\label{e:R-S_2}
\Vert h^2\mathcal{X}_3 \, \psi_h \Vert_{L^2(M)} & \lesssim 1, \\[0.2cm]
\label{e:R-S_3}
\Vert h^3 \mathcal{X}_4 \, \psi_h \Vert_{L^2(M)} & \lesssim 1,
\end{align}
where the implicit constants in the above estimates dependend only on $M$.

\subsection{Main results}
\label{s:main_results}

We now state our main results describing the measure $\nu_\infty$. Define:
\begin{align}
\label{e:w_and_sigma}
w = w(x,\eta,\zeta) := \eta + x^2 \zeta, \quad \sigma = \sigma(x,\zeta) := 2x\zeta.
\end{align} 
Let us consider $a \in \mathcal{C}_c^\infty(M_{x,y,z} \times \R^4_{w,\xi, \sigma,\zeta})$ and we denote the class of those symbols by $\mathcal{S}_{\mathcal{M}}$. Using pseudodifferential calculus and estimates \eqref{e:R-S_1}, \eqref{e:R-S_2} and \eqref{e:R-S_3}, we will see that, for any $a \in \mathcal{S}_{\mathcal{M}}$, the distribution $\Op_h^{\w}(a) \psi_h$ 
is bounded in $L^2(M)$ for each solution $\psi_h$ to \eqref{e:semiclassical_eigenvalue_problem}. Let $\chi \in \mathcal{C}_c^\infty(\R;[0,1])$ be a bump function near zero with support in $[-1,1]$ and which equals one in $[-\frac{1}{2},\frac{1}{2}]$, and set $\breve{\chi} := 1 - \chi$. We fix this cut-off function $\chi$ for the rest of the article. Given $R > 0$, we define:
\begin{equation}
\label{e:R_symbol}
\mathbf{a}_{h,R}(x,y,z,\xi,\eta,\zeta) := \breve{\chi}\left( \frac{ \zeta}{R} \right) a(x,y,z, w, \xi, h\sigma, h^2\zeta).
\end{equation}
We introduce the adapted Wigner distribution
\begin{equation}
\label{e:original_Wigner_distribution}
\mathcal{S}_{\mathcal{M}} \ni a \longmapsto I_{h,R}(a) := \big \langle \Op_h^{\w}(\mathbf{a}_{h,R}) \psi_h, \psi_h \big \rangle_{L^2(M)},
\end{equation}
which extends the Wigner distribution \eqref{e:Wigner_distribution_compact} to the noncompact part of $H_{\mathcal{M}}^{-1}(1)$.

\subsubsection{Existence of two-microlocal semiclassical measures}

Our first result establish the existence of two microlocal semiclassical measures lifting the measure $\nu_\infty$ to phase space.
\begin{teor}
\label{t:existence}
Let $(\psi_h) \subset L^2(M)$ be a normalized sequence of solutions to \eqref{e:semiclassical_eigenvalue_problem}. Then there exists a subsequence $(h)$ and positive Radon measures\footnote{Let $\mathcal{H}$ be a Hilbert space, $\mathcal{L}^1(\mathcal{H})$ denotes the space of trace-class operators on $\mathcal{H}$.}
\begin{align*}
M_1 & \in \mathcal{M}_+(\mathbb{T}^2_{y,z} \times \R_\eta \times \R^*_\zeta ; \mathcal{L}^1(L^2(\R_x))), \\[0.2cm]
M_2 & \in \mathcal{M}_+(M_{x,y,z} \setminus \mathscr{S} \times \R^*_\sigma; \mathcal{L}^1(L^2(\R_w))), \\[0.2cm]
m_3^{\jmath} & \in \mathcal{M}_+(\mathbb{T}^2_{y,z} \times \R_\xi \times \R_\eta \times \R_\sigma \times \R_\varsigma), \quad \jmath \in  \{0,1\}, \\[0.2cm]
m_4 & \in \mathcal{M}_+(M_{x,y,z}\setminus \mathscr{S} \times \mathbb{T}^2_{(y,z)} \times \R_w \times \R_\xi),
\end{align*} 
such that, for each $a \in \mathcal{S}_{\mathcal{M}}$ one has:
\begin{align*}
\lim_{R \to + \infty} \lim_{h \to 0 } I_{h,R}(a) & =   \int_{\mathbb{T}^2_{y,z} \times \R_\eta \times \R^*_\zeta} \small \operatorname{Tr} \Big( \Op_{(x,\xi)}^\R\big( a \big(0 , y, z,  w(x,\eta,\zeta), \xi, \sigma(x,\zeta),  \zeta \big) \big) dM_1 \Big)  \\[0.2cm]
 & \quad +  \int_{M_{x,y,z}\setminus \mathscr{S} \times \R^*_\sigma} \operatorname{Tr} \left( \Op_{(w,\upsilon)}^\R \left( a \left( x,y,z, w, \sigma \upsilon, \sigma, 0 \right) \right) dM_2 \right) \\[0.2cm]
 & \quad + \sum_{ \jmath \in \{0,1 \}} \int_{\mathbb{T}^2_{y,z} \times \R_\xi \times \R_\eta \times \R_\sigma \times \R_\varsigma} a \big(0, y ,z, \eta + (-1)^{\jmath} \varsigma^2, \xi, \sigma, 0 \big) dm_3^{\jmath} \\[0.2cm]
 &  \quad + \int_{M_{x,y,z}\setminus \mathscr{S} \times \R_w \times \R_\xi} a\big(x,y,z,w,\xi,0,0 \big) dm_4.
\end{align*}
Moreover,
\begin{align}
\label{e:projection_nu_infty}
\nu_\infty  =   \int_{\R_\eta \times \R^*_\zeta} \operatorname{Tr} dM_1  +  \int_{ \R^*_\sigma} \operatorname{Tr}  dM_2 + \sum_{\jmath \in \{0,1 \}} \int_{\R_\xi \times \R_\eta \times \R_\sigma \times \R_\varsigma} dm_3^{\jmath}  + \int_{\R_w \times \R_\xi} dm_4.
\end{align}
\end{teor}
\begin{remark}
The measures $M_1$, $M_2$, $m_3^\jmath$ and $m_4$ capture the asymptotic phase-space distribution of the sequence $(\psi_h)$ in different regimes. Roughly speaking, the measure $M_1$ captures oscillations in the direction of $\mathcal{X}_4$ at scale $h^{-3}$, which is critical with respect to the Rothschild-Stein estimate \eqref{e:R-S_3}, and hence this measure concentrates on the singular set $\mathscr{S}$. On the other hand, the measure $m_3^\jmath$ is the scalar counterpart of $M_1$, oscillating in the direction of $\mathcal{X}_4$ at lower scales than $h^{-3}$. We highlight that the index $\jmath \in \{0,1\}$ stores the information about the direction of oscillation in the variable $\zeta$. Similarly, the measure $M_2$ captures oscillations of the sequence $(\psi_h)$ in the direction of $\mathcal{X}_3$ at scale $h^{-2}$, which is in this case critical for the Rothschild-Stein estimate \eqref{e:R-S_2}. This measure concentrates away from the singular set $\mathscr{S}$. Finally, the measure $m_4$ is the scalar counterpart of $M_2$, which captures oscillations at scales lower than $h^{-2}$ in the direction of $\mathcal{X}_3$.
\end{remark}
\subsubsection{High oscillating invariance and concentration properties}
Let us introduce the quartic oscillator on $L^2(\R_x)$:
\begin{align}
\label{e:effective_quantum_hamiltonian_1}
\widehat{H}_{\eta,\zeta} & := D_x^2 + \left( \eta  + x^2 \zeta \right)^2, \quad (\eta,\zeta) \in \R_\eta \times \R_\zeta^*, \quad D_x := \frac{1}{i} \partial_x.
\end{align}

\begin{prop}{\cite[Corollary B.4]{Bahouri23}}
\label{p:simple_montgomery_martinet}
For each $(\eta,\zeta) \in \R_\eta \times \R_\zeta^*$, the spectrum of $\widehat{H}_{\eta,\zeta}$ is discrete,  with increasing to infinity sequence of eigenvalues given by:
\begin{equation}
\label{e:Martinet_Montgomery}
\lambda_k(\eta,\zeta) = \vert \zeta \vert^{2/3}\Lambda_k \left( \frac{\eta \sgn(\zeta)}{  \vert \zeta \vert^{1/3}} \right),
\end{equation}
where each eigenvalue is simple and $\Lambda_k(\mu)$ is the corresponding simple eigenvalue of the quartic oscillator
\begin{equation}
\label{e:Montgomery_family}
\widehat{H}_\mu := D_x^2 + (\mu + x^2)^2, \quad \mu \in \R.
\end{equation}
\end{prop}
\begin{remark}
By \cite[Prop. 1.1]{Helffer_Leautaud22} (see also \cite{Bahouri23}), the quartic oscillator $\widehat{H}_\mu$ with domain
$$
\mathcal{D}(\widehat{H}_\mu) = \big \{ \psi \in L^2(\R) \, : \, \widehat{H}_\mu \psi \in L^2(\R) \big \}
$$
is self-adjoint and has compact resolvent. Hence its spectrum is discrete, given by $\{ \Lambda_k(\mu) \}_{k\in \mathbb{N}}$ satisfying
$$
0 < \Lambda_1(\mu) < \cdots < \Lambda_k(\mu) < \cdots \to + \infty.
$$ 
Moreover, each eigenvalue $\Lambda_k(\mu)$ is simple, depending analytically on $\mu \in \R$.
\end{remark}
The spectral properties of the Montgomery family of operators $\widehat{H}_\mu$ has been the object of much study in many different settings (see \cite{Helffer_Leautaud22} and the references therein), such as quantum mechanics, irreducible representation of certain nilpotent Lie groups of rank 3 or the study of Schrödinger operators with vanishing magentic fields \cite{Montgomery95} (hence the name of Montgomery attributed to this family).

We remark that, although the quartic oscillator \eqref{e:effective_quantum_hamiltonian_1} governs the regime giving the measure $M_1$ associated with the singular set $\mathscr{S}$, the measure $M_2$ concentrating on the set $M \setminus \mathscr{S}$ will satisfies invariance properties issued from the usual harmonic oscillator. We introduce the quantum harmonic oscillator:
\begin{align}
\label{e:effective_quantum_hamiltonian_2}
\widehat{G}_\sigma & :=\sigma^2 D_w^2 + w^2, \quad \sigma \in \R_\sigma^*.
\end{align}
For each $\sigma \in \R^*$, the sepectrum of $\widehat{G}_\sigma$ is discrete, each eigenvalue being simple and given by:
$$
\nu_j(\sigma) = \vert \sigma \vert \left(2 j + 1 \right), \quad j \in \mathbb{N}.
$$
The measures $m_3^\jmath$ and $m_4$ are respectively the scalar counterparts of $M_1$ and $M_2$ in the subcritical regimes. To describe the high-oscillation invariance associated with these measures, we also introduce the classical Hamiltonians:
\begin{align*}
H^{\jmath}_\eta(\varsigma,\xi) & := \xi^2 + (\eta + (-1)^{\jmath}\varsigma^2)^2, \quad \jmath \in \{0,1\}, \\[0.2cm]
G(w,\xi) & := \xi^2 + w^2.
\end{align*}
Our next result reads as follows:
\begin{teor}
\label{t:oscillation_invariance}
The measures $M_1$, $M_2$, $m_3^{\jmath}$, and $m_4$ verify respectively the following invariance properties:
\begin{align}
\label{e:invariance_1} &0 =  [\widehat{H}_{\eta,\zeta}, M_1 ]_{L^2(\R_x)}, \\[0.2cm]
\label{e:invariance_2} & 0 = [\widehat{G}_\sigma, M_2 ]_{L^2(\R_w)}, \\[0.2cm]
\label{e:invariance_3} & 0 = \{H_\eta^{\jmath}, m_3^{\jmath} \}_{\varsigma,\xi}, \\[0.2cm]
\label{e:invariance_4} & 0 = \{G, m_4 \}_{w,\xi},
\end{align}
where these invariance properties have to be understood in the sense of measures.
\end{teor}
\begin{remark}
The invariance property \eqref{e:invariance_1} has the following consequence. Fix $(\eta,\zeta) \in \R^* \times \R^*$. Let $(\varphi_k(\eta,\zeta))_{k \in \mathbb{N}}$ be the orthonormal basis of $L^2(\R_x)$ given by eigenfunctions of $\widehat{H}_{\eta,\zeta}$ with associated eigenvalues $\lambda_k(\eta,\zeta)$.  Since for each $k \in \mathbb{N}$, the eigenvalue $\lambda_k(\eta,\zeta)$ is simple, there exists a measure $\mu_{1,k} \in \mathcal{M}_+(\T^2_{y,z} \times \R_\eta \times \R^*_\zeta)$ such that:
\begin{equation}
\label{e:eigenvalue_measure_1}
M_{\eta,\zeta} \, \varphi_k = \mu_{1,k} \, \varphi_k, \quad k \in \mathbb{N}.
\end{equation}
Similarly, let $(\phi_j(\sigma))_{j \in \mathbb{N}}$ be the orthonormal basis of eigenfunctions of $\widehat{G}_\sigma$ with associated eigenvalues $\nu_j(\sigma)$. The invariance property  \eqref{e:invariance_2} and the simplicity of the eigenvalues $\nu_j(\sigma)$ imply in this case that there exists positive Radon measures $\mu_{2,j} \in \mathcal{M}_+(M\setminus \mathscr{S} \times \mathbb{T}^2_{y,z} \times \R^*_\sigma)$ such that
\begin{equation}
\label{e:eigenvalue_measure_2}
M_2 \, \phi_j  = \mu_{2,j} \, \phi_j, \quad j \in \mathbb{N}.
\end{equation}
\end{remark}
Our next result describes the concentration properties of the two-microlocal semiclassical measures $M_1$, $M_2$, $m_3^\jmath$, and $m_4$:
\begin{teor}
\label{t:concentration}
The measures $\mu_{1,k}$, $\mu_{2,j}$, $m_3^{\jmath}$, and $m_4$ verify respectively the following concentration properties:
\begin{align}
\label{e:concentration_1}
&  \operatorname{supp} \mu_{1,k} \subset  \big \{ (y,z,\eta,\zeta) \in \mathbb{T}^2_{y,z} \times \R_\eta \times \R^*_\zeta \, : \,  \lambda_k(\eta,\zeta) = 1 \big \}, \\[0.2cm] 
\label{e:concentration_2}
& \operatorname{supp} \mu_{2,j} \subset  \big \{ (x,y,z,\sigma) \in M\setminus \mathscr{S} \times \R^*_\sigma \, : \, \nu_j(\sigma) = 1 \big \}, \\[0.2cm]
\label{e:concentration_3}
& \operatorname{supp} m_3^{\jmath} \subset \big \{ (y,z,\eta,\varsigma) \in \mathbb{T}^2_{y,z} \times \R_\xi \times \R_\eta \times \R_\varsigma \, : \, H_\eta^{\jmath}(\varsigma,\xi) = 1 \big \}, \\[0.2cm]
\label{e:concentration_4}
&  \operatorname{supp} m_4 \subset \big \{ (x,y,z,w,\xi) \in M\setminus \mathscr{S} \times \R_w \times \R_\xi \, : G(w,\xi) = 1 \, \big \}.
\end{align}
\end{teor}
\subsubsection{Drift invariance} In this section we describe the drift  invariance properties satisfied by the two-microlocal semiclassical measures obtained in Theorem \ref{t:existence}. These invariance properties arise in a sub-principal scale with respect to the high-oscillation propagation stated in Theorem \ref{t:oscillation_invariance}. This can be interpreted in analogy with the propagation of a particle subject to the action of a strong magnetic field (see \cite{Riviere24}, \cite{Riv23}).

\begin{definition}
\label{d:reeb_vector_field}
The Reeb vector field associated with the distribution $\mathscr{D} = \operatorname{Span} \{ \mathcal{X}_1, \mathcal{X}_2 \}$ is:
$$
\mathcal{Z} := \mathcal{X}_3 + \bm{\alpha} \mathcal{X}_2, \quad  \bm{\alpha}(x) = - \frac{1}{x}, \quad x \in \R^*.
$$
Indeed, $\mathcal{Z}$ is the Reeb vector field associated with the one-form
$$
\beta := \frac{1}{2} \left(  -x \text{d}y + \frac{1}{x} \text{d} z \right).
$$
That is, $\mathscr{D} = \ker(\beta)$ and $\beta$ satisfies the conditions $\beta(\mathcal{Z}) = 1$ and $ \text{d} \beta( \mathcal{Z}, \cdot ) = 0$.
\end{definition}
In view of the high-oscillation invariance property \eqref{e:invariance_3}, for each $a \in \mathcal{C}^\infty(\R_\varsigma \times \R_\xi)$, we define:
\begin{equation}
\label{e:average_by_quartic}
\langle a \rangle_\jmath(\varsigma,\xi) :=  \lim_{T \to + \infty} \int_0^T  a \circ \phi_t^{H_\eta^{\jmath}}(\varsigma,\xi) dt.
\end{equation}
 
 The function $\langle a \rangle_\jmath$ is well defined for every $\eta \in \R$. To see this, we briefly comment  on the properties of the Hamiltonian dynamical system generated by $H_\eta^\jmath$. For $(-1)^\jmath \eta \geq 0$, the flow $\phi_t^{H_\eta^\jmath}(\varsigma,\xi)$ is periodic for each $(\varsigma,\xi) \in \R^2$ and the level set $(H_\eta^\jmath)^{-1}(E)$ consists of a unique periodic orbit for each $E \geq \eta^2$, symmetric with respect to the two axis. While if $(-1)^\jmath \eta < 0$, the flow $\phi_t^{H_\eta^\jmath}(\varsigma,\xi)$ is periodic except at the points $(\varsigma,\xi) \in \mathcal{E}_\eta^\jmath := (H_\eta^\jmath)^{-1}(\eta^2)$. In this case, for $0 \leq E < \eta^2$, the set $H_\eta^{-1}(E)$ consists of two disjoint periodic orbtis, symmetric with respect to the vertical axis $\xi = 0$, while for $E > \eta^2$, the set $(H_\eta^\jmath)^{-1}(E)$ consists of a unique periodic orbit. Finally, the set $\mathcal{E}_\eta^\jmath$ consists of two homoclinic orbits converging to $(\varsigma, \xi ) =0$. In particular, for each $(\varsigma, \xi) \in \mathcal{E}_\eta^\jmath$, 
$$
\lim_{t \to \pm \infty} \phi_t^{H_\eta^{ \jmath}}(\varsigma,\xi) = 0.
$$
Hence $\langle a \rangle_{\jmath} \vert_{\mathcal{E}^\jmath_\eta} = a(0,0)$. As a consequence of this discussion, the following holds for each $\eta \in \R$:
$$
H_\eta^{\jmath}(\varsigma_1,\xi_1) = H_\eta^{\jmath}(\varsigma_2,\xi_2) \implies \langle a \rangle_{\jmath}(\varsigma_1,\xi_1) = \langle a \rangle_{\jmath}(\varsigma_2,\xi_2).
$$
This justifies the following:
\begin{definition}
For each $\jmath \in \{0,1 \}$, and $(-1)^\jmath \eta \in (-\infty, 1]$, we set:
\begin{align}
\label{e:elliptic_integral}
\Upsilon_{\jmath}(\eta) := \langle \partial_\eta H_\eta^{\jmath} \rangle_{\jmath}(\varsigma,\xi), \quad (\varsigma,\xi) \in \big( H_\eta^{\jmath} \big)^{-1}(1).
\end{align}
\end{definition}
We next state our result concerning the drift-invariance of the two-microlocal semiclassical measures given by Theorem \ref{t:existence}.
\begin{teor}
\label{t:drift_invariance}
The measures $\mu_{1,k}$, $\mu_{2,j}$, $m_3^{\jmath}$, and $m_4$ verify respectively the following invariance properties:
\begin{align}
\label{e:drift_1}
& 0 = \partial_\eta \lambda_k(\eta,\zeta) \partial_y  \mu_{1,k}, \quad k \in \mathbb{N}, \\[0.2cm]
\label{e:drift_2}
& 0 = (2j + 1)\mathcal{Z} \mu_{2,j}, \quad j \in \mathbb{N}, \\[0.2cm]
\label{e:drift_3}
& 0 = \Upsilon_{\jmath}(\eta) \partial_y m_3^{\jmath}, \\[0.2cm]
\label{e:drift_4}
& 0 = \mathcal{Z} m_4.
\end{align}
\end{teor}
\begin{remark}
Notice that $2j+1 = \sgn(\sigma) \partial_{\sigma} \nu_j(\sigma)$ and that \eqref{e:drift_2} just means that the measure $\mu_{2,j}$ is invariant by the flow generated by $\mathcal{Z}$. 
\end{remark}

\begin{remark}
As also observed in the study of the Schrödinger propagation on the Engel group \cite[Thm. 1.2]{Benedetto24}, condition \eqref{e:drift_1} has some consequences on the dispersive properties associated with the operator $\Delta_{\mathcal{M}}$. In view of  \eqref{e:Martinet_Montgomery}, the points $(\eta,\zeta) \in \R \times \R^*_\zeta$ such that 
$$
\partial_\eta \lambda_k(\eta,\zeta) = 0
$$
correspond to those critical points $\mu_* \in \R$ of $\Lambda_k(\mu)$, via the condition
\begin{align}
\label{e:curve_critical}
\eta = \mu_*   \sgn(\zeta) \vert \zeta \vert^{1/3}.
\end{align}
These points represent an obstruction to dispersion. Since the support of $\mu_{1,k}$ is contained in the level set $$
\lambda_k(\eta,\zeta) = \vert \zeta \vert^{2/3} \Lambda_k \left( \frac{ \eta \sgn(\zeta)}{\vert \zeta \vert^{1/3}} \right) = 1,
$$ 
there are exactly two points $(\eta_{\jmath}, \zeta_{\jmath})$, with $\jmath \in \{0,1\}$, on the support of $\mu_{1,k}$ satisfying \eqref{e:curve_critical}, namely:
\begin{equation}
\label{e:critical_points}
\zeta_{\jmath} = (-1)^\jmath \left( \frac{1}{ \Lambda_k(\mu_*)} \right)^{3/2}, \quad \eta_{\jmath} =   \mu_* (-1)^{ \jmath}  \vert \zeta_{\jmath} \vert^{1/3}.
\end{equation}
\end{remark}
By \cite[Th. 1.2]{Helffer_Leautaud22}, there is a unique critical point $\mu_* \in \R$ of $\Lambda_1(\mu)$ corresponding to a non-degenerate global minimum. Moreover, by \cite[Th. 1.5]{Helffer_Leautaud22}, there exists $k_0 \in \mathbb{N}$ such that, for all $k \geq k_0$, there is a unique critical point $\mu_* \in \R$ of $\Lambda_k(\mu)$ corresponding to a non-degenerate global minimum. It is conjectured in \cite[Conj. 1.6]{Helffer_Leautaud22} that this property is true for all eigenvalues of $\widehat{H}_\mu$, with numerical evidence provided to the case $k = 2$.

\begin{remark}
The situation is similar fot the invariance property \eqref{e:drift_3}. By Lemma \ref{l:averaging_elliptic} and Remark \ref{r:critical_point}, there exists a unique $\eta_* \in (-1,1)$ such that
\begin{equation}
\label{e:eta_critical}
\Upsilon_{\jmath}((-1)^{\jmath} \eta_*) = 0.
\end{equation}
\end{remark}

\subsubsection{Weak drift invariance} 

In this section we state some additional weak-dispersive properties of the measure $M_1$ at the critical points $(\eta_{\jmath}, \zeta_{\jmath})$ given by \eqref{e:critical_points}, assuming that these points are non-degenerate critical points of $\partial_\eta \lambda_{\mathbf{k}}$ for a given $\mathbf{k}$. However, we are not able to say anything new for the measure $m_3^{\jmath}$ near the points $(-1)^{\jmath} \eta_*$ given by \eqref{e:eta_critical}. The idea is to perform a further two-microlocalization near $(\eta,\zeta) = (\eta_\jmath,\zeta_\jmath)$ to obtain another invariance law. This idea was originally introduced by Macià \cite{Macia10} (see also \cite{AnantharamanMacia14}) in the study of quantum limits of the Laplacian on the torus, and has also been considered in the study of obstruction to dispersion in the Engel group (see \cite[Th. 1.14]{Benedetto24}).
\medskip

Let us first consider $\mathbf{k} \in \mathbb{N}$ and assume that $\mu_* \in \R$ is a non-degenerate critical point of $\Lambda_{\mathbf{k}}(\mu)$. Then, for each $\jmath \in \{0,1 \}$, the point $(\eta_\jmath , \zeta_\jmath)$ given by \eqref{e:critical_points} satisfies
$$
\lambda_{\mathbf{k}}(\eta_{\jmath},\zeta_\jmath) = 1, \quad \partial_\eta \lambda_{\mathbf{k}} (\eta_\jmath,\zeta_\jmath) = 0.
$$
Then the curve $\lambda_{\mathbf{k}}(\eta,\zeta) = 1$ is tangent to the vertical axis $\zeta = 0$ at the point $(\eta_\jmath, \zeta_\jmath)$. The next theorem gives the decomposition of the measure $\mu_{1,\mathbf{k}}$ in sum of two-microlocal semiclassical measures.

\begin{teor}
\label{t:more_two_microlocal}
Let $\mathbf{k} \in \mathbb{N}$  and suppose that $\mu_* \in \R$ is a non-degenerate critical point of $\Lambda_{\mathbf{k}}(\mu)$. For each $\jmath \in \{0,1 \}$, let $(\eta_\jmath, \zeta_\jmath)$ be the critical point given by \eqref{e:critical_points}. Fix an open neighborhood $ \mathcal{B}_\jmath \subset \R_\eta \times \R^*_\zeta$ containing the point $(\eta_\jmath, \zeta_\jmath)$ such that
$$
(\eta, \zeta) \in \mathcal{B}_\jmath \implies \lambda_k(\eta,\zeta) \neq 1, \quad \forall k \neq \mathbf{k}.
$$ 
Then, modulo the extraction of a further subsequence of the sequence $(\psi_h)$ solving \eqref{e:semiclassical_eigenvalue_problem}, there exist positive Radon measures $\mathbf{M}^\jmath_{1,\mathbf{k}} \in \mathcal{M}_+(\T^1_z \times \R^*_\zeta ; \mathcal{L}^1(L^2(\T_y)))$ and $\bm{\nu}^\jmath_{1,\mathbf{k}} \in \mathcal{M}_+(\T^2_{y,z} \times \mathbb{S}^0_\theta)$ such that, for each symbol
$$
a \in \mathcal{C}^\infty( \T^2_{y,z} \times \mathcal{B}_\jmath \times \R_v)
$$
being compactly supported in the variables $(y,z,\eta,\zeta)$, and  homogeneous of degree zero at infinity\footnote{That is, there exists $a_\infty \in \mathcal{C}^\infty(\T^2_{y,z} \times \mathcal{B}_\jmath \times \mathbb{S}_\theta^0)$ such that, there exists $R > 0$ such that
$$
a(y,z,\eta,\zeta, v) = a_\infty \left( y,z,\eta,\zeta, \frac{v}{\vert v \vert} \right), \quad \text{if } \vert v \vert \geq R.
$$} in the variable $v$, defining 
\begin{equation}
\label{e:symbol_for_new_two_microlocal}
\mathbf{a}^\jmath_{h,R,\epsilon}(y,z,\eta,\zeta) := \breve{\chi} \left( \frac{h^2\zeta}{\epsilon} \right) \breve{\chi}\left( \frac{\zeta}{R} \right) a\left(y,z,\eta,h^2 \zeta, \frac{\eta - \eta_\jmath}{h} \right),
\end{equation}
one has:
\begin{align*}
\lim_{h \to 0} \big \langle \Op_h^\w( \mathbf{a}^\jmath_{h,R,\epsilon}) \psi_h, \psi_h \big \rangle_{L^2(M)} & \\[0.2cm]
 & \hspace*{-3cm} =  \int_{\T^2_{y,z} \times \R_\eta \times \R^*_\zeta} \mathbf{1}_{\eta \neq \eta_\jmath} a_\infty \left( y,z,\eta,\zeta, \frac{\eta - \eta_\jmath}{\vert \eta - \eta_\jmath \vert} \right) d\mu_{1,\mathbf{k}} \\[0.2cm]
 & \hspace*{-3cm} \quad + \int_{\T^2_{y,z} \times \mathbb{S}^0_\theta} a_\infty\left( y,z,\eta_\jmath, \zeta_\jmath, \theta \right) d\bm{\nu}_{1,\mathbf{k}}^{\jmath} \\[0.2cm]
 & \hspace*{-3cm} \quad +  \int_{\T^1_z } \operatorname{Tr} \Big( \Op_{(y,\eta)}^{\T} \left( a\left(y,z,\eta_\jmath, \zeta_\jmath, \eta \right) \right) d \mathbf{M}_{1,\mathbf{k}}^{\jmath} \Big).
\end{align*}
In particular, by testing against symbols that do not depend on $v$, we get by projection:
$$
\mu_{1,\mathbf{k}} =\sum_{\jmath \in \{0,1\}} \left( \mathbf{1}_{\eta \neq \eta_\jmath} \mu_{1,\mathbf{k}} + \bm{\nu}^\jmath_{1,\mathbf{k}} + \operatorname{Tr} \mathbf{M}^\jmath_{1,\mathbf{k}} \right).
$$
\end{teor}
We finally state the weak-invariance properties satisfied by the measures $\bm{\nu}_{1,\mathbf{k}}^\jmath$ and $\mathbf{M}_{1,\mathbf{k}}^\jmath$.
\begin{teor}
\label{t:weak_drift_invariance_1}
Let $\mathbf{k} \in \mathbb{N}$  and suppose that $\mu_* \in \R$ is a non-degenerate critical point of $\Lambda_{\mathbf{k}}(\mu)$. Let $ (\eta_\jmath, \zeta_\jmath)$ be given by \eqref{e:critical_points}. Then the measures $\bm{\nu}_{1,\mathbf{k}}^\jmath$ and $\mathbf{M}^\jmath_{1,\mathbf{k}}$ given by Theorem \ref{t:more_two_microlocal} satisfy the invariance properties:
\begin{align}
\label{e:weak_1}
0 & = \partial_y \bm{\nu}_{1,\mathbf{k}}^\jmath, \\[0.2cm]
\label{e:weak_2}
0 & =  [D_y^2, \mathbf{M}_{1,\mathbf{k}}^\jmath].
\end{align}
In particular, the measure $\mu_{1,\mathbf{k}}$ is constant in the variable $y$.
\end{teor}

\subsection*{Organization of the article} We show Theorem \ref{t:existence} in Section \ref{s:existence}, which is the most technical part of the work devoted to the construction of two-microlocal measures. Next, Theorem \ref{t:oscillation_invariance} is shown in Section  \ref{s:high_oscillation}. Section \ref{s:concentration} is devoted to prove Theorem \ref{t:concentration}, and Section \ref{s:Drift} to prove Theorem \ref{t:drift_invariance}. Finally, Theorems \ref{t:more_two_microlocal} and \ref{t:weak_drift_invariance_1} are proven in Section \ref{s:weak_drift}. In Appendix \ref{a:spectral_properties} we review some spectral properties of the operator $\Delta_{\mathcal{M}}$, and in Appendix \ref{a:appendix} we give some basic results on semiclassical analysis.

\subsection*{Acknowledgements} The author warmly thanks Gabriel Rivière for many conversations on the study of sub-Laplacians. This work is dedicated to Professor Didier Robert on his 80th birthday, celebrated in Nantes on May 2025. The author is supported by ANR projects SpecDiMa and STENTOR.


\section{Two microlocal semiclassical measures}
\label{s:existence}
This section is devoted to prove Theorem \ref{t:existence}. Before starting, we review some rudiments of semiclassical analysis and provide some adapted Calderón-Vaillancourt type lemmas for a family of semiclassical pseudodifferential operators including the class of symbols $\mathcal{S}_{\mathcal{M}}$ when these operators are restricted to the space of solutions to \eqref{e:semiclassical_eigenvalue_problem}. We then proceed to construct  the two-microlocal semiclassical measures which captures the energy of the sequence $(\psi_h)$ solving \eqref{e:semiclassical_eigenvalue_problem} in the different regimes associated to the degenerate directions $\mathcal{X}_3$ and $\mathcal{X}_4$.
\medskip

We denote by $(X,\Xi,X') = (x,y,z,\xi,\eta,\zeta,x',y',z') \in \R^3_X \times \R^3_{\Xi} \times \R^3_{X'}$
the phase-space variables and consider a family of symbols $\mathbf{a} \in \mathcal{C}^\infty( \R^3_{X} \times \R^3_\Xi \times \R^3_{X'})$ being $2\pi$-periodic in the variables $(y,z,y',z')$, and slowly growing in all variables, that is:
\begin{equation}
\label{e:slowly_growing}
\forall \alpha \in \mathbb{N}^3, \; \exists C_\alpha > 0, \; \exists N_\alpha \in \mathbb{Z}, \quad \underset{(X,\Xi,X') \in \R^9}{\sup} \big(1 + \vert \Xi \vert + \vert X \vert + \vert X' \vert \big)^{-N_\alpha}  \big \vert \partial^\alpha  \mathbf{a} \big \vert  \leq C_\alpha .
\end{equation}
For such symbols, we define the semiclassical quantization $\Op_h^{M}(\mathbf{a})$ acting on $\psi \in \mathcal{\mathscr{S}}(M)$ by:
\begin{align*}
\Op_h^{M}(\mathbf{a})\psi(X) & :=  \int_{\R^3 \times \R^3} \mathbf{a}(X,h\Xi,X') \psi(X') e^{i \Xi \cdot(X - X' )} \text{d}(X',\Xi)
\end{align*}
with respect to the measure
\begin{equation}
\label{e:measure}
\text{d}(X',\Xi) :=  dx' \otimes \frac{\mathbf{1}_{\T^2}(y',z')}{(2\pi)^2} dy'dz'  \otimes d\xi \otimes \left(\sum_{k \in \mathbb{Z}^2} \delta\big( (\eta,\zeta) - k \big) \right),
\end{equation}
where $\mathbf{1}_{\T^2}$ denotes the indicator function of any fundamental domain of $\T^2$ as a subset of $\R^2$. The operator $\Op_h^{M}(\mathbf{a})$ defined in this way has to be interpreted in distributional sense as map from $\mathscr{S}(M)$ to $\mathscr{S}'(M)$ (see for instance \cite[Thm. 4.2]{Zworski12}).
\begin{example}
If $\mathbf{a}(X,\Xi,X') = a\left( \frac{X + X'}{2},\Xi \right)$ then we recover the semiclassical Weyl quantization on $M$, and we denote in this case $\Op_h^M(\mathbf{a}) = \Op_h^{\w}(a)$.
\end{example}
We  next define the measure:
\begin{equation}
\label{e:measure_2}
\text{d}(X',\Xi,X) := \text{d}(X',\Xi) \otimes  dx \otimes \frac{\mathbf{1}_{\T^2}(y,z)}{(2\pi)^2} dy \, dz,
\end{equation}
so that we can write, for $(\psi_h)$ solving \eqref{e:semiclassical_eigenvalue_problem}, the Wigner-type distribution
\begin{align}
\label{e:Wigner_distribution}
\big \langle \Op^{M}_h(\mathbf{a})\psi_h , \psi_h \big \rangle_{L^2(M)} =  \int_{\R^3 \times \R^3 \times \R^3} \mathbf{a}(X,h\Xi,X') \psi_h(X') \overline{\psi}_h(X) e^{i \Xi \cdot(X - X' )} \text{d}(X',\Xi,X). 
\end{align}
Our aim is to study the Wigner distribution \eqref{e:Wigner_distribution} with $\mathbf{a} = \mathbf{a}_{h,R}$ given by \eqref{e:R_symbol}, and we will estimate it using Calderón-Vaillancourt theorem (Lemma \ref{l:Calderon_Vaillancourt}) and Rothschild-Stein estimates \eqref{e:R-S_1}, \eqref{e:R-S_2}, and \eqref{e:R-S_3}.  To do this, let us introduce a change of coordinates balancing the position and momentum variables. Set:
$$
\bm{\kappa}_h :   M \times \R^3 \times M   \rightarrow  \R_x \times h^{-1/2} \mathbb{T}^2_{(y,z)} \times \R^3_{(\xi,\eta,\zeta)} \times \R_{x'} \times h^{-1/2} \mathbb{T}^2_{(y',z')} 
$$
given by dilation-contraction map
\begin{align}
\label{e:symplectic_change_of_coordinates}
\bm{\kappa}_h(X,\Xi,X') &  = \Big( hx,h^{1/2}y, h^{1/2}z, h^{-1}\xi,h^{-1/2} \eta, h^{-1/2} \zeta, hx', h^{1/2}y', h^{1/2} z' \Big).
\end{align}
Define also the unitary transformation quantizing the diffeomorphism $\bm{\kappa}_h$ (which we restrict to $\psi_h$ for simplicity). Set:
\begin{equation}
\label{e:h_hilbert_spaces}
 \psi_h \mapsto \Psi_h \in L^2(\R_x \times h^{-1/2} \T^2_{y,z})
 \end{equation}
 given by the formula
\begin{equation}
\label{e:unitary_transformation}
\Psi_h(x,y,z) := h \psi_h \big( hx,h^{1/2}y,h^{1/2}z \big).
\end{equation}
Then, since $\bm{\kappa}_h$ is a linear symplectic transformation, by using the exact version of Egorov's theorem,  we can rescale the Wigner distribution \eqref{e:Wigner_distribution} as:
$$
\big \langle \Op_h^{M}(\mathbf{a})\psi_h , \psi_h \big \rangle_{L^2(M)} =  \int_{\R^3 \times \R^3 \times \R^3} \mathbf{a} \circ \bm{\kappa}_h(X,h\Xi,X') \Psi_h(X') \overline{\Psi}_h(X) e^{i \Xi \cdot(X - X' )} \text{d}_h(X',\Xi,X),
$$
where the measure $\text{d}_h$ is now given by
\begin{align}
\label{e:change_measure_iota}
\text{d}_h(X',\Xi,X) = \big( \bm{\kappa}_h \big)_* \text{d}(X',\Xi,X).
\end{align}

\begin{lemma}
\label{l:control_L^2_norm}
Let $a \in \mathcal{C}^\infty(M_{x,y,z} \times \R^4_{w,\xi, \sigma,\zeta} \times M_{x',y',z'})$ be bounded together with all its derivatives.  Define, for $R> 0$ and $h > 0$:
$$
\mathbf{a}_{h,R}(X,\Xi,X') = \breve{\chi}\left( \frac{\zeta}{R} \right) a(x,y,z,w,\xi,h \sigma, h^2\zeta,x',y',z'),
$$ 
where $\sigma = \sigma(x,\zeta) = 2x \zeta$ and $w = w(x,\eta,\zeta) = \eta + x^2 \zeta$. Then, for any sequence $(\psi_h)$ solving \eqref{e:semiclassical_eigenvalue_problem},
\begin{equation}
\label{e:control_in_h_psi}
\Vert \Op^{M}_h(\mathbf{a}_{h,R}) \psi_h \Vert_{L^2(M)} \lesssim_{a,R} 1, \quad \text{as } h \to 0.
\end{equation}
\end{lemma}

\begin{remark}
Notice that the symbol $\mathbf{a}_{h,R}$ is slowly growing in the sense of \eqref{e:slowly_growing}, but it does not lie in any symbol class $S(m)$ with $m$ an order function in the sense of \cite[\S 4.4]{Zworski12}. However, Lemma \ref{l:control_L^2_norm} shows that the operator $\Op^{M}_h(\mathbf{a}_{h,R})$  is uniformly bounded in the $L^2$-norm when restricted to the space of solutions to \eqref{e:semiclassical_eigenvalue_problem}.
\end{remark}

\begin{remark}
The constant in estimate \eqref{e:control_in_h_psi}, although not written explicitely, depends only on a finite sum of $L^\infty$ norms of derivatives of $a$ and dimensional factors.
\end{remark}

\begin{proof}
We will make use of estimates \eqref{e:R-S_2} and \eqref{e:R-S_3} together with Calderón-Vaillancourt theorem (see precisely Lemma \ref{l:Calderon_Vaillancourt}). 
Let us define:
$$
\Sigma_h := (h\sigma,h^2\zeta), \quad  \vert \Sigma_h \vert = \sqrt{1 +  (h\sigma)^2  +  (h^2 \zeta)^2}.
$$
By using the semiclassical symbolic calculus, we can expand:
\begin{align*}
\Op^M_h \big( \mathbf{a}_{h,R} \big) & = \Op^M_h \big( \mathbf{b}_{h,R} \big) \Op^M_h\big( \vert \Sigma_h \vert \big) + \Op^M_h \big( \mathbf{b}_{h,R} \cdot \mathcal{E}_h \big), 
\end{align*}
where 
$$
\mathbf{b}_{h,R} = \frac{ \mathbf{a}_{h,R}}{ \vert \Sigma_h \vert}, \quad \mathcal{E}_h(x,x',\zeta)  = \vert \Sigma_h(x,\zeta) \vert - \vert \Sigma_h(x',\zeta) \vert.
$$
By Taylor's theorem,
\begin{align*}
\mathcal{E}_h(x,x',\zeta) & = \frac{ \vert \Sigma_h(x,\zeta) \vert^2 - \vert \Sigma_h(x',\zeta) \vert^2}{ \vert \Sigma_h(x,\zeta)\vert  + \vert \Sigma_h(x',\zeta) \vert} \\[0.2cm]
 & =  \frac{ (x-x') \mathcal{R}_1(x,x') (h\zeta) (h\sigma(x',\zeta))}{  \vert \Sigma_h(x,\zeta)\vert  + \vert \Sigma_h(x',\zeta) \vert} + \frac{(x-x')^2 \mathcal{R}_2(x,x') (h\zeta)^2}{  \vert \Sigma_h(x,\zeta)\vert  + \vert \Sigma_h(x',\zeta) \vert}, 
\end{align*}
for some $\mathcal{R}_1,\mathcal{R}_2 \in \mathcal{C}^\infty(\T_x \times \T_{x'})$.
Now, by Lemma \ref{l:Calderon_Vaillancourt} we have:
$$
\left \Vert \Op_h \big( \mathbf{b}_{h,R} \big)  \right \Vert_{\mathcal{L}(L^2(M))} \lesssim_{a,R} 1,
$$
where notice that, in Lemma \ref{l:Calderon_Vaillancourt}, only derivatives of order one are taken with respect to the variables $x$ and $x'$. Moreover, we observe that 
$$
\Op_h^{M}(h\sigma) = -i h^2 \mathcal{X}_3,  \quad\quad \Op^{M}_h(h^2 \zeta) = - i h^3 \mathcal{X}_4.
$$
Hence, since no dual variables appear at same time in the symbols $\sigma$ and $\zeta$, we have the exact quantization formula
$$
\Op^{M}_h(\vert \Sigma_h \vert) = \sqrt{  1 + (-ih^2 \mathcal{X}_3)^2 + (- i h^3 \mathcal{X}_4)^2 }.
$$
Then, by estimates \eqref{e:R-S_2} and \eqref{e:R-S_3}, we obtain:
$$
\left \Vert  \Op^{M}_h(\vert \Sigma_h \vert) \psi_h \right \Vert_{L^2(M)} \lesssim 1.
$$
 On the other hand, defining the error terms
 \begin{align*}
 \mathbf{R}_1(x,x',\zeta) & :=  \frac{  \mathcal{R}_1(x,x')(h^2\zeta) (h\sigma(x',\zeta))}{  \vert \Sigma_h(x,\zeta)\vert  + \vert \Sigma_h(x',\zeta) \vert}, \\[0.2cm]
  \mathbf{R}_2(x,x',\zeta) & := \frac{\mathcal{R}_2(x,x') (h^2\zeta)^2}{  \vert \Sigma_h(x,\zeta)\vert  + \vert \Sigma_h(x',\zeta) \vert},
 \end{align*}
observing that, for each $(x,\xi,x') \in \R^3$, we have
 $$
D_\xi e^{i \xi (x-x') } =  (x-x') \cdot e^{i \xi (x-x') },
 $$  
and using integration by parts we can write, for any $\varphi \in L^2(M)$:
 \begin{align*}
 \Big \langle \Op^{M}_h \big( \mathbf{b}_{h,R} \cdot \mathcal{E}_h \big)  \psi_h, \varphi \Big \rangle_{L^2(M)} & \\[0.2cm]
  & \hspace*{-3.5cm} = \Big \langle \Op^{M}_h \big(-D_\xi \mathbf{b}_{h,R} \cdot \mathbf{R}_1 \big) \, \psi_h, \varphi \Big \rangle_{L^2(M)} +  \Big \langle \Op^{M}_h \big( D_\xi^2\mathbf{b}_{h,R} \cdot \mathbf{R}_2 \big) \, \psi_h, \varphi \Big \rangle_{L^2(M)}.
 \end{align*}
Finally, using estimates \eqref{e:R-S_2}, \eqref{e:R-S_3} and Lemma \ref{l:Calderon_Vaillancourt}, we get:
$$
\left \vert \Big \langle \Op^{M}_h \big( \mathbf{b}_{h,R} \cdot \mathcal{E}_h \big)  \psi_h, \varphi \Big \rangle_{L^2(M)} \right \vert \lesssim_{a,R} \Vert \psi_h \Vert_{L^2(M)} \Vert \varphi \Vert_{L^2(M)},
$$
and thus the claim.
\end{proof}

\subsection{$h$-oscillation properties}

We next state some $h$-oscillation properties of the sequence $(\psi_h)$ of solutions to \eqref{e:semiclassical_eigenvalue_problem}. By $h$-oscillation we mean that the energy of the sequence concentrates asymptotically on some regions of the phase-space. Since $(\psi_h)$ solves \eqref{e:semiclassical_eigenvalue_problem}, we first observe that the sequence $(\psi_h)$ concentrates on the region $H_{\mathcal{M}} \lesssim 1$, and we will also show $h$-oscillation in the regions $h \sigma \lesssim 1$ and $h^2 \zeta \lesssim 1$, which are consequences respectively of Rothschild-Stein estimates \eqref{e:R-S_2} and \eqref{e:R-S_3}. From now on, we will restrict ourselves to the Weyl quantization, and we will use the notation $\Op_h^{\w}$ for the Weyl quantization on $M$.

\begin{definition}
Let $\chi$ be the cut-off function fixed at the beginning of Secion \ref{s:main_results}. Let $(\varphi_h) \subset L^2(M)$ be a normalized sequence. We say that $(\varphi_h)$ is $(h^2\Delta_{\mathcal{M}})$-oscillating if:
$$
\lim_{R \to + \infty} \lim_{h \to 0} \left \Vert \chi\left( \frac{h^2 \Delta_{\mathcal{M}}}{R} \right) \varphi_h \right \Vert_{L^2(M)} = 1.
$$
\end{definition}

\begin{prop}
\label{p:h_oscillation}
Any sequence $(\psi_h)$ satisfying \eqref{e:semiclassical_eigenvalue_problem} is $(h^2\Delta_{\mathcal{M}})$-oscillating. Moreover,
\begin{align}
\label{e:h_oscillation_1}
0 & = \lim_{R' \to + \infty} \lim_{R \to + \infty} \lim_{h \to 0} \left \Vert  \Op^{\w}_h \left( \breve{\chi}\left( \frac{\xi^2 + (\eta + x^2 \zeta)^2}{R} \right) \right) \breve{\chi}\left( \frac{hD_z}{R'} \right) \psi_h \right \Vert_{L^2(M)}, \\[0.2cm]
\label{e:h_oscillation_2}
0 & = \lim_{R' \to + \infty} \lim_{R \to + \infty} \lim_{h \to 0} \left \Vert \Op^{\w}_h \left( \breve{\chi} \left( \frac{ h \sigma(x,\zeta)}{R} \right) \right) \breve{\chi}\left( \frac{hD_z}{R'} \right) \psi_h  \right \Vert_{L^2(M)}, \\[0.2cm]
\label{e:h_oscillation_3}
0 & =  \lim_{R \to + \infty} \lim_{h \to 0} \left \Vert  \breve{\chi}\left( \frac{h^3D_z}{R} \right) \psi_h \right \Vert_{L^2(M)}.
\end{align}
\end{prop}

\begin{proof}
The fact that the sequence $(\psi_h)$ is $(h^2\Delta_{\mathcal{M}})$-oscillating is evident by using \eqref{e:semiclassical_eigenvalue_problem}. To show \eqref{e:h_oscillation_1}, let $H_{\mathcal{M}}(X,\Xi) = \xi^2 + (\eta + x^2 \zeta)^2$ be the symbol of $h^2 \Delta_{\mathcal{M}}$. Denoting $\chi_R := \chi(\cdot/R)$, we observe that:
\begin{align*}
\Op^{\w}_h \big( 1 - \chi_R(H_{\mathcal{M}}) \big)  & = \Op^{\w}_h \left( \frac{1 - \chi_R(H_{\mathcal{M}})}{H_\mathcal{M}} \cdot H_\mathcal{M} \right)  &  \\[0.2cm]
 &  = \Op^{\w}_h \left( \frac{1 - \chi_R(H_\mathcal{M})}{H_\mathcal{M}} \right) (h^2 \Delta_{\mathcal{M}}) +  \Op^{M}_h \left( \frac{1 - \chi_R(\mathbf{H}_\mathcal{M})}{\mathbf{H}_\mathcal{M}} \cdot \mathcal{E}_{\mathcal{M}} \right)  \\[0.2cm]
 & =  \Op^{M}_h \left( \frac{1 - \chi_R(\mathbf{H}_\mathcal{M})}{\mathbf{H}_\mathcal{M}}\cdot  (1 + \mathcal{E}_{\mathcal{M}} ) \right) ,
\end{align*}
where
\begin{align*}
\mathbf{H}_{\mathcal{M}}(x,x',\xi,\eta,\zeta) & = H_{\mathcal{M}} \left( \frac{x+x'}{2}, \xi,\eta,\zeta \right), \\[0.2cm]
\mathcal{E}_{\mathcal{M}}(x,x',\eta,\zeta) & = H_{\mathcal{M}}\left( \frac{x+x'}{2},\xi,\eta,\zeta \right) - H_{\mathcal{M}}(x',\xi,\eta,\zeta). 
\end{align*}
Using Taylor's theorem, we can write the term $\mathcal{E}_{\mathcal{M}}$ as:
\begin{align*}
 \mathcal{E}_{\mathcal{M}}(x,x',\eta,\zeta) & = \mathcal{P}_{\mathcal{M}}\big( w(x',\eta,\zeta) , (x-x') \sigma(x',\zeta), (x-x')^2 \zeta \big),
 \end{align*}
 for a certain polynomial $\mathcal{P}_{\mathcal{M}} \in \R[w,\sigma,\zeta]$. Next, by Lemma \ref{l:control_L^2_norm}, we have:
 $$
 \left \Vert  \Op^{M}_h \left( \frac{1 - \chi_R(\mathbf{H}_\mathcal{M})}{\mathbf{H}_\mathcal{M}} \right) \breve{\chi}\left( \frac{hD_z}{R'} \right) \psi_h \right \Vert_{L^2(M)} = \mathcal{O}\left( \frac{1}{R} \right) + \mathcal{O}\left( \frac{1}{R'} \right).
 $$
Moreover, defining:
\begin{align*}
\mathcal{R}_h(x',\eta,\zeta) & := \mathcal{P}_{\mathcal{M}}\big(w(x',\eta,\zeta), h\sigma(x',\zeta), h^2\zeta \big) \\[0.2cm]
\mathbf{b}_{h,R}(x,x',\xi,\eta,\zeta) & := \mathcal{P}_{\mathcal{M}}(1, -D_\xi, D_\xi^2) \frac{1 - \chi_R(\mathbf{H}_{\mathcal{M}})}{\mathbf{H}_{\mathcal{M}}},
\end{align*}
 and using integration by parts, we get, for any $\varphi \in L^2(M)$,
\begin{align*}
\Big \langle \Op^{M}_h \left( \frac{1 - \chi_R(\mathbf{H}_\mathcal{M})}{\mathbf{H}_\mathcal{M}}\cdot   \mathcal{E}_{\mathcal{M}} \right)\breve{\chi}\left( \frac{hD_z}{R'} \right) \psi_h, \varphi \Big \rangle_{L^2(M)} &  \\[0.2cm]
 & \hspace*{-5cm} = \Big \langle \Op^{M}_h \big( \mathbf{b}_{h,R} \cdot \mathcal{R}_h \big) \breve{\chi}\left( \frac{hD_z}{R'} \right) \psi_h, \varphi \Big \rangle_{L^2(M)}.
\end{align*}
Finallly, using estimates \eqref{e:R-S_1}, \eqref{e:R-S_2} and \eqref{e:R-S_3}, and Lemma  \ref{l:Calderon_Vaillancourt}, we get:
\begin{align*}
\left \Vert  \Op^{M}_h \left( \frac{1 - \chi_R(\mathbf{H}_\mathcal{M})}{\mathbf{H}_\mathcal{M}}\cdot   \mathcal{E}_{\mathcal{M}} \right) \breve{\chi}\left( \frac{hD_z}{R'} \right) \psi_h \right \Vert_{L^2(M)} = \mathcal{O}\left( \frac{1}{R} \right) + \mathcal{O}\left( \frac{1}{R'} \right),
\end{align*}
then the claim holds true. 

To conclude, we observe that properties \eqref{e:h_oscillation_2} and \eqref{e:h_oscillation_3} follow in similar way, by means of semiclassical pseudodifferential calculus, Lemma \ref{l:control_L^2_norm}, and estimates \eqref{e:R-S_2} and \eqref{e:R-S_3}.
\end{proof}

\subsection{Second microlocalization}

We now proceed to show Theorem \ref{t:existence}, showing the existence of two-microlocal semiclassical measures adapted to the sub-elliptic regimes issued from $h^2 \Delta_{\mathcal{M}}$. Let us fix parameters $R, \epsilon, \delta, \rho  > 0$, and split the symbol $\mathbf{a}_{h,R}$ given by \eqref{e:R_symbol} into the sum of four terms $\mathbf{a}_{h,R} = \mathbf{a}_{h,R}^{\epsilon} +  \mathbf{a}_{h,R,\epsilon}^{\delta,\rho} + \mathbf{a}_{h,R,\epsilon,\delta}^\rho + \mathbf{a}_{h,R,\epsilon,\rho}$, with:
\begin{align*}
\mathbf{a}_{h,R}^\epsilon & := \breve{ \chi}\left( \frac{h^2\zeta}{\epsilon} \right) \mathbf{a}_{h,R}, \\[0.2cm]
\mathbf{a}_{h,R,\epsilon}^{\delta,\rho} & := \breve{\chi} \left( \frac{\eta }{\rho}  \right)  \breve{\chi}\left( \frac{h\sigma}{\delta} \right)   \chi \left( \frac{h^2 \zeta}{\epsilon}  \right)   \, \mathbf{a}_{h,R}, \\[0.2cm]
\mathbf{a}_{h,R,\epsilon,\rho} & := \chi \left( \frac{\eta }{\rho}  \right) \chi \left(\frac{h^2 \zeta}{\epsilon} \right) \mathbf{a}_{h,R}, \\[0.2cm]
\mathbf{a}_{h,R,\epsilon,\delta}^\rho & :=\breve{\chi} \left( \frac{\eta }{\rho}  \right) \chi\left( \frac{h\sigma}{\delta} \right) \chi \left( \frac{h^2 \zeta}{\epsilon}  \right) \mathbf{a}_{h,R}.
\end{align*}
Here, $\sigma = \sigma(x,\zeta) = 2x \zeta$. Let us define the distributions:
\begin{align*}
I_{h,R,\epsilon}^1(a) & := \big \langle \Op^{\w}_h( \mathbf{a}_{h,R}^\epsilon) \psi_h, \psi_h \big \rangle_{L^2(M)}, \\[0.2cm] 
I_{h,R,\epsilon,\delta,\rho}^2(a) & := \big \langle \Op^{\w}_h( \mathbf{a}_{h,R,\epsilon}^{\delta,\rho}) \psi_h, \psi_h \big \rangle_{L^2(M)}, \\[0.2cm]
I_{h,R,\epsilon,\rho}^3(a) & := \big \langle \Op_h^{\w}( \mathbf{a}_{h,R,\epsilon,\rho}) \psi_h, \psi_h \big \rangle_{L^2(M)}, \\[0.2cm]
I_{h,R,\epsilon,\delta,\rho}^4(a) & := \big \langle \Op_h^{\w}( \mathbf{a}_{h,R,\epsilon,\delta}^\rho) \psi_h, \psi_h \big \rangle_{L^2(M)}.
\end{align*}
We will show that each of these distributions produce, in the semiclassical limit, the measures $M_1$, $M_2$, $m_3^\jmath$, and $m_4$. To facilitate the reading, we give the existence of each of these measures in a separate proposition.

\begin{prop}
\label{l:measure_m_1}
There exists a positive Radon measure 
$$
M_1 \in \mathcal{M}_+(\mathbb{T}_y \times \T_z  \times \R_\eta \times \R^*_\zeta;\mathcal{L}^1(L^2(\R_x)))$$ 
such that, modulo the extraction of suitable subsequences,
\begin{align*}
 \lim_{R \to + \infty} \lim_{\epsilon \to 0} \lim_{h \to 0} I_{h,R,\epsilon}^1(a) =  \int_{\mathbb{T}_y \times \T_z  \times \R_\eta \times \R^*_\zeta} \operatorname{Tr}_{L^2(\R_x)} \Big(  \mathcal{A}(y,z,\eta,\zeta) d M_1(y,z,\eta,\zeta) \Big),
\end{align*}
where\footnote{If not mentioned explicitely, we work with the Weyl quantization.}
\begin{align}
\label{e:A_1}
\mathcal{A}(y,z,\eta,\zeta) & = \operatorname{Op}^\R_{(x,\xi)} \big( a\left(0, y,z, w(x,\eta,\zeta), \xi,\sigma(x,\zeta),\zeta \right) \big),
\end{align}
and the functions $w(x,\eta,\zeta)$ and $\sigma(x,\zeta)$ are given by \eqref{e:w_and_sigma}.
\end{prop}

\begin{proof}
We first consider the change of variables of \eqref{e:symplectic_change_of_coordinates} and \eqref{e:unitary_transformation}. In terms of the push-forward measure $
\text{d}_h(X',\Xi,X) := \big( \bm{\kappa}_h \big)_* \text{d}(X',\Xi,X)$, 
we can rewrite the distribution $I_{h,R,\epsilon}^1(a)$  as:
\begin{align*}
I_{h,R,\epsilon}^1(a)  &   =   \int_{\R^9} \mathbf{a}_{h,R}^\epsilon \left( \frac{X+X'}{2},h\Xi,\right)  \psi_h(X') \overline{\psi}_h(X) e^{i \Xi \cdot( X - X')}\text{d}(X',\Xi,X)  \notag \\[0.2cm]
   & = \int_{\R^9}  \mathbf{A}_{h,R}^{\epsilon}\left( \frac{X+X'}{2},h\Xi \right) \Psi_h(X') \overline{\Psi}_h(X) e^{i \Xi \cdot (X - X')}\text{d}_h(X',\Xi,X),
 \end{align*}
 where $ \mathbf{A}_{h,R}^{\epsilon} := \mathbf{a}_{h,R}^{\epsilon} \circ \bm{\kappa}_h$. Defining next the operator-valued symbol
 $$
 \mathcal{A}_{h,R}^{\epsilon}(Y,\Theta) :=\Op_{(x,\xi)}^{\R} \big( \mathbf{A}_{h,R}^{\epsilon} ( \cdot_x, Y, \cdot_{h\xi}, h \Theta ) \big),
 $$
 where we denote $Y = (y,z)$ and $\Theta = (\eta,\zeta)$, we can also write the distribution $I_{h,R,\epsilon}^1(a)$ in trace form as
 \begin{align}
 \label{e:effective_integral}
 I_{h,R,\epsilon}^1(a) = \int_{\R^6} \operatorname{Tr}_{L^2(\R_x)} \left[    \mathcal{A}_{h,R}^{\epsilon}\left( \frac{Y+Y'}{2},\Theta\right) \mathcal{K}_h(Y',Y) \right] e^{i \Theta \cdot (Y - Y')}  \text{d}_h(Y',\Theta,Y),
 \end{align}
 where $\mathcal{K}_h(Y',Y)$ is the trace class operator, defined on $L^2(\R_x)$, by the tensor-product integral kernel
 $$
 \Psi_h \otimes \Psi_h(x',x,Y',Y) := \Psi_h(x',Y') \overline{\Psi}_h(x,Y) .
 $$
Notice moreover that the localization of the cut-off
$$
 \breve{\chi}\left( \frac{h^{2+1/2} \zeta}{\epsilon} \right),
$$ 
together with the condition that $a$ has compact support in the variable $h\sigma$, give that 
\begin{equation}
\label{e:localization_in_x}
 \vert x  \vert \lesssim_{a} \epsilon^{-1}
\end{equation}
on the support of integration of \eqref{e:effective_integral}. Moreover, for $R > 0$ and $\epsilon > 0$ fixed, and $h > 0$ sufficiently small, the effect of the cut-off $\breve{\chi}(h^{1/2} \zeta/R)$ is empty. 
Then, by the semiclassical pseudodifferential calculus, we have:
\begin{align*}
 I_{h,R,\epsilon}^1(a) & =  \int_{\R^6} \operatorname{Tr}_{L^2(\R_x)} \left[    \mathcal{A}\left( \frac{Y_h + Y'_h}{2},\Theta_h\right) \mathcal{T}_{\epsilon}(\Theta_h)    \mathcal{K}_h(Y',Y)  \right]e^{i \Theta \cdot (Y - Y')}  \text{d}_h(Y',\Theta,Y) + \mathcal{O}_\epsilon(h),
\end{align*}
where 
$$
Y_h := h^{1/2} Y, \quad Y'_h := h^{1/2} Y', \quad  \Theta_h = (h^{1/2} \eta, h^{2+1/2}\zeta),
$$ 
the symbol $\mathcal{A}$ is given by \eqref{e:A_1}, and
%
\begin{equation}
\label{e:cut_off_1_operator}
\mathcal{T}_{\epsilon}(\Theta_h) =    \breve{\chi} \left( \frac{h^{2+1/2}\zeta}{\epsilon} \right).
\end{equation}
This motivates the introduction of another distribution acting on $ \mathcal{A}\mathcal{T}_\epsilon$ by:
\begin{equation}
\label{e:distribution_operator_1}
\mathcal{I}_{h}^1(  \mathcal{A}\mathcal{T}_\epsilon) :=  \int_{\R^6} \operatorname{Tr}_{L^2(\R_x)} \Big[    \mathcal{A}(Y_h,\Theta_h) \mathcal{T}_{\epsilon}(\Theta_h)\mathcal{K}_h(Y,Y')   \Big]  e^{i \Theta \cdot (Y - Y')}  \text{d}_h(Y',\Theta,Y).
\end{equation}
By application of the Calderón-Vaillancourt theorem (precisely, we use Lemma \ref{l:Calderon_Vaillancourt}) we obtain:
 \begin{align*}
 \big \vert \mathcal{I}_{h}^1( \mathcal{A} \mathcal{T}_\epsilon) \big \vert &  \lesssim  \Vert  \mathcal{A}(\cdot_Y,\cdot_\Theta) \mathcal{T}_\epsilon(\cdot_\Theta) \Vert_{L^\infty(  \T^2_{Y} \times \R^2_\Theta; \mathcal{L}(L^2(\R_x)))}  \Vert \psi_h \Vert_{L^2(M)}^2 + \mathcal{O}_{\epsilon}(h^{1/2}) \\[0.2cm]
 & = \Vert \mathcal{A}(\cdot_Y,\cdot_\Theta) \Vert_{L^\infty(  \T^2_{Y} \times \R^2_\Theta; \mathcal{L}(L^2(\R_x)))}  \Vert \psi_h \Vert_{L^2(M)}^2 + \mathcal{O}_{\epsilon}(h^{1/2}).
 \end{align*}
Since we can replace the operator $ \mathcal{A}\mathcal{T}_\epsilon$ by any other element of the space of operator-valued symbols $\mathcal{C}_c^\infty( \T^2_Y \times \R^2_\Theta; \mathcal{K}(L^2(\R_x)))$, which is dense in the space  $\mathcal{C}_c( \T^2_Y \times \R^2_\Theta; \mathcal{K}(L^2(\R_x)))$ we get, modulo extraction of a sub-sequence, the existence of a (complex-valued) Radon measure $M_1 \in \mathcal{M}( \mathbb{T}^2_{(y,z)} \times \R_\xi \times \R_\zeta ; \mathcal{L}^1(L^2(\R_x)))$ such that
 $$
 \lim_{h \to 0}\mathcal{I}_{h}^1( \mathcal{A}\mathcal{T}_\epsilon)  =   \int_{\T^2_{y,z} \times \R_\eta \times \R_\zeta}  \operatorname{Tr} \Big(   \mathcal{A}(y,z,\eta,\zeta)\mathcal{T}_\epsilon(\zeta) d M_1 \Big).
 $$ 
 Finally, taking the limit $\epsilon \to 0$, we get the claim.

It remains to show that the measure $M_1$ is positive. To this aim, we proceed as in \cite[Lemma 1.2]{Gerard91}. Assume that
 $$
  \mathcal{A}(Y,\Theta) \geq 0,
 $$ 
 as a symbol taking values in $\mathcal{L}(L^2(\R_x))$. Let us take $\varrho > 0$. Then there exists an operator-valued symbol $\mathcal{B}_{\varrho} \in \mathcal{C}^\infty(\T^2_Y \times \R^2_\Theta; \mathcal{L}(L^2(\R_x)))$ such that
 $$
 \big( \mathcal{B}_{\varrho}(Y,\Theta) \big)^*  \mathcal{B}_{\varrho}(Y,\Theta) = \varrho + \mathcal{A}(Y,\Theta).
 $$
Thus, by the semiclassical pseudodifferential calculus,
 \begin{align*}
 \mathcal{I}_{h}^1( ( \varrho + \mathcal{A})\mathcal{T}_\epsilon) & \\[0.2cm]
  & \hspace*{-2cm} =   \int_{\R^6}  \operatorname{Tr} \left[ \left( \varrho +  \mathcal{A}_\iota \left( \frac{Y_h + Y'_h}{2}, \Theta_h \right)  \right)  \mathcal{K}_h(Y,Y') \right]  \mathcal{T}_{\epsilon}(\Theta_h)e^{i \Theta \cdot (Y-Y')} \text{d}_h(Y',\Theta,Y) \\[0.2cm]
   &\hspace*{-2cm} =  \int_{\R^6} \Big \langle   \mathcal{B}_{\varrho}(Y_h',\Theta_h) \bm{\psi}_h(Y',\Theta), \mathcal{B}_{\varrho}(Y_h,\Theta_h) \bm{\psi}_h(Y,\Theta) \Big \rangle_{L^2(\R_x)}  \mathcal{T}_{\epsilon}(\Theta_h) \text{d}_h(Y',\Theta,Y) \\[0.2cm]
   & \hspace*{-2cm} \quad + \mathcal{O}_{\epsilon}(h^{1/2}),
 \end{align*}
 where
 $$
 \bm{\psi}_h(x,Y,\Theta) := \Psi_h(x,Y,\Theta) \, e^{i \Theta \cdot Y}.
 $$
 Finally, taking limits we conclude that
 \begin{align*}
\lim_{\varrho \to 0}  \lim_{\epsilon \to 0} \lim_{h \to 0} \mathcal{I}_{h}^1(( \varrho + \mathcal{A})\mathcal{T}_\epsilon)  & =  \operatorname{Tr} \int_{\T^2_{y,z} \times \R_\eta \times \R_\zeta}   \mathcal{A}(y,z,\eta,\zeta) dM_1(y,z,\eta,\zeta)  \\[0.2cm]
  &  \geq 0.
 \end{align*}
\end{proof}

\begin{prop}
\label{l:measure_m_2}
There exists a positive Radon measure 
$$
M_2 \in \mathcal{M}_+(M\setminus \mathscr{S} \times \mathbb{T}^2_{y,z} \times \R^*_\sigma; \mathcal{L}^1(L^2(\R_w)))
$$
such that, modulo the extraction of suitable subsequences,
\begin{align*}
\lim_{\rho \to + \infty} \lim_{R \to + \infty} \lim_{\delta \to 0}  \lim_{\epsilon \to 0} \lim_{h \to 0} I_{h,R,\epsilon,\delta,\rho}^2(a) & \\[0.2cm]
 & \hspace*{-3cm} = \operatorname{Tr} \int_{M\setminus \mathscr{S} \times \R^*_\sigma} \Op_{(w,\upsilon)}^\R \Big( a \Big( x,y,z, w, \sigma \upsilon, \sigma, 0 \Big) \Big) dM_2.
 \end{align*}
\end{prop}

\begin{proof}
The main difficulty in this proof is the change of variables aiming to coordinates 
$$
(w',y',z',\upsilon,\eta,\sigma,w,y,z)
$$ 
from the original ones $(X,\Xi,X')$. We will see that the change $(x,\xi) \mapsto (w,\upsilon)$ is asymptotically symplectic (as $\epsilon \to 0$ and $\delta \to 0$ in this order) hence it preserves the pseudodifferential structure. 

We first consider the unitary change of variables given by \eqref{e:symplectic_change_of_coordinates} and \eqref{e:unitary_transformation}. Hence we can rewrite the distribution $I_{h,R,\epsilon,\delta,\rho}^2$ as:
\begin{align}
I_{h,R,\epsilon,\delta,\rho}^2(a)  &   = \int_{\R^9} \mathbf{a}_{h,R,\epsilon}^{\delta,\rho} \left( \frac{X+X'}{2},h\Xi\right)  \psi_h(X') \overline{\psi}_h(X) e^{i \Xi \cdot( X - X')}\text{d}(X',\Xi,X)  \notag \\[0.2cm]
\label{e:integral_after_change}
   & = \int_{\R^9}  \mathbf{A}_{h,R,\epsilon}^{\delta,\rho}\left(\frac{X+X'}{2},h\Xi,\right) \Psi_h(X') \overline{\Psi}_h(X) e^{i \Xi \cdot (X - X')}\text{d}_h(X',\Xi,X),
 \end{align}
 where $\Psi_h$ is given by \eqref{e:unitary_transformation} and $ \mathbf{A}_{h,R,\epsilon}^{\delta,\rho} := \mathbf{a}_{h,R,\epsilon}^{\delta,\rho} \circ \bm{\kappa}_h$. 
Recall that the integral \eqref{e:integral_after_change} is localized by the cut-off 
\begin{align}
\label{e:cut_off_estimate_1}
 \bm{\chi}_{h,R,\epsilon,\delta,\rho}(x+x',\xi,\eta,\zeta) & \\[0.2cm]
  & \hspace*{-2cm}  := \breve{\chi} \left( \frac{h^{1/2}\eta}{\rho} \right)    \breve{\chi} \left( \frac{  h^{2 + 1/2} (x+x') \zeta }{\delta} \right) \chi \left( \frac{h^{2 + 1/2} \zeta}{\epsilon} \right)\breve{\chi} \left(  \frac{h^{1/2} \zeta}{R} \right). \notag
\end{align}
Using next that $a$ has compact support in the $\sigma$ variable, and \eqref{e:cut_off_estimate_1}, we have the following localization of the variable $x$ in the support of integration of \eqref{e:integral_after_change}:
\begin{equation}
\label{e:estimate_on_x}
\frac{\delta}{\epsilon} \leq \frac{ \vert x + x'  \vert}{2}.
\end{equation} 
Moreover, using that $a$ has compact support in the variable $w$, we also have, in the support of integration of \eqref{e:integral_after_change}:
\begin{align}
\label{e:support_a_implication}
- h^{1/2}  \vert \eta \vert + h^{2+1/2} \frac{(x+x')^2}{4} \vert \zeta  \vert \leq \left \vert h^{1/2} \eta  + h^{2+1/2} \frac{(x+x')^2}{4} \zeta \right \vert \lesssim_a 1.
\end{align}
This, together with \eqref{e:estimate_on_x} and the fact that $a$ has compact support in the variable $\sigma$ imply that
\begin{equation}
\label{e:rang_eta_1}
\frac{\delta^2}{\epsilon} \lesssim h^{1/2} \vert \eta  \vert.
\end{equation}
This means that, for $\delta >0$ fixed and $\epsilon > 0$ sufficiently small, the effect of the cut-off 
$$
\breve{\chi} \left( \frac{h^{1/2}\eta}{\rho} \right)  
$$ 
in \eqref{e:cut_off_estimate_1} is empty. To simplify notations, we drop the index $\rho$ in the rest of this proof. We can now make the change of coordinates previously mentioned. Nevertheless, we will slightly deform the variable $\sigma$ to improve their commutation properties. This is reminiscent to the normal form procedure introduced in \cite[\S 5.2]{Ar_Riv24}, and which we will return to in Section \ref{s:normal_form}. Precisely, we define:
\begin{equation}
\label{e:change_theta}
\vartheta \, : \, ( X',\Xi,X) \longmapsto (W',\mathfrak{S},W) = (\mathbf{w}',y',z',\upsilon,\eta,\bm{\sigma},\mathbf{w},y,z)
\end{equation}
by the following formulas: 
\begin{equation}
\label{e:tricky_change}
\upsilon := \frac{\xi}{ h \bm{\sigma} }, \quad \bm{\sigma} := \sgn \big( (x+x')\zeta \big) \sqrt{- 4h \eta \zeta  }, \quad \left \lbrace \begin{array}{l} 
 \mathbf{w} :=  \displaystyle  2h^{1/2} \eta +  h\bm{\sigma}x, \\[0.4cm]
 \mathbf{w}' := \displaystyle 2 h^{1/2}\eta + h \bm{\sigma}x'.
 \end{array} \right.
\end{equation}
This change of variables is well defined when restricted to the support of integration of \eqref{e:integral_after_change}. Set now (with a slightly abuse of notation)
\begin{align*} 
\sigma & = \sigma\left( \frac{h(x+x')}{2},h^{1/2}\zeta \right)   = h^{1+1/2} (x+x') \zeta, \\[0.2cm]
w & = h^{1/2}\eta + \frac{h \sigma x}{2}, \quad w'  = h^{1/2} \eta + \frac{h \sigma x'}{2}.
\end{align*}
As we will see, $\bm{\sigma}$ approximates $\sigma$ as $\epsilon \to 0$, and its absolute value is independent of $x$, which has some adventages regarding commuting properties. Our aim is next to write $\bm{\sigma}$ in terms of $(\sigma,\mathbf{w},\eta)$. To do this, we first notice that
\begin{align}
\label{e:hx_in_terms_of_w_eta}
\frac{hx}{2} & =  \frac{ w - h^{1/2} \eta}{\sigma},\\[0.2cm]
\label{e:hx_prime_in_terms_of_w_eta}
 \frac{hx'}{2} & =  \frac{w' - h^{1/2}\eta}{\sigma}, \\[0.2cm]
\label{e:zeta_in_terms_of_w_eta}
h^{1/2} \zeta & = \frac{\sigma^2}{2(w +w' - 2h^{1/2}\eta)}.
\end{align}
On the other hand, from \eqref{e:hx_in_terms_of_w_eta}, \eqref{e:hx_prime_in_terms_of_w_eta}, and \eqref{e:zeta_in_terms_of_w_eta}, the fact that $a$ has compact support in the variable $\mathbf{w}$, and \eqref{e:rang_eta_1}, we also observe that:
$$
\sgn((x+x')\zeta) = \sgn(\sigma), \quad \sgn(\zeta) = - \sgn(\eta).
$$
Thus we get:
\begin{align}
\label{e:sigma_sigma}
\bm{\sigma} & =    \frac{ \sigma}{\sqrt{1 -  \displaystyle \frac{ w + w'}{ 2h^{1/2} \eta}  }} = \frac{\sigma}{\displaystyle 1 - \frac{ \mathbf{w} + \mathbf{w}'}{4 h^{1/2}\eta}}.
\end{align}
%
By using \eqref{e:rang_eta_1}, the localization of the cut-off \eqref{e:cut_off_estimate_1} and the compact-support properties of $a$, this implies that:
\begin{align}
\label{e:sigma_and_sigma_x}
\bm{\sigma} & = \sigma  \left( 1 + \mathcal{O}\left( \frac{\epsilon}{\delta^2} \right) \right).
\end{align}
This together with \eqref{e:cut_off_estimate_1} gives that $\delta \leq h \vert \bm{\sigma} \vert( 1  + \mathcal{O}(\epsilon \delta^{-2})) \lesssim_a 1$ on the support of integration of \eqref{e:integral_after_change}, and we emphasize that we will take the limit $\epsilon \to 0$ before the limit $\delta \to 0$.   By \eqref{e:tricky_change}, we can also write $\zeta$ in terms of $\eta$ and $\bm{\sigma}$ by
\begin{align}
\label{e:zeta_sigma}
h^{1/2} \zeta & = -\frac{\bm{\sigma}^2}{4h^{1/2} \eta}.
\end{align}
Moreover, by \eqref{e:tricky_change}, we have:
$$
x - x' = \frac{\mathbf{w}-\mathbf{w}'}{h\bm{\sigma}}.
$$
Finally, we can write 
\begin{align}
\label{e:w_to_w}
\frac{w + w'}{2} = \frac{ \mathbf{w} + \mathbf{w}'}{2} - \frac{ (\mathbf{w} + \mathbf{w}')^2}{16h^{1/2} \eta}. 
\end{align}
Therefore we obtain, on the support of integration of the integral \eqref{e:integral_after_change}, the change-of-variables formula:
\begin{align*}
\text{d}_h(X',\Xi,X) & = \frac{1}{2h^2 \vert \eta \vert}  \, \text{d}_h(W',\mathfrak{S},W),
\end{align*}
where the new measure of integration in the variables $(W,\mathfrak{S},W')$ is given by 
$$
\text{d}_h(W',\mathfrak{S},W) = \vartheta_* \, \text{d}_h(X,\Xi,X').
$$ 

We turn at this point to writing the distribution $I^2_{h,R,\epsilon,\delta}(a)$ in terms of the variables $(W',\mathfrak{S},W)$. Denoting $Y = (y,z)$ and in view of \eqref{e:zeta_sigma}, we consider the function: 
\begin{align*}
\bm{\psi}_h(\mathbf{w},Y,\eta,\bm{\sigma}) & :=  \frac{1}{h \sqrt{2 \vert \eta \vert }} \Psi_h \circ \vartheta^{-1} \left( \mathbf{w},Y,\eta,\bm{\sigma} \right)  \exp \left( iY \cdot \left( \eta, - \frac{ \bm{\sigma}^2}{4h \eta} \right) \right).
\end{align*}
Introducing the operator-valued symbol
\begin{align}
\label{e:operator_from_symbol}
\mathcal{A}_{h,R,\epsilon}^\delta(Y,\eta,\bm{\sigma}) :=  \Op_{(\mathbf{w},\upsilon)}^{\R}( \mathbf{A}_{h,R,\epsilon}^\delta \circ \vartheta^{-1} ),
\end{align}
we get the formal expression:
\begin{align}
\label{e:first_trace_form_equation}
 I^2_{h,R,\epsilon,\delta}(a) &\\[0.2cm] 
& \hspace*{-1.3cm} = \int_{\R^4_{(Y ,Y')} \times \R_\eta \times \R_{\bm{\sigma}}} \operatorname{Tr}_{L^2(\R_{\mathbf{w}})} \left[  \mathcal{A}_{h,R,\epsilon}^\delta \left( \frac{Y+Y'}{2},\eta,\bm{\sigma} \right) \, \mathcal{K}_h (Y,Y',\eta,\bm{\sigma}) \right] \text{d}_h(Y,Y',\eta,\bm{\sigma})  \notag ,
\end{align}
where the operator  $\mathcal{K}_h  \in \mathcal{C}^\infty(\R^4_{(Y,Y')}  \times \R_\eta \times \R_\sigma; \mathcal{L}^1(L^2(\R_{\mathbf{w}})))$ is the trace-class (rank-one) operator-valued symbol with integral kernel given by the tensor product:
$$
\bm{\psi}_h \otimes \overline{\bm{\psi}}_h(\mathbf{w},\mathbf{w}',Y,Y',\bm{\sigma},\eta) := \bm{\psi}_h(\mathbf{w}',Y',\eta,\bm{\sigma}) \overline{\bm{\psi}}_h(\mathbf{w},Y,\eta,\bm{\sigma}).
$$
Trace formula \eqref{e:first_trace_form_equation} takes into account the change of variables $\vartheta$ and couples the operator $\mathcal{A}_{h,R,\epsilon}^\delta$ with the trace-class operator $\mathcal{K}_h$. This suggests that the measure $M_2$ will be obtained as an accumulation point of the sequence $(\mathcal{K}_h)_{h> 0}$ for the weak topology in the space of trace-class operators. However, it is not yet clear how to bound \eqref{e:first_trace_form_equation} in terms of the $L^\infty$ norm of the derivatives of $a$. The difficulty comes because although the pseudodifferential structure has been conserved by the change of coordinates $\vartheta$ with respect to the map from variables $(x,\xi)$ to variables $(\mathbf{w},\upsilon)$, this is not the case with respect to the change $(z,\zeta)$ into $(z,\bm{\sigma})$. 

To overcome this difficulty, we come back to \eqref{e:integral_after_change} and perform some integrations by parts, before changing variables, in order to gain decayment in the variables $(Y,\eta,\bm{\sigma},Y')$, similarly to the strategy of proof of the Calderón-Vaillancourt theorem (see Lemma \ref{l:Calderon_Vaillancourt} and \cite{Hwang87}). Let us denote $Y = (y,z)$, $Y' = (y',z')$, $\Theta = (\eta,\zeta)$ and set:
 $$
 \begin{array}{ll}
 \Gamma_h\left(Y \right)  \displaystyle :=  \frac{1}{1 + iy_h} \cdot \frac{1}{1 + iz_h},  & \hspace*{0.2cm} \mathfrak{D}_\Theta  := (1 + \partial_\zeta^h)(1 + \partial_\eta^h),  \\[0.4cm]
 \hspace*{0.25cm} \Gamma(\Theta)  \displaystyle := \frac{1}{1 + i\eta} \cdot \frac{1}{1 + i\zeta}, & \hspace*{0.2cm} \mathfrak{D}_Y  := (1 + \partial_y)(1 + \partial_z),
 \end{array}
$$
where the discrete derivatives $\partial_\xi^h$ and $\partial_\eta^h$ are introduced in Definition \ref{d:discret_derivatives} and the terms $y_h$, $z_h$ are given by \eqref{e:h_points}. We next use the following identities:
\begin{align*}
e^{i \Theta \cdot Y} & = \Gamma_h(Y) \, \mathfrak{D}_\Theta \, e^{i \Theta \cdot Y},  \\[0.2cm]
e^{i \Theta \cdot Y} & = \Gamma(\Theta)\, \mathfrak{D}_Y \,  e^{i \Theta \cdot Y },
\end{align*}
and define, from Definition \ref{d:discret_derivatives}, the adjoint operators
\begin{align*}
\mathfrak{D}^*_\Theta :=  (1 - \overline{\partial}^h_\eta)(1 - \overline{\partial}^h_\zeta), \quad 
\mathfrak{D}^*_Y := (1 - \partial_y)(1 - \partial_z).
\end{align*}
The following functions will appear in the integration-by-parts process:
 \begin{align*}
  \mathbf{B}_{h,R,\epsilon}^\delta(X,\Xi,X')  & := \mathbf{A}_{h,R,\epsilon}^\delta\left( \frac{ X + X'}{2} ,h\Xi \right)  \Gamma^2_h\left(  Y - Y' \right), \\[0.2cm]
  \mathbf{H}_1(x',Y',\Theta) & := e^{-i Y' \cdot \Theta} \int_{\R^2} e^{i Y' \cdot \Lambda} \Gamma(\Lambda - \Theta) \mathcal{F}_{Y'}\Psi_h(x',\Lambda) \text{d}(\Lambda), \\[0.2cm]
\mathbf{H}_2(x,Y,\Theta) & := e^{i Y \cdot \Theta} \int_{\R^2} e^{i Y \cdot \Lambda} \Gamma(\Theta + \Lambda) \mathcal{F}_Y\overline{\Psi}_h(x,\Lambda) \text{d}(\Lambda).
 \end{align*}
Then, by successive integrations by parts in \eqref{e:integral_after_change}, first in the $\Theta$ variable and next in the $Y$ and $Y'$ varibles, we obtain:
 \begin{align*}
I_{h,R,\epsilon,\delta}^2(a) & \\[0.2cm]
 & \hspace*{-1.2cm} =  \int_{\R^9} e^{i \xi (x-x')} \mathfrak{D}^*_Y (\mathfrak{D}^*_\Theta)^2 \mathfrak{D}^*_{Y'} \mathbf{B}_{h,R,\epsilon}^\delta(X,\Xi,X')  \, \mathbf{H}_1(X',\Theta) \mathbf{H}_2(X,\Theta) \text{d}_h(X,\Xi,X').
 \end{align*}
Thus, by changing variables we obtain
\begin{align}
I_{h,R,\epsilon,\delta}^2(a) & \notag \\[0.2cm]
\label{e:distributional_form}
 & \hspace*{-1.2cm} = \int_{\R^9}  \Big(    \mathfrak{D}^*_Y (\mathfrak{D}^*_\Theta)^2 \mathfrak{D}^*_{Y'} \mathbf{B}_{h,R,\epsilon}^\delta  \, \mathbf{H}_1 \, \mathbf{H}_2 \Big) \circ \vartheta^{-1}(W,\mathfrak{S},W') e^{i \upsilon (\mathbf{w} - \mathbf{w}')}   \frac{ \text{d}_h(W',\mathfrak{S},W)}{2h^2 \vert \eta \vert}.
 \end{align}
\medskip

The next step is to rewrite \eqref{e:distributional_form} again in trace form. To this aim, let us define
\begin{align*}
\mathcal{H}_1(\mathbf{w}',Y',\eta,\bm{\sigma}) := \frac{\mathbf{H}_1 \circ \vartheta^{-1}( \mathbf{w}', Y', \eta, \bm{\sigma})}{ h \sqrt{2\vert \eta \vert}}, \quad 
\mathcal{H}_2(\mathbf{w},Y,\eta,\bm{\sigma}) :=  \frac{\mathbf{H}_2 \circ \vartheta^{-1}( \mathbf{w}, Y, \eta, \bm{\sigma})}{h \sqrt{2\vert \eta \vert}},
\end{align*}
and, for each $(Y, Y', \eta,\bm{\sigma})$, let  $\mathcal{K}^h_{\mathcal{H}_1 \otimes \mathcal{H}_2}(Y,Y',\eta,\bm{\sigma})$  be the trace-class (rank-one) operator with integral kernel given by
$$
 \mathcal{H}_2(\mathbf{w},Y,\eta,\bm{\sigma}) \mathcal{H}_1(\mathbf{w}',Y',\eta,\bm{\sigma}) \mathfrak{D}_Y^* \mathfrak{D}_{Y'}^* \Gamma^2_h(Y- Y').
$$ 
Then we can once again rewrite the distribution $I^2_{h,R,\epsilon,\delta}(a)$ in trace form as
\begin{align*}
I^2_{h,R,\epsilon,\delta}(a) & \\[0.2cm]
&  \hspace*{-1.5cm} = \int_{\R^2_Y \times \R_\eta \times \R_{\bm{\sigma}}} \operatorname{Tr}_{L^2(\R_{\mathbf{w}})} \Big[ \mathcal{A}_{h,R,\epsilon}^\delta\left( \frac{Y + Y'}{2} ,\eta,\bm{\sigma} \right) \, \mathcal{K}^h_{ \mathcal{H}_1 \otimes \mathcal{H}_2}(Y,Y',\eta,\bm{\sigma}) \big) \Big] \text{d}_h(Y,Y',\eta,\bm{\sigma}) +  \mathcal{E}^1_{h,R,\epsilon,\delta},
\end{align*}
where  the remainder term $\mathcal{E}^1_{h,R,\epsilon,\delta}$ comes from a sum of derivatives of the cut-off \eqref{e:cut_off_estimate_1} and the symbol $a$ produced by the differential operator $\mathfrak{D}^*_Y (\mathfrak{D}^*_\Theta)^2\mathfrak{D}^*_{Y'}$ acting on the symbol $\mathbf{A}_{h,R,\epsilon}^\delta$. Precisely, one has the estimate
 $$
 \mathcal{E}^1_{h,R,\epsilon,\delta} = \mathcal{O}_{\delta}(\epsilon) + \mathcal{O}\left( \frac{h^{1/2}}{R} \right)  + \mathcal{O}\left( \frac{h^{1+1/2}}{\delta} \right) + \mathcal{O}\left( \frac{h^{2 + 1/2}}{\epsilon} \right).
 $$
Notice moreover that, by  \eqref{e:cut_off_estimate_1}, \eqref{e:hx_in_terms_of_w_eta}, \eqref{e:sigma_sigma}, and \eqref{e:zeta_sigma}, and defining $\alpha(\eta,\bm{\sigma}) = \eta/\bm{\sigma}$, we have:
\begin{align*}
\mathbf{A}_{h,R,\epsilon}^\delta  \circ \vartheta^{-1}(\mathbf{w},Y,\eta,\bm{\sigma},\upsilon) & \\[0.2cm]
 & \hspace*{-3.5cm} =   \bm{\chi}_{h,R,\epsilon,\delta} \circ \vartheta^{-1} (\mathbf{w},\eta,\bm{\sigma}) a \left(\frac{ \mathbf{w} -h^{1/2} \eta }{ \bm{\sigma} } , h^{1/2}Y, w,  \frac{h \bm{\sigma}   \upsilon}{2}, h \sigma,- \frac{(h\bm{\sigma})^2}{4h^{1/2}\eta} \right)  \\[0.2cm]
 & \hspace*{-3.5cm}  =   \bm{\chi}_{h,R,\epsilon,\delta} \circ \vartheta^{-1} (0,\eta,\bm{\sigma}) a \left( \frac{ -h^{1+1/2} \eta}{ h\bm{\sigma} }, h^{1/2}Y, \mathbf{w}, \frac{h \bm{\sigma} \upsilon}{2},h\bm{\sigma},0 \right)  + \mathbf{E}^2_{h,R,\epsilon,\delta}(W,\mathfrak{S}),
\end{align*}
where $\sigma = \sigma(\mathbf{w}, \eta, \bm{\sigma})$ is given by \eqref{e:sigma_sigma} and $w = w(\mathbf{w},\eta)$ is given by \eqref{e:w_to_w}. The remainder term  $\mathbf{E}^2_{h,R,\epsilon,\delta}$, in view of \eqref{e:operator_from_symbol}, must be estimated in operator norm as:
$$
\mathcal{E}^2_{h,R,\epsilon,\delta} := \big \Vert \Op_{(\mathbf{w},\upsilon)}^\R \big( \mathbf{E}^2_{h,R,\epsilon,\delta} \big) \big \Vert_{L^\infty(Y,\eta,\bm{\sigma}; \mathcal{L}(L^2(\R_\mathbf{w})))}  = \mathcal{O}_{R,\epsilon,\delta}(h)+ \mathcal{O}\left( \frac{\epsilon}{\delta^2} \right).
$$

We next pass from the distribution $I^2_{h,R,\epsilon,\delta}$ to a distribution acting on the space of compact operators on $L^2(\R_{\mathbf{w}})$. Let us define, for $\bm{\sigma} \neq 0$,
\begin{align}
\label{e:effective_symbol}
\mathcal{A}_\infty(Y,\eta,\bm{\sigma}) &  :=  \Op_{(\mathbf{w},\upsilon)}
^\R \left(a \left( -\frac{\eta}{\bm{\sigma}}, Y, \mathbf{w}, \bm{\sigma} \upsilon,\bm{\sigma},0 \right) \right), \\[0.2cm]
\mathcal{T}_{\delta}( \bm{\sigma}) & :=  \breve{\chi} \left(\frac{ \bm{\sigma}}{\delta} \right) , \\[0.2cm]
\label{e:cut_offs_epsilon_R}
\bm{\chi}_{h,R,\epsilon}(\eta,\bm{\sigma}) &:= \chi \left(  \frac{ h^2 \bm{\sigma}^2}{4h^{1/2}\eta \epsilon} \right) \breve{\chi} \left( \frac{\bm{\sigma}^2}{4h^{1/2} \eta R} \right).
\end{align}
This motivates the introduction of the distribution:
\begin{align}
\label{e:final_Trace_formula_2}
\mathcal{I}^2_{h,R,\epsilon}( \mathcal{A}_\infty \mathcal{T}_{\delta}) & \\[0.2cm]
 & \hspace*{-2cm}  :=    \int_{\R^2_Y \times \R_\eta \times \R_{\bm{\sigma}}} \operatorname{Tr}_{L^2(\R_{\mathbf{w}})} \Big( \mathcal{A}_{\infty}(h^{1/2}Y,h^{1+1/2}\eta,h\bm{\sigma}) \mathcal{T}_{\delta}(h\bm{\sigma}) \, \mathcal{K}^h_{ \mathcal{H}_1 \otimes \mathcal{H}_2}(Y,Y',\eta,\bm{\sigma}) \Big)  \bm{\chi}_{h,R,\epsilon}(\eta,\bm{\sigma})    \text{d}_h. \notag
\end{align}
We remark that $I^2_{h,R,\epsilon,\delta}(a) = \mathcal{I}^2_{h,R,\epsilon}(\mathcal{A}_\infty \mathcal{T}_\delta) + \mathcal{E}^3_{h,R,\epsilon,\delta}$
with 
\begin{equation}
\label{e:Error_3}
 \mathcal{E}^3_{h,R,\epsilon,\delta} := \left( 1 +  \mathcal{E}^2_{h,R,\epsilon,\delta}  \right) \mathcal{E}^1_{h,R,\epsilon,\delta},
\end{equation}
and we also emphasize that, for $h > 0$ sufficiently small, the effect of the cut-off on $R$ in \eqref{e:cut_offs_epsilon_R} is empty. By using Lemma \ref{l:Hwang}, we obtain the uniform estimate
\begin{align*}
\big \vert  \mathcal{I}^2_{h,R,\epsilon}(\mathcal{A}_\infty \mathcal{T}_{\delta}) \big \vert &\\[0.2cm]
 & \hspace*{-1.5cm} \lesssim \Vert \mathcal{A}_\infty \mathcal{T}_{\delta} \Vert_{L^\infty(\R^2_Y \times \R^*_\eta \times \R^*_{\bm{\sigma}};   \mathcal{L}(L^2(\R_{\mathbf{w}})))} \Vert \mathcal{H}_1 \Vert_{L^2(\R_{\mathbf{w}}\times \R^2_Y \times \R^*_\eta \times \R_{\bm{\sigma}})} \Vert \mathcal{H}_2 \Vert_{L^2(\R_{\mathbf{w}}\times \R^2_Y \times \R^*_\eta \times \R_{\bm{\sigma}})} \\[0.2cm]
 & \hspace*{-1.5cm} \lesssim \Vert \mathcal{A}_\infty \Vert_{L^\infty(\R^2_Y \times \R^*_\eta \times \R^*_{\bm{\sigma}};   \mathcal{L}(L^2(\R_{\mathbf{w}})))} \Vert \mathbf{H}_1 \Vert_{L^2} \Vert \mathbf{H}_2 \Vert_{L^2} \\[0.2cm]
 & \hspace*{-1.5cm} \lesssim \Vert \mathcal{A}_\infty \Vert_{L^\infty(\R^2_Y \times \R^*_\eta \times \R^*_{\bm{\sigma}};   \mathcal{L}(L^2(\R_{\mathbf{w}})))} \Vert \psi_h \Vert_{L^2(M)}^2.
\end{align*}
 Therefore, since we can replace $\mathcal{A}_\infty \mathcal{T}_{\delta}$ by any other element of the space of operator-valued symbols $\mathcal{C}_c^\infty(\R^2_Y \times \R^*_\eta \times \R^*_{\bm{\sigma}};   \mathcal{K}(L^2(\R_{\mathbf{w}})))$, we get by taking limits through succesive subsequences, the existence of a Radon measure $\mathbf{m}_2 \in \mathcal{M}(\mathbb{T}^2_Y \times \R^*_\eta \times \R^*_{\bm{\sigma}})$ such that
\begin{align*}
 \lim_{\epsilon \to 0} \lim_{h \to 0} \mathcal{I}^2_{h,R,\epsilon}(\mathcal{A}_\infty \mathcal{T}_{\delta})  = \int_{\T^2_{Y} \times \R_\eta^* \times \R_{\bm{\sigma}}^*} \operatorname{Tr} \Big[\mathcal{A}_\infty(Y,\eta,\bm{\sigma}) \mathcal{T}_\delta(\bm{\sigma}) \, d\mathbf{m}_2(Y,\eta,\bm{\sigma})  \Big].
\end{align*}
Thus:
$$
\lim_{R \to + \infty} \lim_{\delta \to 0} \lim_{\epsilon \to 0} \lim_{h \to 0} I_{h,R,\epsilon,\delta}^2(a) = \int_{\T^2_{Y} \times \R_\eta^* \times \R_{\bm{\sigma}}^*} \operatorname{Tr} \Big[\mathcal{A}_\infty(Y,\eta,\bm{\sigma})  \, d\mathbf{m}_2(Y,\eta,\bm{\sigma})  \Big].
$$
Finally, in view of \eqref{e:hx_in_terms_of_w_eta}, we make the change of variables : 
$$
\mathbf{x} := \frac{-\eta}{\bm{\sigma}}, \quad \eta = -\bm{\sigma} \mathbf{x},
$$
and define:
$$
M_2(\mathbf{x},y,z,\bm{\sigma}) := \mathbf{m}_2\left(y,z,-\bm{\sigma} \mathbf{x}, \bm{\sigma} \right).
$$
Thus:
\begin{align*}
  \int_{\T^2_{y,z} \times \R_\eta^* \times \R_{\bm{\sigma}}^*} \operatorname{Tr} \Big( \mathcal{A}_\infty(Y,\eta,\bm{\sigma}) \, d\mathbf{m}_2(y,z,\eta,\bm{\sigma}) \Big) & \\[0.2cm]
  & \hspace*{-5cm} =  \int_{M_{\mathbf{x},y,z} \setminus \mathscr{S} \times \R_{\bm{\sigma}}^*}   \operatorname{Tr} \Big( \Op_{\mathbf{w},\upsilon}^\R\left( a\left(\mathbf{x},y,z,\mathbf{w}, \bm{\sigma} \upsilon, \bm{\sigma},0 \right) \right)  \, dM_2(\mathbf{x},y,z,\bm{\sigma}) \Big).
\end{align*}

It remains to show the positivity of $\mathbf{m}_2$, which gives immediately the positivity of $M_2$. Assume that the operator-valued symbol $\mathcal{A}_\infty \in \mathcal{C}^\infty( \R^2_Y \times \R_\eta^* \times \R_{\bm{\sigma}}^* ; \mathcal{L}(L^2(\R_{\mathbf{w}})))$ given by \eqref{e:effective_symbol} takes values, for each $(Y,\eta,\bm{\sigma}) \in \R^2_Y \times \R^*_\eta \times \R^*_{\bm{\sigma}}$ in the space of positive self-adjoint operators. Then, for each $\varrho > 0$, there exists another symbol $\mathcal{B}_{\varrho} \in \mathcal{C}^\infty( \R^2_Y \times \R_\eta^* \times \R_{\bm{\sigma}}^* ; \mathcal{L}(L^2(\R_{\mathbf{w}})))$ such that:
$$
 \mathcal{B}_{\varrho}(Y,\eta,\bm{\sigma})^*  \mathcal{B}_{\varrho}(Y,\eta,\bm{\sigma}) = \varrho + \mathcal{A}_\infty(Y,\eta,\bm{\sigma}).
$$
We define:
$$
\bm{\psi}_h^{\varrho}(\mathbf{w},Y,\eta,\bm{\sigma}) := \mathcal{B}_{\varrho}(h^{1/2}Y,h^{1+1/2}\eta,h\bm{\sigma})  \bm{\psi}_h(\mathbf{w},Y,\eta,\bm{\sigma}).
$$
Then, in view of \eqref{e:first_trace_form_equation}, and using the semiclassical symbolic calculus to replace $Y$ by $Y'$ in the operator-valued symbol $\mathcal{B}_{\varrho}(Y,\eta,\bm{\sigma})^*$, we obtain:
\begin{align*}
\mathcal{I}_{h,R,\epsilon}^2( (\varrho + \mathcal{A}_\infty) \mathcal{T}_{\delta}) & \\[0.2cm]
 & \hspace*{-2.9cm} = \int_{\R^4_{(Y,Y')} \times \R_\eta \times \R_{\bm{\sigma}}} \operatorname{Tr} \big[ ( \varrho + \mathcal{A}_\infty)(h^{1/2} Y,h^{1+1/2} \eta,h\bm{\sigma}) \mathcal{T}_{\delta}(h\bm{\sigma})  \mathcal{K}_h(Y,Y',\eta,\bm{\sigma}) \big] \bm{\chi}_{h,R,\epsilon}(\eta,\bm{\sigma}) \text{d}_h \\[0.2cm]
 & \hspace*{-2.9cm} =  \int_{\R_{\mathbf{w}} \times \R^4_{(Y,Y')} \times \R_\eta \times \R_{\bm{\sigma}}} \mathcal{T}_{\delta}(h\bm{\sigma}) \bm{\psi}_h^{\varrho}(\mathbf{w},Y',\eta,\bm{\sigma}) \overline{ \bm{\psi}_h^{\varrho}(\mathbf{w},Y,\eta,\bm{\sigma})} \bm{\chi}_{h,R,\epsilon}(\eta,\bm{\sigma}) \text{d}_h +\mathcal{E}^3_{h,R,\epsilon,\delta} \\[0.2cm]
 & \hspace*{-2.9cm} = \int_{\R_{\mathbf{w}} \times \R^*_\eta \times \R^*_{\bm{\sigma}}} \mathcal{T}_{\delta}(h\bm{\sigma}) \left \vert \int_{\R_Y}  \bm{\psi}_h^{\varrho}(\mathbf{w},Y,\eta,\bm{\sigma}) \text{d}_h(Y) \right \vert^2 \bm{\chi}_{h,R,\epsilon}(\eta,\bm{\sigma}) \text{d}_h(\mathbf{w},\eta,\bm{\sigma}) +  \mathcal{E}^3_{h,R,\epsilon,\delta},
\end{align*}
where
$$
\mathcal{E}^3_{h,R,\epsilon,\delta} = \mathcal{O}\left( \frac{\epsilon}{\delta^2} \right)  + \mathcal{O}_{R,\epsilon,\delta,\varrho}(h^{1/2}). 
$$
Finally, taking limits we conclude that
$$
 \lim_{\varrho \to +\infty} \lim_{R \to + \infty} \lim_{\delta \to 0} \lim_{\epsilon \to 0} \lim_{h \to 0} \mathcal{I}_{h,R,\epsilon}^2((\varrho + \mathcal{A}_\infty)\mathcal{T}_{\delta}) =  \int_{\T^2_{(y,z)} \times \overline{\R}_\eta^* \times \R_{\bm{\sigma}}^*}  \operatorname{Tr} \Big(  \mathcal{A}_\infty \, d\mathbf{m}_2(y,z,\bm{\sigma},\eta) \Big)  \geq 0.
$$
This concludes the proof.

\end{proof}

\begin{prop}
\label{l:measure_m_4}
For each $\jmath \in \{0,1\}$, there exists a positive Radon measure 
$$
m_3^{\jmath} \in \mathcal{M}_+(\T_y \times \T_z \times   \R_\xi \times \R_\eta \times \R_\varsigma)
$$ 
such that, modulo the extraction of succesive subsequences:
\begin{align*}
\lim_{\rho \to + \infty}  \lim_{R \to + \infty}  \lim_{\epsilon \to 0} \lim_{h \to 0} I_{h,R,\epsilon,\rho}^3(a) & \\[0.2cm]
 & \hspace*{-3.5cm}  = \sum_{\jmath \in \{0,1\}} \int_{\T^2_{y,z} \times  \R_\xi  \times \R_\eta \times \R_\sigma \times \R_\varsigma} a(0,y,z,\eta + (-1)^{ \jmath} \varsigma^2, \xi,\sigma,0) dm_3^{\jmath}.
 \end{align*}
 
\end{prop}

\begin{proof}
Let $\delta > 0$. We consider in this proof the dilation-contraction symplectic change of variables 
$$
\bm{\kappa}_{h,\delta} :   T^*M   \rightarrow \R_x \times h^{-1/2} \mathbb{T}^2_{(y,z)} \times \R_\xi  \times h^{1/2} \mathbb{Z}^2_{(\eta,\zeta)} \times \R_{x'} \times h^{-1/2} \mathbb{T}^2_{(y',z')} $$
given by
\begin{align}
\label{e:symplectic_change_of_coordinates_delta_iota}
\bm{\kappa}_{h,\delta}(x,y,z,\xi,\eta,\zeta,x',y',z') & \\[0.2cm]
 & \hspace*{-2.8cm} = \left( \frac{h x}{\delta^{1/2}} ,h^{1/2}y, h^{1/2}z, \frac{\delta^{1/2}\xi}{h},h^{-1/2} \eta, h^{-1/2} \zeta, \frac{h x'}{\delta^{1/2}}, h^{1/2}y', h^{1/2} z' \right), \notag
\end{align}
and introduce the unitary transformation
\begin{equation}
\label{e:unitary_transformation_delta_iota}
\psi_h \mapsto \Psi_{h,\delta}(x,y,z) := \frac{h}{\delta^{1/4}} \psi_h\left( \frac{hx}{\delta^{1/2}},h^{1/2}y,h^{1/2}z \right).
\end{equation} 
This allows us to rewrite the Wigner distribution $I^3_{h,R,\epsilon,\rho}(a)$ as:
$$
I_{h,R,\epsilon,\rho}^3(a) =  \int_{\R^3 \times \R^3 \times \R^3} \mathbf{A}_{h,R,\epsilon,\delta,\rho} \left( \frac{X+X'}{2},h\Xi, \right) \Psi_{h,\delta}(X') \overline{\Psi}_{h,\delta}(X) e^{i \Xi \cdot(X - X' )} \text{d}_{h,\delta}(X',\Xi,X),
$$
where $\mathbf{A}_{h,R,\epsilon,\delta,\rho} = \mathbf{a}_{h,R,\epsilon,\delta,\rho} \circ \bm{\kappa}_{h,\delta}$ and $ \text{d}_{h,\delta} = (\bm{\kappa}_{h,\delta})_* \text{d}$. Let us define: 
\begin{align}
\label{e:change_to_varsigma}
x_{h,\delta} & := \frac{hx}{\delta^{1/2}}, \quad 
  \varsigma^2(x,\zeta)   := x^2 \vert \zeta \vert, \quad
\jmath (\zeta)  := \left \lbrace \begin{array}{ll}
0, & \text{if } \sgn(\zeta) = 1, \\[0.2cm]
1, & \text{if } \sgn(\zeta) = -1,
\end{array} \right.
\end{align}
and denote
\begin{align}
\label{e:cut_off_scalar_3}
\bm{\chi}_{h,R,\epsilon,\rho}(\eta,\zeta) := \chi \left( \frac{h^{1/2}\eta}{\rho} \right)  \chi\left( \frac{h^{2+ 1/2}\zeta}{\epsilon} \right) \breve{\chi}\left( \frac{h^{1/2} \zeta }{R}\right).
\end{align}
Notice that, the localization of the cut-off \eqref{e:cut_off_scalar_3} and the fact that  $a$ has compact support in the variable $w$ imply that:
$$
\frac{h \vert x  \vert}{\delta^{1/2}} \lesssim_a \left( \frac{\rho}{R} \right)^{1/2}.
$$
Let us next introduce the symbol:
\begin{align}
\label{e:a_jmath}
\mathbf{a}_{\jmath}(y,z,\xi,\eta,\sigma,\varsigma) & := a \left(0, y,z, \eta + (-1)^{ \jmath} \varsigma^2, \xi,\sigma,0 \right), \quad \jmath \in \{0,1\}.
\end{align}
Then, on the support of $\bm{\chi}_{h,R,\epsilon,\delta,\rho}$, we have:
\begin{align*}
 a\left( x_{h,\delta}, h^{1/2}y, h^{1/2}z,w(x_{h,\delta},h^{1/2}\eta,h^{1/2}\zeta),\delta^{1/2}\xi, h \sigma(x_{h,\delta}, h^{1/2}\zeta), h^{2+1/2}\zeta \right) & \\[0.2cm]
& \hspace*{-10cm} =  \mathbf{a}_{\jmath} \circ \vartheta_{h,\delta}(X,\Xi) +\mathcal{O}(\epsilon)  + \mathcal{O}\left( \left( \frac{\rho}{R} \right)^{1/2} \right),
\end{align*}
where
\begin{align}
\label{e:vartheta_h_delta}
\vartheta_{h,\delta}(X,\Xi) = \big( h^{1/2}y, h^{1/2}z, \delta^{1/2}\xi, h^{1/2}\eta, h^{1+1/4} \delta^{-1/2} \varsigma  \big).
\end{align}
Applying Lemma \ref{l:Calderon_Vaillancourt} and using the localization properties of the cut-off \eqref{e:cut_off_scalar_3}, we get:
\begin{align*}
\big \vert I_{h,R,\epsilon,\rho}^3(a) \big \vert & \lesssim  \sum_{\alpha \in \mathcal{N}} \big \Vert \partial^\alpha \mathbf{A}_{h,R,\epsilon,\rho} \big  \Vert_{L^\infty} \Vert \psi_h \Vert_{L^2(M)}^2 \\[0.2cm]
& \lesssim \sum_{\jmath \in \{0,1\}} \Vert \mathbf{a}_{\jmath}   \Vert_{L^\infty( \T^2_{y,z} \times \R_\xi \times \R_\eta \times \R_\sigma \times \R_\varsigma )} \Vert \psi_h \Vert_{L^2(M)}^2 + \mathcal{E}_{h,R,\epsilon,\rho},
\end{align*}
where the remainder term $\mathcal{E}_{h,R,\epsilon,\rho}$ term satisfies
\begin{equation}
\label{e:error_for_m_3}
\mathcal{E}_{h,R,\epsilon,\rho} =  \mathcal{O}_{R,\epsilon,\delta,\rho}(h^{1/2}) + \mathcal{O}\left( \frac{\epsilon}{\delta^{1/2}}\right)+  \mathcal{O}(\delta^{1/2}) +  \mathcal{O}\left( \left( \frac{\rho}{R} \right)^{1/2} \right).
\end{equation}
This motivates the introduction of a new distribution acting on test functions of the form 
$$
\mathbf{a} \in \mathcal{C}_c^\infty(\T^2_{y,z} \times \R_\xi \times \R_\eta \times \R_\sigma \times \R_\varsigma)
$$ 
by
\begin{align}
\label{e:effective_distribution_3}
\mathcal{I}_{h,R,\epsilon,\rho}^{3}(\mathbf{a}) & \\[0.2cm]
 & \hspace*{-1cm} := \int_{\R^3 \times \R^3 \times \R^3}  \bm{\chi}_{h,R,\epsilon,\rho}(\Xi) \, \mathbf{a} \circ \vartheta_h\left( \frac{X + X'}{2},\Xi \right) \Psi_h(X') \overline{\Psi}_h(X) e^{i \Xi \cdot(X - X' )} \text{d}_{h,\delta}(X',\Xi,X) . \notag
\end{align}
Then, modulo the extraction of succesive subsequences, we obtain, for each $\jmath \in \{0,1\}$, the existence of a Radon measure $m_3^{\jmath} \in \mathcal{M}_+(\T^2_{y,z} \times \R_\xi \times \R_\eta \times \R_{\varsigma})$ such that
\begin{align*}
\lim_{\rho \to + \infty} \lim_{R \to + \infty}  \lim_{\epsilon \to 0} \lim_{h \to 0} I^3_{h,R,\epsilon,\rho}(a) & = \lim_{\rho \to + \infty} \lim_{R \to + \infty}  \lim_{\epsilon \to 0} \lim_{h \to 0} \sum_{\jmath \in \{0,1 \}} \mathcal{I}_{h,R,\epsilon,\rho}^{3}(\mathbf{a}_{\jmath}) \\[0.2cm]
 & = \sum_{\jmath \in \{0,1\}} \int_{\T^2_{y,z} \times \R_\xi \times \R_\eta \times \R_\sigma \times \R_{\varsigma} }\mathbf{a}_{\jmath}(y,z,\xi,\eta,\sigma,\varsigma) dm_3^{\jmath}.
\end{align*}
Finally, the positivity of the measure $m_3^{\jmath}$ follows by similar arguments (now in the scalar case) as the ones of the proof of Lemmas \ref{l:measure_m_1} and \ref{l:measure_m_2}. We omit the details.
\end{proof}

\begin{prop}
\label{l:measure_m_3}
There exists a positive Radon measure 
$$
m_4 \in \mathcal{M}_+(M\setminus \mathscr{S} \times \R_w \times \R_\xi)
$$ 
such that, modulo the extraction of succesive subsequences,
\begin{equation}
\label{e:limit_4}
\begin{array}{rl}
\displaystyle  \lim_{\rho \to + \infty} \lim_{R \to + \infty} \lim_{\delta \to 0 }  \lim_{\epsilon \to 0}   \lim_{h \to 0} I_{h,R,\epsilon,\delta,\rho}^4(a) & \\[0.5cm]
 & \hspace*{-3.5cm} \displaystyle  =  \int_{M\setminus \mathscr{S} \times \R_w \times \R_\xi} a \left( x,y,z,w,\xi,0,0 \right) \, dm_4.
 \end{array}
 \end{equation}
\end{prop}

\begin{proof}
Let us consider again the change of variables given by \eqref{e:symplectic_change_of_coordinates_delta_iota} and \eqref{e:unitary_transformation_delta_iota}, and denote:
\begin{equation}
\label{e:change_delta}
\mathbf{A}_{h,R,\epsilon,\delta}^\rho := \mathbf{a}_{h,R,\epsilon,\delta}^\rho \circ \bm{\kappa}_{h,\delta}.
\end{equation}
We have:
\begin{equation}
\label{e:Wigner_distribution_4}
I_{h,R,\epsilon,\delta,\rho}^4(a) =  \int_{\R^3 \times \R^3 \times \R^3} \mathbf{A}_{h,R,\epsilon,\delta,\rho} \left( \frac{X + X'}{2},h\Xi \right) \Psi_{h,\delta}(X') \overline{\Psi}_{h,\delta}(X) e^{i \Xi \cdot(X - X' )} \text{d}_{h,\delta}(X',\Xi,X).
\end{equation}
Then, applying Lemma \ref{l:Calderon_Vaillancourt} and using the localization properties of the cut-off
\begin{equation}
\label{e:cut_off_for_4}
\bm{\chi}_{h,R,\epsilon,\delta,\rho} := \breve{\chi}\left( \frac{h^{1/2} \eta}{\rho} \right) \chi \left( \frac{2h^{2+1/2}x \zeta}{\delta^{3/2}} \right) \chi \left( \frac{h^{2+1/2}\zeta}{\epsilon} \right) \breve{\chi} \left( \frac{ h^{1/2}\zeta }{R} \right)
\end{equation}
we get:
\begin{align*}
\big \vert I_{h,R,\epsilon,\delta,\rho}^4(a) \big \vert & \lesssim \sum_{\alpha \in \mathcal{N}} \big \Vert \partial^\alpha \mathbf{A}_{h,R,\epsilon,\delta}^\rho \big  \Vert_{L^\infty(\R_{X} \times \R_\Xi )} \Vert \psi_h \Vert_{L^2(M)}^2 \\[0.2cm]
& \lesssim \big \Vert a \left( \cdot_x, \cdot_y ,\cdot_z, \cdot_w, \cdot_\xi, 0,0 \right)  \big \Vert_{L^\infty}  \Vert \psi_h \Vert_{L^2(M)}^2 + \mathcal{E}_{h,R,\epsilon,\delta,\rho},
\end{align*}
where the involved constants depend on the $L^\infty$ norm of a finite number of derivatives of $a$ and dimensional factors, and the remainder term $\mathcal{E}_{h,R,\epsilon,\delta,\rho}$ satisfies
$$
\mathcal{E}_{h,R,\epsilon,\delta,\rho} = \mathcal{O}\left( \frac{\epsilon}{\delta^{3/2}} \right) + \mathcal{O}_{R,\epsilon,\delta,\rho}(h^{1/2}) + \mathcal{O}(\delta^{1/2}).
$$
Moreover, the cut-off
$$
 \breve{\chi}\left( \frac{h^{1/2}\eta }{\rho} \right)
$$
together with the compact support property of $a$ in the variable $w$ localizes the integral in the region
\begin{equation}
\label{e:x_away_from_crit}
\frac{h \vert x \vert}{\delta^{1/2}} \gtrsim \left( \frac{\rho}{R} \right)^{1/2}.
\end{equation}
Then, modulo extraction of succesive subsequences, we obtain the existence of a Radon measure $m_4 \in \mathcal{M}(M \setminus \mathscr{S} \times  \T_y \times \T_z \times \R_w \times \R_\xi)$ such that \eqref{e:limit_4} holds. Finally, the positivity of the measure $\mu_4$ follows by similar arguments (now in the scalar case) as the ones of the proof of Propositions \ref{l:measure_m_1} and \ref{l:measure_m_2}.
\end{proof}

To conclude the proof of Theorem \ref{t:existence}, it only remains to show \eqref{e:projection_nu_infty}. To do that, let us take $b \in \mathcal{C}_c^\infty(M)$. Then, consider a sequence $(\psi_h)$ of solutions to \eqref{e:semiclassical_eigenvalue_problem} and split:
\begin{align*}
\big \langle b \, \psi_h, \psi_h \big \rangle_{L^2(M)} & = \left \langle \Op^{\w}_h \left( \chi\left( \frac{\zeta}{R}\right)  \cdot b(x,y,z) \right)  \psi_h, \psi_h \right \rangle_{L^2(M)} \\[0.2cm]
 & \quad +  \left \langle \Op^{\w}_h \left(  \breve{\chi}\left( \frac{ \zeta}{R} \right) \cdot b(x,y,z)  \right)  \psi_h, \psi_h \right \rangle_{L^2(M)}  \\[0.2cm]
 & =: I_{h,R}^0(b) + I_{h,R}^{\infty}(b).
\end{align*}
On the one hand, modulo the extraction of a subsequence, we have:
$$
\lim_{h \to 0 } \big \langle b \, \psi_h, \psi_h \big \rangle_{L^2(M)} = \int_{M} b \, d\nu.
$$
On the other hand, by using that $(\psi_h)$ is $(h^2 \Delta_{\mathcal{M}})$-oscillating and the semiclassical symbolic calculus, we get:
$$
\lim_{R \to + \infty} \lim_{h \to 0} I_{h,R}^0(b) = \int_M b \, \pi_* d  \mu,
$$
where $\mu$ is the semiclassical measure of the sequence $(\psi_h)$ given by \eqref{e:semiclassical_measure}.
Finally, by using \eqref{e:h_oscillation_1}, \eqref{e:h_oscillation_2}, and \eqref{e:h_oscillation_3}, the semiclassical pseudodifferential calculus together with estimates \eqref{e:R-S_1}, \eqref{e:R-S_2}, and \eqref{e:R-S_3}, we obtain, by Propositions \ref{l:measure_m_1}, \ref{l:measure_m_2}, \ref{l:measure_m_4}, and \ref{l:measure_m_3},
\begin{align*}
\lim_{R \to 0}  \lim_{h \to 0} I_{h,R}^\infty(b) &  =  \int_{\mathscr{S} \times \R_\eta \times \R^*_\zeta} b \,  \operatorname{Tr} dM_1  +  \int_{ M \setminus \mathscr{S} \times \R^*_\sigma} b \, \operatorname{Tr}  dM_2 \\[0.2cm]
&  \quad + \sum_{\jmath \in \{0,1 \}} \int_{\mathscr{S} \times \R_\xi \times \R_\eta \times \R^*_\sigma \times \R_\varsigma} b \, dm_3^{\jmath}  + \int_{M \setminus \mathscr{S} \times \R_w \times \R_\xi} b \, dm_4.
\end{align*}
Thus we get:
$$
\nu_\infty =  \int_{\R_\eta \times \R^*_\zeta} \operatorname{Tr} dM_1  +  \int_{ \R^*_\sigma} \operatorname{Tr}  dM_2 + \sum_{\jmath \in \{0,1 \}} \int_{\R_\xi \times \R_\eta \times \R^*_\sigma \times \R_\varsigma} dm_3^{\jmath}  + \int_{\R_w \times \R_\xi} dm_4.
$$ 
This concludes the proof of Theorem \ref{t:existence}.

\section{High oscillation invariance}
\label{s:high_oscillation}
This section is devoted to prove Theorem \ref{t:oscillation_invariance}. First we outline the main idea. We will start the proof from the Wigner equation
\begin{equation}
\label{e:Wigner_general_equation}
0 = \big \langle\big[ h^2 \Delta_{\mathcal{M}} , \Op_h^{\w}( \mathbf{a}_{h,R}) \big] \psi_h, \psi_h \big \rangle_{L^2(M)}, 
\end{equation}
and we will restrict this equation to each of the regimes that give rise to the measures $M_1$, $M_2$, $m_3^\jmath$ and $m_4$. Recall that
$$
h^2 \Delta_{\mathcal{M}} = \Op_h^{\w}(H_{\mathcal{M}}), \quad H_{\mathcal{M}} =  \xi^2 + (\eta + x^2 \zeta)^2,
$$  
hence, by the semiclassical pseudodifferential calculus, we get the commutator formula
\begin{align}
\label{e:commutator_to_poisson}
\big[ h^2 \Delta_{\mathcal{M}} , \Op_h^{\w}( \mathbf{a}_{h,R}) \big] & = \Op_h^{\T^2} \big( [\Op_h^\R(H_{\mathcal{M}}), \Op_h^\R(\mathbf{a}_{h,R})]_{L^2(\R_x)} \big)  \\[0.2cm]
 & \quad +\frac{h}{i} \Op_h^{\w} \big( \partial_\eta H_{\mathcal{M}} \partial_y \mathbf{a}_{h,R} \big) \notag \\[0.2cm]
 & \quad + \frac{h}{i} \Op_h^\w \big( \partial_\zeta H_{\mathcal{M}} \partial_z \mathbf{a}_{h,R} \big) . \notag
\end{align}
Then, by localizing separately \eqref{e:commutator_to_poisson} at the different regimes established by the decomposition  $\mathbf{a}_{h,R} = \mathbf{a}_{h,R}^{\epsilon} +  \mathbf{a}_{h,R,\epsilon}^{\delta,\rho} + \mathbf{a}_{h,R,\epsilon,\delta}^\rho + \mathbf{a}_{h,R,\epsilon,\rho}$, we will get the different invariance properties \eqref{e:invariance_1}, \eqref{e:invariance_2}, \eqref{e:invariance_3}, and \eqref{e:invariance_4}, respectively for each of the measures $M_1$, $M_2$, $m_3^\jmath$ and $m_4$. 

\begin{proof}[Proof of Theorem \ref{t:oscillation_invariance}]
We first show \eqref{e:invariance_1}. We consider in this case the equation
\begin{equation}
\label{e:Wigner_1}
\big \langle \big[ h^2 \Delta_{\mathcal{M}} , \Op_h^{\w}( \mathbf{a}_{h,R}^{\epsilon}) \big] \psi_h, \psi_h \big \rangle_{L^2(M)} = 0.
\end{equation}
Then, using commutator formula \eqref{e:commutator_to_poisson} and following the lines of the proof of Proposition \ref{l:measure_m_1}, replacing $\Op_h^{\w}( \mathbf{a}_{h,R}^{\epsilon})$ by the commutator $\big[ h^2 \Delta_{\mathcal{M}} , \Op^\w_h( \mathbf{a}_{h,R}^{\epsilon}) \big]$, yields to:
\begin{align}
\label{e:commutator_in_trace_1}
0 & = \big \langle \big[ h^2 \Delta_{\mathcal{M}} , \Op^\w_h( \mathbf{a}_{h,R}^{\epsilon}) \big] \psi_h, \psi_h \big \rangle_{L^2(M)}   \\[0.2cm]
 & =   \int_{\R^6} \operatorname{Tr} \left[ \left[ \widehat{H}_{\Theta_h}, \mathcal{A}\left( \frac{Y_h + Y'_h}{2},\Theta_h \right) \right]_{L^2(\R_x)}\mathcal{T}_{\epsilon}(\Theta_h)  \mathcal{K}_h(Y,Y') \right]   e^{i \Theta \cdot (Y- Y')} \text{d}_h(Y',\Theta,Y) \notag  \\[0.2cm]
 &  \quad + \mathcal{O}_{\epsilon}(h) \notag,
\end{align}
where $\Theta_h = (h^{1/2}\eta,h^{2+1/2}\zeta)$. Passing to the limit through the corresponding subsequences we obtain:
\begin{align*}
  \int_{\T^2_{y,z} \times \R_\eta \times \R_\zeta} \operatorname{Tr} \Big( [ \widehat{H}_{\eta,\zeta}, \mathcal{A}]_{L^2(\R_x)} dM_1 \Big) = 0.
\end{align*}

We next show \eqref{e:invariance_2}. We consider in this case the Wigner equation:
\begin{equation}
\label{e:Wigner_commutator_2}
\big \langle \big[ -h^2 \Delta_{\mathcal{M}} , \Op_h^{M}( \mathbf{a}_{h,R,\epsilon}^\delta) \big] \psi_h, \psi_h \big \rangle_{L^2(M)} = 0.
\end{equation}
Denote, for $(\eta,\zeta) \in \R^* \times \R^*$,
$$
H_{\eta,\zeta}(x,\xi) := H_{\mathcal{M}}(x,\xi,\eta,\zeta) = \xi^2 + (\eta + x^2 \zeta)^2.
$$
Then, by the Weyl semiclassical pseudodifferential calculus (see for instance \cite[Thm. 4.11]{Zworski12}), we have:
\begin{align}
\label{e:canonical_pseudodifferential}
 \big[ \widehat{H}_{\Theta_h}, \Op_{x,\xi}^\R( \mathbf{A}_{h,R,\epsilon}^\delta) \big]_{L^2(\R_x)} & = 2 \sum_{n=0}^1  \left( \frac{ 1}{2i} \right)^{2n+1} \Op_{x,\xi}^\R \Big( \{ H_{\Theta_h},  \mathbf{A}_{h,R, \epsilon}^\delta \}_{2n+1} \Big), 
\end{align}
where the Moyal bracket $\{ \cdot , \cdot \}_{2n+1}$ is given by
$$
\{ \mathbf{a}, \mathbf{b} \}_{2n+1} = \sum_{\alpha + \beta = 2n+1} \frac{ (-1)^\alpha}{\alpha! \beta!} \partial_x^\alpha \partial_\xi^\beta \mathbf{a} \, \partial_\xi^\alpha \partial_x^\beta \mathbf{b}.
$$
The only non-vanishing term appearing in the sum of \eqref{e:canonical_pseudodifferential} with $n = 1$ is
$$
\partial_x^3 H_{\Theta_h} \partial_\xi^3 \mathbf{A}_{h,R,\epsilon}^\delta = 6 h^{2+1/2} \zeta \cdot h\sigma(hx, h^{1/2}\zeta)  \partial_\xi^3 \mathbf{A}_{h,R,\epsilon}^\delta.
$$
Since $h \sigma$ is bounded in this regime (by the support properties of $a$) and $\vert h^{2+1/2}\zeta \vert \leq \epsilon$ (by the localization of the cut-offs), the key observation is now that the only non-negligible contribution of the symbol of the commutator \eqref{e:canonical_pseudodifferential} in the Wigner equation \eqref{e:Wigner_commutator_2} is the Poisson bracket
$$
\{ H_{\Theta_h}, \mathbf{A}_{h,R,\epsilon}^\delta \}_{x,\xi} := 2 \big( \xi \partial_x - \sigma w \partial_\xi \big) \mathbf{A}_{h,R,\epsilon}^\delta
$$
with the rest ot terms being of lower order by the localization of the cut-offs. Moreover:
$$
\big \{ H_{\Theta_h}, \mathbf{A}_{h,R,\epsilon}^\delta \big \}_{x,\xi} =  \Big( \{ G_{h\bm{\sigma}} , \mathbf{A}_{h,R,\epsilon}^\delta \circ \vartheta^{-1} \}_{\mathbf{w},\upsilon} \circ \vartheta \Big)  \left( 1 + \mathcal{O}\left( \frac{\epsilon}{\delta^2} \right) \right),
$$
where $G_{\bm{\sigma}}(\mathbf{w},\upsilon) = \bm{\sigma}^2 \upsilon^2 + \mathbf{w}^2$. Since $G_{\bm{\sigma}}(\mathbf{w},\upsilon)$ is a quadratic Hamiltonian, we also have
$$
\big[ \widehat{G}_{h\bm{\sigma}},  \Op_{\mathbf{w},\upsilon}^\R ( \mathbf{A}_{h,R,\epsilon}^\delta \circ \vartheta^{-1}) \big] = \frac{1}{i} \Op_{\mathbf{w},\upsilon}^\R \big( \{ G_{h\bm{\sigma}} , \mathbf{A}_{h,R,\epsilon}^\delta \circ \vartheta^{-1} \}_{\mathbf{w},\upsilon} \big).
$$
Then, using \eqref{e:final_Trace_formula_2} together with \eqref{e:commutator_to_poisson} gives that:
\begin{align*}
0 & =\big \langle \big[ -h^2 \Delta_{\mathcal{M}} , \Op^\w_h( \mathbf{a}_{h,R,\epsilon}^\delta) \big] \psi_h, \psi_h \big \rangle_{L^2(M)}  \\[0.2cm]
 &  =  \int_{\R^2_Y \times \R_\eta \times \R^*_{\bm{\sigma}}}  \operatorname{Tr} \Big(  \big[ \widehat{G}_{h \bm{\sigma}}, \mathcal{A}_\infty \mathcal{T}_{\delta}(h \bm{\sigma}) \big]_{L^2(\R_{\mathbf{w}})} \mathcal{K}_{\mathcal{H}_1 \otimes \mathcal{H}_2}^h(Y,\eta,\bm{\sigma}) \Big) \bm{\chi}_{h,R,\epsilon}(\eta,\bm{\sigma}) \text{d}_h \\[0.2cm]
 &  \quad + \mathcal{O}_{R,\epsilon,\delta}(h) + \mathcal{E}_{h,R,\epsilon,\delta}^3,
\end{align*}
where $\mathcal{E}_{h,R,\epsilon,\delta}^3$ is of the form \eqref{e:Error_3}.
Taking limits through succesive subsequences we get:
\begin{align*}
 \int_{\T^2_{y,z} \times \R^*_\eta \times \R^*_{\bm{\sigma}}} \operatorname{Tr} \Big( \big[ \widehat{G}_{\bm{\sigma}}, \mathcal{A}_\infty(y,z,\eta,\bm{\sigma}) \big]_{L^2(\R_{\mathbf{w}})} d\mathbf{m}_2 \Big) &  \\[0.2cm]
 & \hspace*{-5cm} =  \int_{M_{\mathbf{x},y,z} \setminus \mathscr{S} \times \R^*_{\bm{\sigma}}}  \operatorname{Tr} \Big( \big[ \widehat{G}_{\bm{\sigma}}, \Op_{\mathbf{w},\upsilon}^\R\big( a \big(\mathbf{x},y,z,\mathbf{w}, \bm{\sigma}\upsilon, \bm{\sigma},0 \big) \big) \big]_{L^2(\R_{\mathbf{w}})} dM_2 \Big) \\[0.2cm]
 & \hspace*{-5cm} = 0.
\end{align*}

We next show \eqref{e:invariance_3}. We start in this case from the Wigner equation
$$
\big \langle \big[ -h^2 \Delta_{\mathcal{M}} , \Op_h^{\w}( \mathbf{a}_{h,R,\epsilon,\rho}) \big] \psi_h, \psi_h \big \rangle_{L^2(M)} = 0.
$$
It is important in this case to replace $a$ by $a/\sqrt{\vert \zeta \vert}$ in the definition of $\mathbf{a}_{h,R}$. Indeed, notice, in view of \eqref{e:change_to_varsigma} and \eqref{e:a_jmath}, that the main contribution of the Poisson bracket $\{ H_{\eta,\zeta}, \mathbf{a}_{h,R,\epsilon,\rho} \}_{x,\xi} \circ \bm{\kappa}_{h,\delta}$ is the term
$$
  (2x\zeta)  \Big( 2\xi \partial_w a - 2 w  \partial_\xi a \Big)  \circ \bm{\kappa}_{h,\delta} =  \sqrt{\vert \zeta \vert} \{ H_\eta^\jmath, \mathbf{a}_\jmath \}_{\varsigma,\xi }  \circ \vartheta_{h,\delta},
$$
where $\mathbf{a}_\jmath$ is given by \eqref{e:a_jmath} and $\vartheta_{h,\delta}$ is given by \eqref{e:vartheta_h_delta}. Using \eqref{e:commutator_to_poisson} and in view of \eqref{e:effective_distribution_3}, we get in this case:
\begin{align*}
0 & = \sum_{\jmath \in \{0,1\}} \int_{\R^3 \times \R^3 \times \R^3}  \bm{\chi}_{h,R,\epsilon,\delta,\rho}(\Xi) \, \{ H_\eta^\jmath,  \mathbf{a}_\jmath \}_{\varsigma,\xi} \circ \vartheta_{h,\delta}\left( \frac{X+X'}{2},\Xi \right) \Psi_{h,\delta}(X') \overline{\Psi}_{h,\delta}(X) e^{i \Xi \cdot(X - X' )} \text{d}_{h,\delta} \\[0.2cm]
 & \quad + \mathcal{O}_{R,\epsilon,\rho}(h) +\mathcal{O}_\delta(\epsilon)  + \mathcal{O}\left( \left( \frac{\rho}{R} \right)^{1/2} \right).
\end{align*}
Passing to the limit we get:
$$
0 = \int_{\T^2_{y,z} \times \R_\xi \times \R_\eta \times \R^*_\sigma \times \R_\varsigma} \{ H_\eta^\jmath, \mathbf{a}_\jmath \}_{\varsigma,\xi} dm_3^\jmath = 0.
$$
Finally, to show \eqref{e:invariance_4}, we start from the equation
$$
\big \langle \big[ -h^2 \Delta_{\mathcal{M}} , \Op_h^{M}( \mathbf{a}_{h,R,\epsilon,\delta}^\rho) \big] \psi_h, \psi_h \big \rangle_{L^2(M)} = 0,
$$
where we replace $a$ by $a/\sigma$ in the definition of $\mathbf{a}_{h,R}$. Using again \eqref{e:commutator_to_poisson}, noting that the main contribution of the Poisson bracket $\{ H_{\eta,\zeta}, \mathbf{a}_{h,R,\epsilon,\delta}^\rho \} \circ \bm{\kappa}_{h,\delta}$ is in this case
$$
\sigma ( 2\xi \partial_w a - 2w \partial_w a) \circ \bm{\kappa}_{h,\delta} = \sigma \{ G, a \}_{w,\xi} \circ \bm{\kappa}_{h,\delta}
$$ 
and taking into account \eqref{e:Wigner_distribution_4}, we obtain in this case
\begin{align*}
0 & =   \int_{\R^3 \times \R^3 \times \R^3}  \bm{\chi}_{h,R,\epsilon,\rho} \, \{ G,  a \}_{w,\xi}\circ \bm{\kappa}_{h,\delta} \left( \frac{X+X'}{2}, h\Xi \right)  \Psi_{h,\delta}(X') \overline{\Psi}_{h,\delta}(X) e^{i \Xi \cdot(X - X' )} \text{d}_{h,\delta} \\[0.2cm]
 & \quad + \mathcal{O}_{R,\epsilon,\delta,\rho}(h) +\mathcal{O}_\delta(\epsilon)  + \mathcal{O}(\delta^{3/2}),
\end{align*}
where $\bm{\chi}_{h,R,\epsilon,\delta,\rho}$ is given by \eqref{e:cut_off_for_4}. We thus get:
$$
0 = \int_{M_{x,y,z} \setminus \mathscr{S} \times \R_w \times \R_\xi} \{G, a \}_{w,\xi} (x,y,z,w,\xi,0,0) dm_4.
$$
This concludes the proof.
\end{proof}

\section{Concentration properties}
\label{s:concentration}

This section is devoted to prove Theorem \ref{t:concentration}. The main idea consists of localizing the phase-space distribution of the sequence $(\psi_h)$ near the level set $h^2 \Delta_{\mathcal{M}} \sim 1$. We will refine the $h$-oscillation properties established by \ref{p:h_oscillation} by consider the different localizations of the cut-offs giving rise to the measures $M_1$, $M_2$, $m_3^\jmath$ and $m_4$.

\begin{proof}[Proof of Theorem \ref{t:concentration}]
\noindent We first show \eqref{e:concentration_1}. On the one hand, by using \eqref{e:semiclassical_eigenvalue_problem}  we have:
\begin{align*}
\bigg \langle \breve{\chi}\left( \frac{h^3 D_z }{\epsilon} \right) \breve{\chi} \left( \frac{hD_z}{R} \right)  \chi \big( h^2 \Delta_{\mathcal{M}} - 1 \big) \psi_h, \psi_h \bigg \rangle_{L^2(M)} & \\[0.2cm]
 & \hspace*{-5cm} = \left \langle  \breve{\chi}\left( \frac{h^3 D_z }{\epsilon} \right) \breve{\chi} \left( \frac{hD_z}{R} \right) \psi_h, \psi_h \right \rangle_{L^2(M)}.
\end{align*}
On the other hand, by the fact that the operator $\Delta_{\mathcal{M}}$ is a Fourier multyplier in $D_z$ and $D_y$, we have the exact functional calculus formula
\begin{equation}
\label{e:exact_functionl_calculus}
\chi \big( h^2 \Delta_{\mathcal{M}} - 1 \big) = \Op_h^{\T^2} \Big( \chi \big( \Op_h^{\R}(H_{\mathcal{M}})  - 1 \big) \Big).
\end{equation}
Then, we obtain
\begin{align*}
\label{e:commutator_in_trace_1}
0 & =   \int_{\R^6} \operatorname{Tr} \Big[ \Big( \operatorname{Id} - \chi( \widehat{H}_{\Theta_h} - 1)  \Big) \mathcal{K}_h(Y,Y')  \Big]  \mathcal{T}_{h,R,\epsilon}(\Theta) e^{i \Theta \cdot (Y- Y')} \text{d}_h(Y',\Theta,Y),
\end{align*}
where $ \mathcal{T}_{h,R,\epsilon}(\Theta)$ is given by \eqref{e:cut_off_1_operator}. Thus, taking limits we get:
\begin{align*}
  \int_{\T^2_{y,z} \times \R_\eta \times \R_\zeta^*}  \operatorname{Tr} \Big( \chi \big( \widehat{H}_{\eta,\zeta}  - 1 \big)   dM_1(y,z,\eta,\zeta) \Big) & \\[0.2cm]
  & \hspace*{-3cm} =  \int _{\T^2_{y,z} \times \R_\eta \times \R_\zeta^*}  \operatorname{Tr}  dM_1(y,z,\eta,\zeta).
\end{align*}
Finally, let $\{ \varphi_k(\eta,\zeta) \}_{k \in \mathbb{N}}$ be the orthonormal basis of $L^2(\R_x)$ consisting of eigenfunctions of the operator $\widehat{H}_{\eta,\zeta}$ with eigenvalues $\lambda_k(\eta,\zeta)$. By using \eqref{e:eigenvalue_measure_1}, we obtain
\begin{align*}
\sum_{k \in \mathbb{N}} \int_{\T^2_{y,z} \times \R_\eta \times \R_\zeta^*} \chi \big( \lambda_k(\eta,\zeta) - 1 \big)  d \mu_{1,k}(y,z,\eta,\zeta) & \\[0.2cm]
 & \hspace*{-4cm} =  \sum_{k \in \mathbb{N}} \int _{\T^2_{y,z} \times \R_\eta \times \R_\zeta^*} d\mu_{1,k}(y,z,\eta,\zeta).
\end{align*}
Property \eqref{e:concentration_1} then follows since each $\mu_{1,k}$ is a positive measure and we can adjust the cut-off function $\chi$ to be supported arbitrarily close to zero.
\medskip

We next prove \eqref{e:concentration_2}. In this case, we use a similar strategy to that of \cite[\S 4]{Ar_Riv24} instead of the functional calculus used in \eqref{e:exact_functionl_calculus}. Alternatively, we observe that:
\begin{equation}
\label{e:AA_form_G}
 \Delta_{\mathcal{M}}  =  ( - i \mathcal{X}_1 -  \mathcal{X}_2)^*( - i \mathcal{X}_1 -  \mathcal{X}_2) -  i \mathcal{X}_3 =: \mathcal{D}^* \mathcal{D} - i \mathcal{X}_3.
\end{equation}
From this observation, we start again using \eqref{e:semiclassical_eigenvalue_problem}, so that
\begin{equation}
\label{e:support_0_2}
0 = \bigg \langle  \breve{\chi} \left( \frac{2h^2 x D_z}{\delta} \right)  \chi\left( \frac{h^2 D_z }{\epsilon} \right) \breve{\chi} \left( \frac{hD_z}{R} \right)  \big( h^2 \Delta_{\mathcal{M}} - 1 \big) \psi_h, \psi_h \bigg \rangle_{L^2(M)}.
\end{equation}
The idea is now to use the factorization \eqref{e:AA_form_G} and take the commutator of $h^2 \Delta_{\mathcal{M}}$ with the cut-offs. Precisely, notice that 
\begin{align}
\label{e:commutator_with_sigma}
\left[  \breve{\chi} \left( \frac{2h^2 x D_z}{\delta} \right), h( -i\mathcal{X}_1 \mp \mathcal{X}_2) \right] & =  \pm \frac{2 i h^3 D_z}{\delta} \cdot  \chi' \left( \frac{2h^2 x D_z}{\delta} \right),
\end{align}
which produces a negligible term of size $\mathcal{O}(\epsilon/\delta)$ in this regime. Then, by using identity \eqref{e:AA_form_G},  \eqref{e:commutator_with_sigma}, estimates  \eqref{e:R-S_1}, \eqref{e:R-S_3}, and following the lines of the proof of Proposition \ref{l:measure_m_2} together with the factorization formula
$$
\widehat{G}_{\bm{\sigma}} = \bm{\sigma}^2D_{\mathbf{w}}^2 + \mathbf{w}^2 =  \mathbf{D}_{\bm{\sigma}}^* \mathbf{D}_{\bm{\sigma}} +  \vert \bm{\sigma} \vert, \quad \mathbf{D}_{\bm{\sigma}} :=  \vert \bm{\sigma} \vert D_{\mathbf{w}} - i \mathbf{w},
$$ 
we get:
\begin{align*}
  \int_{\T^2_{y,z} \times \R_\eta^* \times \R_{\bm{\sigma}}^*} \operatorname{Tr} \Big(  \big( \vert \bm{\sigma} \vert -  1  \big)  dM_2(y,z,\eta,\bm{\sigma}) \Big) & \\[0.2cm]
  & \hspace*{-4cm} =  -  \int _{\T^2_{y,z} \times \R_\eta^* \times \R_{\bm{\sigma}}^*} \operatorname{Tr} d \big( \mathbf{D}^*_{\bm{\sigma}} \mathbf{D}_{\bm{\sigma}} M_2(y,z,\eta,\bm{\sigma})  \big).
\end{align*}
Now, let $\{ \phi_j(\bm{\sigma}) \}_{j \in \mathbb{N}}$ be the orthonormal basis of $L^2(\R_{\mathbf{w}})$ consisting of eigenfunctions of the operator $\widehat{G}_{\bm{\sigma}}$ with eigenvalues $\nu_j(\bm{\sigma}) = \vert \bm{\sigma} \vert (2j +1)$. Then, by using \eqref{e:eigenvalue_measure_2} and taking the trace we get:
\begin{align*}
\sum_{j = 0}^\infty \int_{\T^2_{y,z} \times \R_\eta^* \times \R_{\bm{\sigma}}^*} ( \vert \bm{\sigma} \vert - 1) d\mu_{2,j} = - \sum_{j=1}^\infty   \int_{\T^2_{y,z} \times \R_\eta^* \times \R_{\bm{\sigma}}^*} 2 j \vert \bm{\sigma} \vert d\mu_{2,j}
\end{align*}
Thus:
\begin{align*}
\supp \mu_{2,j} \subset \{ \vert \bm{\sigma} \vert (2j+1) = 1 \}, \quad j \geq 0.
\end{align*}


The proofs of \eqref{e:concentration_3} and \eqref{e:concentration_4} are similar, now in the scalar case. We start respectively from equations:
\begin{align*}
0 & = \big \langle \Op_h^{\w} (\mathbf{a}_{h,R,\epsilon,\rho}) \big( \chi(h^2 \Delta_{\mathcal{M}} - 1) - 1 \big) \psi_h, \psi_h \big \rangle_{L^2(M)}, \\[0.2cm]
0 & = \big \langle \Op_h^{\w} (\mathbf{a}_{h,R,\epsilon,\delta}^\rho) \big( \chi(h^2 \Delta_{\mathcal{M}} - 1) - 1 \big) \psi_h, \psi_h \big \rangle_{L^2(M)},
\end{align*} 
and use again \eqref{e:exact_functionl_calculus}. Then, by the functional pseudodifferential calculus and taking limits we obtain both claims. We omit the details.
\end{proof}

\section{Drift invariance}
\label{s:Drift}
In this section, we prove Theorem \ref{t:drift_invariance}. The idea is to consider again the Wigner equation \eqref{e:Wigner_general_equation} and the commutator formula \eqref{e:commutator_to_poisson} and look at the sub-principal scale to obtain another invariance. To capture this sub-principal scale, it is needed to put in place an averaging process or normal form procedure, in order to neutralize the high oscillation described in Section \ref{s:high_oscillation}. This averaging method will be significantly different for the quartic oscillator than for the quadratic oscillator.

\subsection{Averaging by anharmonic quartic oscillators} 

The invariance properties \eqref{e:invariance_1} and \eqref{e:invariance_3}, stated respectively for $M_1$ and $m_3^{\jmath}$, are given by means of the anharmonic quartic oscillators $\widehat{H}_{\eta,\zeta}$ and $H_\eta^{\jmath}$. Then, when studying the Wigner equation \eqref{e:Wigner_general_equation}, and in view of \eqref{e:commutator_to_poisson}, at the sub-principal scale we will find a new invariance property given roughly by the term
$$
\partial_\eta H_{\mathcal{M}} \partial_y + \partial_\zeta H_{\mathcal{M}} \partial_z = 2w\partial_y + 2x^2w\partial_z
$$
coming from \eqref{e:commutator_to_poisson}. Then, the invariance properties \eqref{e:invariance_1} and \eqref{e:invariance_3} impose the average of the coefficient $w$ by these anharmonic oscillator, which, in contrast with the harmonic osillator, does not vanish. This will produce a new invariance with respect to the flow $\partial_y$ which will not appear for $M_2$ and $m_4$, which are governed by the (quadratic) harmonic oscillator, and which is behind the phenomenon described by \cite{CdV_Letrouit22}.

\begin{lemma}
\label{l:averaging_anharmonic} 
For each  $k \in \mathbb{N}$:
\begin{align*}
  \big \langle  w \,  M_1 \, \varphi_k, \varphi_k \big \rangle_{L^2(\R_x)}  = \mu_{1,k}\,   \partial_\eta \lambda_k.
\end{align*}
\end{lemma}

\begin{proof} By using \eqref{e:eigenvalue_measure_1}, we have:
\begin{align*}
\big \langle w \, M_1 \, \varphi_k,  \varphi_k \big \rangle_{L^2(\R_x)} &  = \big \langle w \,   \mu_{1,k} \, \varphi_k , \varphi_k \big \rangle_{L^2(\R_x)} \\[0.2cm]
 & =  \mu_{1,k}  \big \langle w \, \varphi_k, \varphi_k \big \rangle_{L^2(\R_x)}  \\[0.2cm]
 & =\mu_{1,k}   \big \langle \partial_\eta \widehat{H}_{\eta,\zeta} \, \varphi_k , \varphi_k\big \rangle_{L^2(\R_x)} \\[0.2cm]
 &  =  \mu_{1,k} \, \partial_\eta  \lambda_k,
\end{align*}
where to obtain the last equality we derive with respect to $\eta$ the identity
$$
\big \langle (\widehat{H}_{\eta,\zeta} - \lambda_k) \, \varphi_k , \varphi_k \big \rangle_{L^2(\R_x)} = 0.
$$
\end{proof}

\begin{lemma}
\label{l:averaging_elliptic} 
For each $\jmath \in \{0,1\}$ and $\eta \in \R$, the function $\Upsilon_{\jmath}(\eta)$ defined by \eqref{e:elliptic_integral} is given by $\Upsilon_{\jmath}(\eta) = \Upsilon((-1)^{ \jmath} \eta)$, with:
$$
\Upsilon(\eta) = \left \lbrace \begin{array}{ll}
\displaystyle \frac{2E(k)}{K(k)} - 1, \quad k = \sqrt{\frac{1-\eta}{2}}, & \quad \text{if } \; \eta \in (-1,1], \\[0.7cm]
-1, & \quad \text{if } \; \eta = -1, \\[0.5cm]
\displaystyle \frac{2E(k)}{K(k)} + k^2 - 2, \quad  k = \sqrt{\frac{2}{1-\eta}}, &  \quad \text{if } \; \eta < -1,
\end{array} \right.
$$
where $E(k)$ and $K(k)$ are the complete elliptic integrals:
\begin{align*}
K(k) & := \int_0^{\frac{\pi}{2}} \frac{dt}{\sqrt{1 - k^2 \sin^2(t)}}, \\[0.2cm]
E(k) & := \int_0^{\frac{\pi}{2}} \sqrt{ 1 - k^2 \sin^2(t)} \, dt.
\end{align*}

\end{lemma}

\begin{remark}
\label{r:critical_point}
In particular, we have $\Upsilon(1) = 1$, and
$$
\lim_{\eta \to -1^+} \Upsilon(\eta) = -1, \quad \lim_{\eta \to -1^-} \Upsilon(\eta) = - 1.
$$
In particular (see \cite{Byrd_Friedman71}), there exists a unique $\eta_* \in (-1,1)$ such that $\Upsilon(\eta_*) = 0$ and $\Upsilon'(\eta_*) \neq 0$.
\end{remark}

\begin{proof} Assume that $ \jmath = 0$ and denote for simplicity $H_\eta := H_\eta^{0}$,  and $\Upsilon(\eta) := \Upsilon_{0}(\eta)$. The other cases are obtained similarly since $\Upsilon_{\jmath}(\eta) = \Upsilon((-1)^{ \jmath} \eta)$.
The fact that $H_\eta(\varsigma,\xi) = 1$ implies that $\eta \leq 1$. Let us next consider the equation:
\begin{equation}
\label{e:axis_with_energy}
(\eta + \varsigma^2)^2 = 1.
\end{equation}


Assume first that $\eta \in (-1,1]$. In this case, equation \eqref{e:axis_with_energy} has a unique solution $\varsigma_0 \in \R_+$, given by $\varsigma_0 = \displaystyle \sqrt{1- \eta}$. Moreover, the level set $H_\eta^{-1}(1)$ consists of a single periodic orbit. We set:
$$
\mathcal{T}_{\eta} := \arccos( \eta ).
$$ 
Define now, for  $0  \leq t \leq \mathcal{T}_{\eta}$, the (non-Hamiltonian) flow $\bm{\phi}_t^{H_\eta}(\varsigma_0,0) = (\bm{\varsigma}(t), \bm{\xi}(t))$ by:
\begin{equation}
\label{e:pseudoflow}
\begin{array}{rcl}
\displaystyle \bm{\varsigma}(t)  &  \hspace*{-0.2cm} :=  &  \hspace*{-0.2cm} \displaystyle\sqrt{ \cos(t) - \eta }, \\[0.2cm]
\displaystyle \bm{\xi}(t) & \hspace*{-0.2cm} := &  \hspace*{-0.2cm} \displaystyle  - \sin(t).
\end{array}
\end{equation}
We have that $\bm{\phi}_0^{H_\eta}(\varsigma_0,0) = ( \varsigma_0,0)$, and $H_\eta( \bm{\varsigma}(t), \bm{\xi}(t)) = H_\eta(\varsigma_0, 0)$. This means that the flow $(\bm{\varsigma}(t),\bm{\xi}(t))$ travels the Hamiltonian orbit of $H_\eta$. More precisely, in time $\mathcal{T}_\eta$, the flow $\bm{\phi}_t^{H_\eta}$ starting at $(\varsigma_0,0)$ travels one quadrant of the periodic orbit $H_\eta^{-1}(1)$, ending at the intersection with the vertical axis 
$$
( -\sqrt{1 - \eta^2}, 0 ).
$$
Observe next that:
\begin{align*}
\frac{d}{dt} \bm{\varsigma}(t) & =  \frac{ -  \sin(t) }{2 \sqrt{ \cos(t)  - \eta}}  = \frac{ \bm{\xi}(t)}{2 \bm{\varsigma}(t)}, \\[0.2cm]
\frac{d}{dt} \bm{\xi}(t)  & = -  \cos(t)  =   -( \eta + \bm{\varsigma}(t)^2) .
\end{align*}
Then, to reparameterize the flow $\bm{\phi}_t^{H_\eta}$ into a Hamiltonian flow, we define:
$$
\varsigma(t) := \bm{\varsigma}( \tau(t)), \quad \xi(t) := \bm{\xi}(\tau(t)),
$$
and impose that $\tau : [0, T_\eta] \to [0, \mathcal{T}_{\eta}]$, satisfies $\tau(0) = 0$ and
\begin{align*}
\frac{d}{dt} \varsigma(t) & = \frac{\partial H_\eta}{\partial \xi} =  2 \xi(t) = \frac{ \xi(t)}{2  \varsigma(t)} \dot{\tau}(t), \\[0.2cm]
\frac{d}{dt} \xi(t) & = - \frac{\partial H_\eta}{\partial \varsigma} = - 4 \varsigma(t) ( \eta + \varsigma(t)^2) = - (  \eta + \varsigma(t)^2 ) \dot{\tau}(t).
\end{align*}
This implies that
$$
\dot{\tau}(t) = 4 \bm{\varsigma}(\tau(t)).
$$
Now, let $a \in \mathbb{C}^\infty(\R_{\varsigma,\xi}^2)$. Assume for simplicity that $a$ is symmetric with respect to the two canonical orthonormal axis of $\R^2_{\varsigma,\xi}$. We compute the average $\langle a \rangle(\eta)$ of the function $a$ by the Hamiltonian flow $\phi_t^{H_\eta}(\varsigma_0,0) = (\sigma(t),\xi(t))$ on the orbit $H_\eta^{-1}(1)$ by:
\begin{align*}
\langle a \rangle(\eta) & :=\frac{1}{T_{\eta}}  \int_0^{T_{\eta}} a( \varsigma(t), \xi(t)) dt \\[0.2cm]
 & = \frac{1}{T_{\eta}}   \int_0^{\mathcal{T}_{\eta}} a( \bm{\varsigma}(s), \bm{\xi}(s)) \frac{ ds}{\dot{\tau}(\tau^{-1}(s))}  \\[0.2cm]
 & = \frac{1}{T_{\eta}}   \int_0^{\mathcal{T}_{\eta}}  a(  \bm{\varsigma}(s), \bm{\xi}(s)) \frac{ ds}{4 \bm{\varsigma}(s)}.
\end{align*}
By taking $a = 1$, we get the value of $T_{\eta}$ (see \cite[290.00]{Byrd_Friedman71} and \cite[110.06]{Byrd_Friedman71}) as:
\begin{align*}
T_{\eta} & = \int_0^{\mathcal{T}_{\eta}} \frac{ds}{ \bm{\varsigma}(s)} ds \\[0.2cm]
 & = \int_0^{\mathcal{T}_{\eta}} \frac{1}{\sqrt{ \cos(s) -\eta}} ds \\[0.2cm]
 & = \sqrt{2} K( k), \quad k = \sqrt{\frac{1-\eta}{2}}.
\end{align*}
On the other hand, taking $a(\varsigma,\xi) = \eta + \varsigma^2$, we obtain (see \cite[290.01]{Byrd_Friedman71} and \cite[110.07]{Byrd_Friedman71}):
\begin{align*}
\langle a \rangle(\eta) = \Upsilon(\eta) & = \frac{1}{T_{\eta}}  \int_0^{\mathcal{T}_{\eta}}  \frac{ \cos(s)}{ \sqrt{ \cos(s) -\eta}} ds \\[0.2cm]
 & = \frac{ 2 E(k)}{K(k)} - 1.
\end{align*}
In particular, $\Upsilon(1) = 1$, and, by \cite[111.05]{Byrd_Friedman71},
$$
\lim_{\eta \to -1^+} \Upsilon(\eta) = -1.
$$

Let us next consider the case $\eta = -1$. In this case, equation \eqref{e:axis_with_energy} has two positive solutions $\varsigma_0 = 0$ and $\varsigma_0 = \sqrt{2}$. The level set $H_\eta^{-1}(1)$ consists in an equilibrium $(\varsigma,\xi) = (0,0)$, and two homoclinic orbits which are symmetric with respect to the vertical axis $\varsigma = 0$. We have:
$$
\mathcal{T}_\eta = \pi, \quad T_\eta = \int_0^\pi \frac{1}{\sqrt{\cos(s) + 1}} ds = + \infty.
$$
Then we have
$$
\lim_{t \to \pi} (\bm{\varsigma}(t), \bm{\xi}(t)) = (0,0).
$$
Therefore,
$$
\Upsilon(-1) = \lim_{t \to \pi} (-1 + \bm{\varsigma}(t)) = -1.
$$

We finally consider the case $\eta < -1$. In this case, equation \eqref{e:axis_with_energy} has two positive solutions $0 < \varsigma_1 < \varsigma_2$ (and two negative solutions, symmetric with respect to zero). Moreover, the level set $H_\eta^{-1}(1)$ consists of two periodic orbits, symmetric with respect to the vertical axis $\varsigma = 0$. We set in this case  $\mathcal{T}_{\eta} := \pi$, and consider the flow $\bm{\phi}_t^{H_\eta}(\varsigma_2,0) = (\bm{\varsigma}(t), \bm{\xi}(t))$ given by \eqref{e:pseudoflow}. In time $\mathcal{T}_\eta = \pi$, the flow $\bm{\phi}_t^{H_\eta}(\varsigma_2,0)$ travels half the periodic orbit passing through $(\varsigma_2,0)$, starting at this point and ending at $(\varsigma_1,0)$. We have (see \cite[289.00]{Byrd_Friedman71}):
\begin{align*}
T_{\eta} & = \int_0^{\pi} \frac{ds}{ \bm{\varsigma}(s)} ds \\[0.2cm]
 & = \int_0^{\pi} \frac{1}{\sqrt{\cos(s) -\eta}} ds \\[0.2cm]
 & = \frac{2}{\sqrt{1-\eta}} K(k), \quad k = \sqrt{\frac{2}{1-\eta}}.
\end{align*}
Taking $a = \eta + \varsigma^2$, we obtain in this case (see \cite[289.06]{Byrd_Friedman71}):
\begin{align*}
\langle a \rangle(\eta) = \Upsilon(\eta) & = \frac{1}{T_{\eta}}  \int_0^{\pi}  \frac{  \cos(s)}{ \sqrt{\cos(s) -\eta}} ds \\[0.2cm]
  & = \frac{2}{K(k)} \left( E(k) + \frac{\eta}{1 - \eta} K(k) \right) \\[0.2cm]
  & = \frac{2E(k)}{K(k)} + k^2 - 2.
\end{align*}
\end{proof}

We can now prove the invariance properties \eqref{e:drift_1} and \eqref{e:drift_3} of Theorem \ref{t:drift_invariance}.

\begin{proof}[Proof of \eqref{e:drift_1}] 
Let us take a symbol $a \in \mathcal{C}_c^\infty(\T^2_{y,z} \times \R_\eta \times \R_\zeta)$ and define from this symbol
$$
\mathbf{a}_{h,R}^\epsilon := \breve{\chi} \left( \frac{h^2\zeta}{\epsilon} \right) \breve{\chi} \left( \frac{\zeta}{R} \right) a(y,z,\eta,h^2\zeta).
$$
We consider again the Wigner equation:
\begin{equation*}
\big \langle \big[ -h^2 \Delta_{\mathcal{M}} , \Op_h^{\w}( \mathbf{a}_{h,R}^{\epsilon}) \big] \psi_h, \psi_h \big \rangle_{L^2(M)} = 0.
\end{equation*}
Using \eqref{e:commutator_to_poisson} and taking limits we get:
$$
 \int_{\T^2_{y,z} \times \R_\eta \times \R_\zeta}  \operatorname{Tr} \Big( w(x,\eta,\zeta) \partial_y a(y,z,\eta,\zeta) dM_1  \Big) = 0.
$$
Finally, by computing the trace in the above expression in terms of the orthonormal basis $\{ \varphi_k \}_{k \in \mathbb{N}}$ and using Lemma \ref{l:averaging_anharmonic}, we obtain, for each $k \in \mathbb{N}$,
$$
  \int_{\T^2_{y,z} \times \R_\eta \times \R_\zeta} \partial_\eta \lambda_k(\eta,\zeta) \partial_y a(y,z,\eta,\zeta) d\mu_{1,k} = 0.
$$ 
Then \eqref{e:drift_1} holds.
\end{proof}

\begin{proof}[Proof of \eqref{e:drift_3}] Let us consider a function $a \in \mathcal{C}_c^\infty(\T^2_{y,z} \times \R_\eta)$ and define:
$$
\mathbf{a}_{h,R,\epsilon,\rho} :=  \chi \left( \frac{\eta }{\rho} \right) \chi \left( \frac{h^2 \zeta}{\epsilon} \right) \breve{\chi} \left( \frac{\zeta}{R} \right) a(y,z,\eta).
$$
From the Wigner equation 
\begin{equation*}
\big \langle \big[ -h^2 \Delta_{\mathcal{M}} , \Op_h^{\w}( \mathbf{a}_{h,R,\epsilon,\delta,\rho}) \big] \psi_h, \psi_h \big \rangle_{L^2(M)} = 0,
\end{equation*}
using \eqref{e:commutator_to_poisson}, and taking limits we get:
$$
\sum_{\jmath \in \{0,1 \}} (\eta + (-1)^{\jmath} \varsigma^2) \partial_y  a(y,z,\eta) dm_3^{\jmath} = 0.
$$
Finally, since
$$
\partial_\eta H_\eta^{\jmath} = \eta + (-1)^{ \jmath} \varsigma^2,
$$
using \eqref{e:invariance_3} and Lemma \ref{l:averaging_elliptic}, we obtain \eqref{e:drift_3}.
\end{proof}

\subsection{Normal form reduction}
\label{s:normal_form}
The aim of this section is to prove the invariance properties \eqref{e:drift_2} and \eqref{e:drift_4} of Theorem \ref{t:drift_invariance}. To this aim, we use the same strategy of \cite[\S 5]{Ar_Riv24}, consisting of a normal form procedure with respect to the harmonic oscillator.

\begin{proof}[Proof of \eqref{e:drift_2}] Let us first take $a \in \mathcal{C}_c^\infty(M_{x,y,z} \times \R_\sigma \times \R_\zeta)$ and define:
\begin{align}
\label{e:adapted_symbol_drift_2}
\mathbf{a}_{h,R,\epsilon}^\delta (X,\Xi)  :=  \breve{\chi}\left( \frac{h\sigma}{\delta} \right) \chi\left( \frac{h^2 \zeta}{\epsilon} \right) \breve{\chi} \left( \frac{ \zeta}{R} \right) a\left( x,y,z, h\sigma, h^2 \zeta \right).
\end{align}
Consider the Wigner equation
\begin{align*}
0 & = \big \langle [-h^2 \Delta_{\mathcal{M}}, \Op_h^{\w}( \mathbf{a}_{h,R,\epsilon}^\delta) \psi_h, \psi_h \big \rangle_{L^2(M)} \\[0.2cm]
& = \frac{h}{i} \big \langle \Op_h^{\w} ( \{ H_{\mathcal{M}}, \mathbf{a}_{h,R,\epsilon}^\delta \}) \psi_h, \psi_h \big \rangle_{L^2(M)} \\[0.2cm]
 & = \frac{h}{i} \int_{\R^9} \{ H_{\mathcal{M}}, \mathbf{a}_{h,R,\epsilon}^\delta \} \circ \bm{\kappa}_h\left( \frac{X + X'}{2} ,\Xi \right) \Psi_h(X') \overline{\Psi}_h(X) e^{i \Xi \cdot ( X - X')} \text{d}_h(X',\Xi,X) \\[0.2cm]
 & = \frac{h}{i} \int_{\R^9} \{ H_{\mathcal{M}} \circ \bm{\kappa}_h , \mathbf{A}_{h,R,\epsilon}^\delta \} \Psi_h(X') \overline{\Psi}_h(X) e^{i \Xi \cdot ( X - X')} \text{d}_h(X',\Xi,X),
\end{align*}
where recall that $\mathbf{A}_{h,R,\epsilon}^\delta = \mathbf{a}_{h,R,\epsilon}^\delta \circ \bm{\kappa}_h$. Let us denote:
\begin{align*}
w & = h^{1/2} \eta +  h^{2+1/2} x^2 \zeta, \quad \sigma  = 2h^{1+1/2}x \zeta,
\end{align*}
and, with a slightly abuse of notation, we denote:
\begin{align}
\label{e:abuse_notation_A}
\mathbf{A}_{h,R,\epsilon}^\delta(x,y,z,\sigma,\zeta) :=  \breve{\chi}\left( \frac{h\sigma}{\delta} \right) \chi\left( \frac{h^{2+1/2} \zeta}{\epsilon} \right) \breve{\chi} \left( \frac{ h^{1/2} \zeta}{R} \right) a\left( hx,h^{1/2} y, h^{1/2} z, h\sigma, h^{2+1/2} \zeta \right).
 \end{align}
Next observe that $\{  \xi , \sigma \} = 2 h^{2+1/2} \zeta  = - h \, \bm{\alpha}(hx) \sigma$, where $\bm{\alpha}$ has been introduced in Definition \ref{d:reeb_vector_field}. Thus we can summarize the commutator relations enter into play as:
\begin{equation}
\label{e:main_commutation_relations}
\{ \xi , \sigma \} = - h \bm{\alpha} \, \sigma, \quad \{ \xi, w \} = h \sigma, \quad \{ w , \sigma \} = 0.
\end{equation}
Now it is convenient (as done in \cite[\S 5]{Ar_Riv24}) to use complex coordinates. Defining:
$$
Z :=  \xi + i w, \quad \overline{Z} := \xi - iw,
$$
we have:
\begin{equation}
\label{e:Z_properties}
H_{\mathcal{M}} \circ \bm{\kappa}_h = \vert Z \vert^2 =  \xi^2 + w^2, \quad \{Z , \overline{Z} \} = 2 i \{ w, \xi \} = -2i h \, \sigma.
\end{equation}
Therefore, we can rewrite:
\begin{align*}
\{ \vert Z \vert^2 , \mathbf{A}_{h,R,\epsilon}^\delta \} & = Z \{ \overline{Z}, \mathbf{A}_{h,R,\epsilon}^\delta \} + \overline{Z} \{ Z , \mathbf{A}_{h,R,\epsilon}^\delta \}.
\end{align*}
We next deform slightly $\mathbf{A}_{h,R,\epsilon}^\delta$ in order to improve its commuting properties with $H_{\mathcal{M}} \circ \bm{\kappa}_h$ (see \cite{Ar_Riv24}). To this aim, we replace $\mathbf{A}_{h,R,\epsilon}^\delta$ by $\mathbf{A}_{h,R,\epsilon}^\delta + \mathbf{B}_{h,R,\epsilon}^\delta$, with: 
\begin{equation}
\label{e:B}
\mathbf{B}_{h,R,\epsilon}^\delta := \frac{- Z}{2ih  \sigma} \{ \overline{Z}, \mathbf{A}_{h,R,\epsilon}^\delta \} + \frac{\overline{Z}}{2 ih  \sigma} \{ Z , \mathbf{A}_{h,R,\epsilon}^\delta \}.
\end{equation}
Then we have:
\begin{align*}
\{ \vert Z \vert^2 , \mathbf{B}_{h,R,\epsilon}^\delta  \} & = - \frac{Z \{ \overline{Z}, Z \} \{ \overline{Z} , \mathbf{A}_{h,R,\epsilon}^\delta \} }{2 i h \sigma} + \frac{ \overline{Z} \{ Z , \overline{Z} \} \{Z , \mathbf{A}_{h,R,\epsilon}^\delta \}}{2i h \sigma} \\[0.2cm]
 & \quad - \frac{Z}{2i h \sigma} \{ \vert Z \vert^2, \{\overline{Z}, \mathbf{A}_{h,R,\epsilon}^\delta \} \} + \frac{\overline{Z}}{2i h  \sigma} \{ \vert Z \vert^2, \{ Z , \mathbf{A}_{h,R,\epsilon}^\delta \} \} \\[0.3cm]
  & \quad + \frac{\vert Z \vert^2 \big( \{Z, \sigma \} \{ \overline{Z} , \mathbf{A}_{h,R,\epsilon}^\delta \} - \{ \overline{Z}, \sigma \} \{ Z , \mathbf{A}_{h,R,\epsilon}^\delta \} \big)}{2 ih \sigma^2}  \\[0.2cm]
  & \quad + \frac{ Z^2 \{ \overline{Z}, \sigma \} \{ \overline{Z}, \mathbf{A}_{h,R,\epsilon}^\delta \}}{2ih \sigma^2} - \frac{ \overline{Z}^2 \{ Z, \sigma \} \{ Z, \mathbf{A}_{h,R,\epsilon}^\delta \}}{2i h\sigma^2}.
\end{align*}
Therefore, using \eqref{e:Z_properties}, we get:
\begin{align*}
\{ \vert Z \vert^2 , \mathbf{A}_{h,R,\epsilon}^\delta + \mathbf{B}_{h,R,\epsilon}^\delta  \} & = - \frac{Z}{2i h\sigma} \{ \vert Z \vert^2, \{\overline{Z}, \mathbf{A}_{h,R,\epsilon}^\delta \} \} + \frac{\overline{Z}}{2i h\sigma} \{ \vert Z \vert^2, \{ Z , \mathbf{A}_{h,R,\epsilon}^\delta \} \} \\[0.2cm]
 & \quad + \frac{\vert Z \vert^2 \big( \{Z, \sigma \} \{ \overline{Z} , \mathbf{A}_{h,R,\epsilon}^\delta \} - \{ \overline{Z}, \sigma \} \{ Z , \mathbf{A}_{h,R,\epsilon}^\delta \} \big)}{2 i h\sigma^2}  \\[0.2cm]
  & \quad + \frac{ Z^2 \{ \overline{Z}, \sigma \} \{ \overline{Z}, \mathbf{A}_{h,R,\epsilon}^\delta \}}{2i h\sigma^2} - \frac{ \overline{Z}^2 \{ Z, \sigma \} \{ Z, \mathbf{A}_{h,R,\epsilon}^\delta \}}{2ih \sigma^2} \\[0.2cm]
 & = \frac{ \vert Z \vert^2}{ 2ih \sigma} \Big( - \{ Z , \{ \overline{Z}, \mathbf{A}_{h,R,\epsilon}^\delta \} \} + \{ \overline{Z}, \{ Z, \mathbf{A}_{h,R,\epsilon}^\delta \} \} \Big)  \\[0.2cm]
 & \quad + \frac{1}{2 ih \sigma} \Big( \overline{Z}^2 \{ Z , \{ Z,  \mathbf{A}_{h,R,\epsilon}^\delta \} \} - Z^2 \{ \overline{Z}, \{ \overline{Z}, \mathbf{A}_{h,R,\epsilon}^\delta \} \} \Big)  \\[0.2cm]
 & \quad + \frac{\vert Z \vert^2 \big( \{Z, \sigma \} \{ \overline{Z} , \mathbf{A}_{h,R,\epsilon}^\delta \} - \{ \overline{Z}, \sigma \} \{ Z , \mathbf{A}_{h,R,\epsilon}^\delta \} \big)}{2 ih \sigma^2}  \\[0.2cm]
  & \quad + \frac{ Z^2 \{ \overline{Z}, \sigma \} \{ \overline{Z}, \mathbf{A}_{h,R,\epsilon}^\delta \}}{2ih \sigma^2} - \frac{ \overline{Z}^2 \{ Z, \sigma \} \{ Z, \mathbf{A}_{h,R,\epsilon}^\delta \}}{2ih \sigma^2}.
\end{align*}
On the other hand, we have, from \eqref{e:abuse_notation_A},
\begin{align}
\label{e:action_Z}
\big \{ Z, \mathbf{A}_{h,R,\epsilon}^\delta \big \} & = X_Z \mathbf{A}_{h,R,\epsilon}^\delta + h \{ Z, \sigma \} \partial_\sigma \mathbf{A}_{h,R,\epsilon}^\delta, \\[0.2cm]
\label{e:action_Z_bar}
\big \{ \overline{Z} , \mathbf{A}_{h,R,\epsilon}^\delta \big \} & = X_{\overline{Z}} \mathbf{A}_{h,R,\epsilon}^\delta +h  \{ \overline{Z}, \sigma \} \partial_\sigma \mathbf{A}_{h,R,\epsilon}^\delta,
\end{align}
where 
$$
X_Z := \partial_x + ih^{1/2} \big( \partial_y + h^2 x^2 \partial_z \big), \quad X_{\overline{Z}} := \partial_x - ih^{1/2} \big( \partial_y + h^2 x^2 \partial_z \big).
$$
By \eqref{e:main_commutation_relations}, $\{ Z, \sigma \} = \{\overline{Z}, \sigma \} = - h \bm{\alpha}\, \sigma$. 
This, \eqref{e:action_Z}, \eqref{e:action_Z_bar}, and the definition \eqref{e:B} of $\mathbf{B}_{h,R,\epsilon}^\delta$, implies that 
$$
 \int_{\R^9} \mathbf{B}_{h,R,\epsilon}^\delta  \Psi_h(X') \overline{\Psi}_h(X) e^{i \Xi \cdot ( X - X')} \text{d}_h(X',\Xi,X) = \mathcal{O}\left( \frac{h}{\delta} \right),
$$
hence
$$
\lim_{R \to + \infty} \lim_{\delta \to 0} \lim_{\epsilon \to 0} \lim_{h \to 0}  \int_{\R^9} \mathbf{B}_{h,R,\epsilon}^\delta  \Psi_h(X') \overline{\Psi}_h(X) e^{i \Xi \cdot ( X - X')} \text{d}_h(X',\Xi,X) = 0.
$$
In other words, the limit of $I_{h,R,\epsilon,\delta}^2(a)$ does not change by the addition of the perturbation $\mathbf{B}_{h,R,\epsilon}^\delta$ to $\mathbf{A}_{h,R,\epsilon}^\delta$.  On the other hand:
\begin{align*}
\{ Z, \{ \overline{Z} , \mathbf{A}_{h,R,\epsilon}^\delta \} \} & \\[0.2cm]
 & \hspace*{-1.2cm}  = X_Z X_{\overline{Z}} \mathbf{A}_{h,R,\epsilon}^\delta - h^2 X_Z (  \bm{\alpha}\, \sigma \partial_\sigma \mathbf{A}_{h,R,\epsilon}^\delta )  - h^2  \bm{\alpha}\, \sigma \partial_\sigma (X_{\overline{Z}} \mathbf{A}_{h,R,\epsilon}^\delta) +  h^4  (\bm{\alpha}\, \sigma)^2 \partial_\sigma^2 \mathbf{A}_{h,R,\epsilon}^\delta, \\[0.2cm]
\{ \overline{Z}, \{ Z , \mathbf{A}_{h,R,\epsilon}^\delta \} \} & \\[0.2cm]
 & \hspace*{-1.2cm} = X_{\overline{Z}} X_Z \mathbf{A}_{h,R,\epsilon}^\delta - h^2 X_{\overline{Z}} (  \bm{\alpha}\, \sigma \partial_\sigma \mathbf{A}_{h,R,\epsilon}^\delta ) - h^2  \bm{\alpha}\, \sigma \partial_\sigma (X_{Z} \mathbf{A}_{h,R,\epsilon}^\delta) +  h^4  (\bm{\alpha}\, \sigma)^2 \partial_\sigma^2 \mathbf{A}_{h,R,\epsilon}^\delta, \\[0.2cm]
\{ Z, \{ Z , \mathbf{A}_{h,R,\epsilon}^\delta \} \} & \\[0.2cm]
 & \hspace*{-1.2cm}= X_Z X_Z \mathbf{A}_{h,R,\epsilon}^\delta - h^2 X_Z (  \bm{\alpha}\, \sigma \partial_\sigma \mathbf{A}_{h,R,\epsilon}^\delta )  - h^2  \bm{\alpha}\, \sigma \partial_\sigma (X_{Z} \mathbf{A}_{h,R,\epsilon}^\delta) +  h^4  (\bm{\alpha}\, \sigma)^2 \partial_\sigma^2 \mathbf{A}_{h,R,\epsilon}^\delta, \\[0.2cm]
\{ \overline{Z}, \{ \overline{Z} , \mathbf{A}_{h,R,\epsilon}^\delta \} \} & \\[0.2cm]
 & \hspace*{-1.2cm} = X_{\overline{Z}} X_{\overline{Z}} \mathbf{A}_{h,R,\epsilon}^\delta - h^2 X_{\overline{Z}} (  \bm{\alpha}\, \sigma \partial_\sigma \mathbf{A}_{h,R,\epsilon}^\delta )  - h^2  \bm{\alpha}\, \sigma \partial_\sigma (X_{\overline{Z}} \mathbf{A}_{h,R,\epsilon}^\delta) +  h^4  (\bm{\alpha}\, \sigma)^2 \partial_\sigma^2 \mathbf{A}_{h,R,\epsilon}^\delta.
\end{align*}
Moreover, we have:
\begin{align*}
X_Z  ( \bm{\alpha}\, \sigma) = X_{\overline{Z}} (  \bm{\alpha}\, \sigma)  =  \partial_x ( \bm{\alpha}\, \sigma).
\end{align*}
Thus, we obtain:
\begin{align*}
- \{ Z , \{ \overline{Z}, \mathbf{A}_{h,R,\epsilon}^\delta \} \} + \{ \overline{Z}, \{ Z , \mathbf{A}_{h,R,\epsilon}^\delta \} \}  & = \big( X_{\overline{Z}} X_Z - X_Z X_{\overline{Z}} \big) \mathbf{A}_{h,R,\epsilon}^\delta  \\[0.2cm]
 & = 2ih [\mathcal{X}_1, \mathcal{X}_2] \mathbf{A}_{h,R,\epsilon}^\delta \\[0.2cm]
 & = 2ih \mathcal{X}_3 \mathbf{A}_{h,R,\epsilon}^\delta,
\end{align*}
and similarly, 
\begin{align*}
- h \bm{\alpha} \, \sigma \big( \{ \overline{Z}, \mathbf{A}_{h,R,\epsilon}^\delta \} - \{ Z, \mathbf{A}_{h,R,\epsilon}^\delta \} \big) =  2hi \bm{\alpha} \mathcal{X}_2 \mathbf{A}_{h,R,\epsilon}^\delta.
\end{align*}
Therefore, using the change of variables \eqref{e:tricky_change} and taking limits we obtain:
\begin{align*}
\lim_{R \to + \infty} \lim_{\delta \to 0} \lim_{\epsilon \to 0} \lim_{h \to 0} \frac{1}{h^2} \int_{\R^9} \{\vert Z \vert^2 , \mathbf{A}_{h,R,\epsilon}^\delta + \mathbf{B}_{h,R,\epsilon}^\delta \} \Psi_h(X') \overline{\Psi}_h(X) e^{i \Xi \cdot ( X - X')} \text{d}_h(X',\Xi,X) & \\[0.2cm]
& \hspace*{-13cm} =  \operatorname{Tr} \int_{M_{x,y,z} \setminus \mathscr{S} \times \R_\sigma^*} \left( \frac{\widehat{G}_\sigma}{\sigma} \mathcal{Z} (a) + \left( \sigma D_w + iw \right)^2 \mathcal{Y}_1 (a) +  \left( \sigma D_w - iw \right)^2 \mathcal{Y}_1 (a) \right) dM_2,
\end{align*}
where $\mathcal{Y}_1$ and $\mathcal{Y}_2$ are two vector fields on $M_{x,y,z} \setminus \mathscr{S} \times \R_\sigma^*$. Finally, using the invariance property \eqref{e:invariance_2} and the fact that
\begin{align*}
\frac{1}{T_\sigma} \int_0^{T_\sigma} e^{-it \widehat{G}_\sigma} \left( \sigma D_w + i w \right)^2 e^{it \widehat{G}_\sigma} \, dt & = 0, \\[0.2cm]
\frac{1}{T_\sigma} \int_0^{T_\sigma} e^{-it \widehat{G}_\sigma} \left( \sigma D_w - i w \right)^2 e^{it \widehat{G}_\sigma} \, dt & = 0,
\end{align*}
where $T_\sigma = 2\pi/\sqrt{\vert \sigma \vert}$ is the period of the Hamiltonian flow $e^{-it \widehat{G}_\sigma}$, we get, by computing the trace with respect to the orthonormal basis $\{\phi_j \}_{j \in \mathbb{N}}$ of eigenvectors of $\widehat{G}_\sigma$, that for each $j \in \mathbb{N}$,
$$
 \frac{ \nu_j(\sigma)}{\vert \sigma \vert} \mathcal{Z} \mu_{2,j} = (2j + 1) \mathcal{Z}\mu_{2,j} = 0.
$$
\end{proof}

\begin{proof}[Proof of \eqref{e:drift_4}]

The proof is very similar to that of \eqref{e:drift_2}. Let us take first $a \in \mathcal{C}_c^\infty(M_{x,y,z} \times \R^2_{\sigma,\zeta})$ not depending on $(w,\xi)$, and define:
\begin{align}
\label{e:a_4}
\mathbf{a}_{h,R,\epsilon,\delta}^\rho =  \breve{\bm{\chi}}_\rho(\eta,\zeta) \chi\left( \frac{h\sigma}{\delta}  \right)\chi\left( \frac{h^2 \zeta}{\epsilon} \right) \breve{\chi} \left( \frac{\zeta}{R} \right) a\left( x,y,z, h\sigma, h^2 \zeta \right).
\end{align}
Consider next the Wigner equation
\begin{align*}
0 & = \big \langle [-h^2 \Delta_{\mathcal{M}}, \Op_h^{M}(\mathbf{a}_{h,R,\epsilon,\delta}^\rho) \psi_h, \psi_h \big \rangle_{L^2(M)} \\[0.2cm]
& = \frac{h}{i} \big \langle \Op_h^{M} ( \{ H_{\mathcal{M}}, \mathbf{a}_{h,R,\epsilon,\delta}^\rho \}) \psi_h, \psi_h \big \rangle_{L^2(M)} \\[0.2cm]
 & = \frac{h}{i} \int_{\R^9} \{ H_{\mathcal{M}}, \mathbf{a}_{h,R,\epsilon,\delta}^\rho \} \circ \bm{\kappa}_{h,\delta}\left( \frac{X+X'}{2} ,\Xi \right) \Psi_h(X') \overline{\Psi}_h(X) e^{i \Xi \cdot ( X - X')} \text{d}_h(X',\Xi,X) \\[0.2cm]
 & = \frac{h}{i} \int_{\R^9} \{ H_{\mathcal{M}} \circ \bm{\kappa}_{h,\delta} , \mathbf{A}_{h,R,\epsilon,\delta,\rho} \} \Psi_h(X') \overline{\Psi}_h(X) e^{i \Xi \cdot ( X - X')} \text{d}_h(X',\Xi,X),
\end{align*}
where recall that $\mathbf{A}_{h,R,\epsilon,\delta}^\rho = \mathbf{a}_{h,R,\epsilon,\delta}^\rho \circ \bm{\kappa}_{h,\delta}$ with respect to the change of variables \eqref{e:change_delta}. Set $\hbar_\delta := \delta^{-1/2}h$ and define:
\begin{align*}
w & = h^{1/2} \eta + (\hbar_\delta x)^2 h^{1/2} \zeta, \\[0.2cm]
\sigma & := 2 h^{1/2} \hbar_\delta x \zeta, \\[0.2cm]
\xi_\delta & := \delta^{1/2}\xi.
\end{align*}
Let us denote:
\begin{align}
\label{e:abuse_of_notation_A2}
\mathbf{A}_{h,R,\epsilon,\delta}^\rho (x,y,z,\sigma,\zeta) & \\[0.2cm]
 & \hspace*{-2.8cm} := \breve{\chi}\left( \frac{h^{1/2} \eta}{\rho} \right)  \chi\left( \frac{h\sigma}{\delta} \right)\chi\left( \frac{h^{2+1/2} \zeta}{\epsilon} \right) \breve{\chi} \left( \frac{h^{1/2} \zeta}{R} \right) a\left( \hbar_\delta x, h^{1/2}y,h^{1/2}z, h\sigma, h^{2+1/2} \zeta \right). \notag
\end{align}
Next observe that:
$$
\{  \xi_\delta , \sigma \} = 2h^{1+1/2} \zeta  = - h \, \bm{\alpha}(\hbar_\delta x) \sigma,
$$
where recall that $\bm{\alpha} = - 1/x$. Thus we have the commutator relations:
\begin{equation}
\label{e:main_commutation_relations_2}
\{ \xi_\delta , \sigma \} = - h \bm{\alpha} \, \sigma, \quad \{ \xi_\delta , w \} = h \sigma, \quad \{ w , \sigma \} = 0.
\end{equation}
Now pass to complex coordinates, defining:
$$
Z :=  \xi_\delta + i w, \quad \overline{Z} := \xi_\delta - iw.
$$
We have:
\begin{equation}
\label{e:Z_properties_2}
H_{\mathcal{M}} \circ \bm{\kappa}_{h,\delta} = \vert Z \vert^2 =  \xi_\delta^2 + w^2, \quad \{Z , \overline{Z} \} = 2 i \{ w, \xi_\delta \} = -2i h \, \sigma.
\end{equation}
Therefore, we can rewrite:
\begin{align*}
\{ \vert Z \vert^2 , \mathbf{A}_{h,R,\epsilon,\delta}^\rho \} & = Z \{ \overline{Z}, \mathbf{A}_{h,R,\epsilon,\delta}^\rho \} + \overline{Z} \{ Z , \mathbf{A}_{h,R,\epsilon,\delta}^\rho \}.
\end{align*}
We next deform $\mathbf{A}_{h,R,\epsilon,\delta}^\rho$ to improve its commuting properties with $H_{\mathcal{M}} \circ \bm{\kappa}_{h,\delta}$. To this aim, we introduce the perturbative term 
\begin{equation}
\label{e:B_2}
\mathbf{B}_{h,R,\epsilon,\delta}^\rho := \frac{- Z}{2i h \sigma} \{ \overline{Z}, \mathbf{A}_{h,R,\epsilon}^\delta \} + \frac{\overline{Z}}{2 i h \sigma} \{ Z , \mathbf{A}_{h,R,\epsilon}^\delta \}.
\end{equation}
On the other hand, we have, from \eqref{e:abuse_of_notation_A2},
\begin{align}
\label{e:action_Z_2}
\big \{ Z, \mathbf{A}_{h,R,\epsilon}^\delta \big \} & = X_Z \mathbf{A}_{h,R,\epsilon}^\delta + h \{ Z, \sigma \} \partial_\sigma \mathbf{A}_{h,R,\epsilon}^\delta, \\[0.2cm]
\label{e:action_Z_bar_2}
\big \{ \overline{Z} , \mathbf{A}_{h,R,\epsilon}^\delta \big \} & = X_{\overline{Z}} \mathbf{A}_{h,R,\epsilon}^\delta +h  \{ \overline{Z}, \sigma \} \partial_\sigma \mathbf{A}_{h,R,\epsilon}^\delta,
\end{align}
where 
$$
X_Z := \delta^{1/2} \partial_x + ih^{1/2} \big( \partial_y + V(hx) \partial_z \big), \quad X_{\overline{Z}} := \delta^{1/2} \partial_x - ih^{1/2} \big( \partial_y + V(hx) \partial_z \big).
$$
By \eqref{e:main_commutation_relations}, $\{ Z, \sigma \} = \{\overline{Z}, \sigma \} = - h \bm{\alpha}\, \sigma$. 
Moreover, by \eqref{e:x_away_from_crit}, we have:
$$
\vert \sigma \vert \gtrsim (\rho R)^{1/2}.
$$
This, \eqref{e:action_Z_2}, \eqref{e:action_Z_bar_2}, 
$$
 \int_{\R^9} \mathbf{B}_{h,R,\epsilon,\delta}^\rho  \Psi_h(X') \overline{\Psi}_h(X) e^{i \Xi \cdot ( X - X')} \text{d}_h(X',\Xi,X) = \mathcal{O}((\rho R)^{-1/2}) +  \mathcal{O}(h),
$$
hence
$$
\lim_{\rho \to ¡ \infty} \lim_{R \to + \infty} \lim_{\delta \to 0} \lim_{\epsilon \to 0} \lim_{h \to 0}  \int_{\R^9} \mathbf{B}_{h,R,\epsilon,\delta}^\rho  \Psi_h(X') \overline{\Psi}_h(X) e^{i \Xi \cdot ( X - X')} \text{d}_h(X',\Xi,X) = 0,
$$
that is, the limit of $I_{h,R,\epsilon,\delta}^2(a)$ does not change by the addition of the perturbation $\mathbf{B}_{h,R,\epsilon}^\delta$ to $\mathbf{A}_{h,R,\epsilon}^\delta$.  On the other hand, replacing $a (x,y,z,h\sigma,h^2\zeta)$ by $h\sigma \cdot a (x,y,z,h\sigma,h^2\zeta)$ in the definition of $\mathbf{a}_{h,R,\epsilon,\delta}^\rho$ given by \eqref{e:a_4} to compensate the denominator appearing in \eqref{e:B_2}, repeating the calculus of the proof of \eqref{e:drift_2}, taking limits we obtain in this case:
\begin{align*}
\lim_{\rho \to + \infty} \lim_{R \to + \infty} \lim_{\delta \to 0} \lim_{\epsilon \to 0} \lim_{h \to 0} \frac{1}{h^2} \int_{\R^9} \{\vert Z \vert^2 ,  \mathbf{A}_{h,R,\epsilon}^\delta + \mathbf{B}_{h,R,\epsilon}^\delta \} \Psi_h(X') \overline{\Psi}_h(X) e^{i \Xi \cdot ( X - X')} \text{d}_h(X',\Xi,X) & \\[0.2cm]
& \hspace*{-15cm} =  \int_{M_{x,y,z} \setminus \mathscr{S} \times \R_w \times \R_\xi} \left( G(w,\xi)\,  \mathcal{Z} (a) + \left(\xi + iw \right)^2 \mathcal{Y}^\infty_1 (a) +  \left( \xi - iw \right)^2 \mathcal{Y}^\infty_2 (a) \right) dm_4,
\end{align*}
where $\mathcal{Y}^\infty_1$ and $\mathcal{Y}^\infty_2$ are two vector fields on $M_{y,z,\zeta} \setminus \mathscr{S} \times \R_\sigma$. Finally, using the invariance property \eqref{e:invariance_4} and the fact that
\begin{align*}
\frac{1}{2\pi} \int_0^{2\pi}  \left( \xi + i w \right)^2 \circ \phi_t^{G} \, dt  = 
\frac{1}{2\pi} \int_0^{2\pi}  \left(\xi - i w \right)^2 \circ \phi_t^G \, dt  = 0,
\end{align*}
we get
$$
G(w,\xi) \, \mathcal{Z} m_4 \vert_{M} = 0.
$$
The proof of \eqref{e:drift_4} is then complete.
\end{proof}

\section{Weak drift invariance}
\label{s:weak_drift}

In this section we prove the additinal weak drift invariance properties of the measures $\mu_{1,k}$ estabilshed by Theorems \ref{t:more_two_microlocal} and \ref{t:weak_drift_invariance_1}, by a fruther two-microlocalization near the critical points of $\lambda_k(\eta,\zeta)$. 

\begin{proof}[Proof of Theorem \ref{t:more_two_microlocal}]
This proof is standard (see  \cite{An_Mac_Leau16}, \cite{AnantharamanMacia14}, \cite{An_Leau14}, \cite{Fer00}, \cite{Mac_Riv18}), and we only sketch it. Let us consider the family of symbols $a_\jmath \in \mathcal{C}^\infty(\T^2_{y,z} \times \mathcal{B}_\jmath  \times \R_v)$ being compactly supported in the variables $(y,z,\eta,\zeta)$ and homogeneus of degree zero at infinity in the variable $v$, that is, there exists $R_0 > 0$ and $a_\infty \in \mathcal{C}_c^\infty( \T^2_{y,z} \times \mathcal{B}_\jmath \times \mathbb{S}^0_\theta)$ such that:
$$
a(y,z,\eta,\zeta,v) = a_\infty \left(y,z,\eta,\zeta, \frac{v}{\vert v \vert} \right), \quad \text{if } \vert v \vert \geq R_0.
$$
For $R,\epsilon,d,r > 0$, and each $\jmath \in \{0,1\}$, we decompose the symbol $\mathbf{a}_{h,R,\epsilon}^\jmath$ defined by \eqref{e:symbol_for_new_two_microlocal} into 
$$
\mathbf{a}_{h,R,\epsilon}^\jmath =\mathbf{a}_{h,R,\epsilon}^{\jmath,d} + \mathbf{a}_{h,R,\epsilon,d}^{\jmath, r} +\mathbf{a}_{h,R,\epsilon,r}^\jmath,
$$
where we set:
\begin{align}
\label{e:cut_1}
\mathbf{a}_{h,R,\epsilon}^{\jmath,d} &:=\breve{\chi} \left( \frac{\eta - \eta_\jmath}{d} \right) \mathbf{a}_{h,R}^\jmath, \\[0.2cm]
\label{e:cut_2}
\mathbf{a}_{h,R,\epsilon,d}^{\jmath, r} &:=\chi \left( \frac{\eta - \eta_\jmath}{d} \right)  \breve{\chi}\left( \frac{ \eta - \eta_\jmath}{hr} \right)  \mathbf{a}_{h,R}^\jmath, \\[0.2cm]
\label{e:cut_3}
\mathbf{a}_{h,R,\epsilon,d,r}^\jmath & := \chi \left( \frac{\eta - \eta_\jmath}{d} \right)  \chi \left( \frac{\eta - \eta_\jmath }{hr} \right) \mathbf{a}_{h,R}^\jmath.
\end{align}
Then, we define respectively the following distributions acting on $a$ by
\begin{align*}
I^{1,\jmath}_{h,R,\epsilon,d}(a) & := \Big \langle \Op_h^\w ( \mathbf{a}_{h,R,\epsilon}^{\jmath,d}) \psi_h, \psi_h \Big \rangle_{L^2(M)}, \\[0.2cm]
I^{2,\jmath}_{h,R,\epsilon,d,r}(a) & := \Big \langle \Op_h^\w ( \mathbf{a}_{h,R,\epsilon,d}^{\jmath,r}) \psi_h, \psi_h \Big \rangle_{L^2(M)}, \\[0.2cm]
I^{3,\jmath}_{h,R,\epsilon,d,r}(a) & := \Big \langle \Op_h^\w ( \mathbf{a}_{h,R,\epsilon,d,r}^{\jmath}) \psi_h, \psi_h \Big \rangle_{L^2(M)}.
\end{align*}
The first one, $I^{1,\jmath}_{h,R,\epsilon,d}(a)$ can be bounded using Calderón-Vaillancourt theorem. Revisiting the proof of Proposition \ref{l:measure_m_1}, we obtain, modulo subsequences, that:
\begin{align*}
\lim_{d \to 0}  \lim_{R \to +\infty} \lim_{\epsilon \to 0} \lim_{h \to 0} I^{1,\jmath}_{h,R,\epsilon,d}(a) & \\[0.2cm]
 & \hspace*{-2cm} =  \int_{\T^2_{y,z} \times \R_\eta \times \R^*_\zeta} \mathbf{1}_{\eta \neq \eta_\jmath} a_\infty \left( y,z,\eta,\zeta, \frac{\eta}{\vert \eta \vert} \right) d\mu_{1,\mathbf{k}}.
\end{align*}
Similarly, for $I^{2,\jmath}_{h,R,\epsilon,d,r}(a)$ using \eqref{e:concentration_1}, one gets the existence of $\bm{\nu}^\jmath_{1,\mathbf{k}} \in \mathcal{M}_+(\T^2_{y,z} \times \mathbb{S}^1_\theta)$ such that
\begin{align*}
\lim_{d \to 0} \lim_{r \to +\infty} \lim_{R \to + \infty} \lim_{\epsilon \to 0} \lim_{h \to 0}  I^{2,\jmath}_{h,R,\epsilon,d,r}(a) =  \int_{\T^2_{y,z} \times \mathbb{S}^0_\theta} a_\infty\left( y,z,\eta_\jmath, \zeta_\jmath, \theta \right) d\bm{\nu}_{1,\mathbf{k}}^{\jmath}.
\end{align*}
Finally, for $I^{3,\jmath}_{h,R,\epsilon,d,r}(a)$, defining the cut-off
$$
\bm{\chi}_{h,R,\epsilon,r,d}(\eta,\zeta): =  \chi \left( \frac{\eta - \eta_\jmath}{d} \right)  \chi \left( \frac{ \eta - \eta_\jmath}{hr} \right) \breve{\chi}\left( \frac{h^2 \zeta}{\epsilon} \right) \breve{\chi}\left( \frac{ \zeta}{R} \right),
$$ 
we have:
\begin{align*}
I^{3,\jmath}_{h,R,\epsilon,d,r}(a) & \\[0.2cm]
 & \hspace*{-1.5cm} = \int_{\R^6} \bm{\chi}_{h,R,\epsilon,r,d}(h\eta,h\zeta) a\left( \frac{Y + Y'}{2}, h\eta, h^3\zeta, \frac{h\eta - \eta_\jmath}{h} \right)  \\[0.2cm]
 & \hspace*{5cm} \times  \psi_h(x,Y')\overline{\psi}_h(x,Y) e^{i \Theta \cdot (Y - Y')} \text{d}(x,Y',\Theta,Y) \\[0.2cm]
 & \hspace*{-1.5cm} = \int_{\R^6} \bm{\chi}_{h,R,\epsilon,r,d}\left(\eta_\jmath + h\eta, h\zeta \right) a\left( \frac{Y + Y'}{2},\eta_\jmath + h\eta, h^3\zeta, \eta\right) \\[0.2cm]
 & \hspace*{5cm} \times  \psi_h(x,Y')\overline{\psi}_h(x,Y)  e^{i \Theta \cdot (Y - Y')} \text{d}(x,Y',\Theta,Y)
\end{align*}
Using the fact that
$$
a\left( \frac{Y + Y'}{2},\eta_\jmath + h\eta,  h^3\zeta, \eta \right) = a\left(  \frac{Y + Y'}{2},\eta_\jmath, h^3 \zeta, \eta \right) + O(h),
$$
and again the concentration property \eqref{e:concentration_1}, one obtains the existence of a positive Radon measure $\mathbf{M}^\jmath_{1,\mathbf{k}} \in \mathcal{M}_+(\T^1_z \times \R^*_\zeta ; \mathcal{L}^1(L^2(\T_y)))$ such that
\begin{align*}
\lim_{d \to 0} \lim_{r \to +\infty} \lim_{R \to + \infty} \lim_{\epsilon \to 0} \lim_{h \to 0 } I^{3,\jmath}_{h,R,\epsilon,d,r}(a) =  \int_{\T^1_z } \operatorname{Tr} \Big( \Op_{(y,\eta)}^{\T} \left( a\left(y,z,\eta_\jmath, \zeta_\jmath, \eta \right) \right) d \mathbf{M}_{1,\mathbf{k}}^{\jmath} \Big).
\end{align*}
This concludes the proof.
\end{proof}

\begin{proof}[Proof of Theorem \ref{t:weak_drift_invariance_1}]
For each $(\eta, \zeta) \in \R \times \R^*$, we consider the semiclassical operator
$$
\widehat{H}_h(\eta,\zeta) := \Op_h^{\R}\big( \xi^2 + (\eta + x^2 \zeta)^2 \big) = h^2 D_\xi^2 + (\eta + x^2 \zeta)^2.
$$
Let $w(x,\eta,\zeta) = \eta + x^2\zeta$ be defined as multiplication operator on $L^2(\R_x)$, we take the average of $w$ by the quantum flow $e^{-it \widehat{H}_h(\eta,\zeta)}$ as the formal limit:
$$
\langle w \rangle_h := \lim_{T \to + \infty} \frac{1}{T} \int_0^T  e^{it \widehat{H}_h(\eta,\zeta)} \, w(x,\eta,\zeta) \, e^{-it \widehat{H}_h(\eta,\zeta)} dt.
$$
We do not give for the moment a more precise sense to the operator $\langle w \rangle_h$, but it suffices to remark that it is a diagonal operator with respect to the eigenbasis of $\widehat{H}_h(\eta,\zeta)$, which let us recall that, by Proposition \ref{p:simple_montgomery_martinet}, has discrete spectrum with simple eigenvalues. Indeed one can exactly calculate its matrix elements in this basis (see \eqref{e:matrix_element_w} below). Let us also introduce the operator:
\begin{equation}
\label{e:solution_cohomological}
\widehat{\mathcal{G}}_h(\eta,\zeta) := \lim_{T \to +\infty} \frac{1}{T} \int_0^T \int_0^t e^{is \widehat{H}_h(\eta,\zeta)} \big( w - \langle w \rangle_h \big) e^{-is \widehat{H}_h(\eta,\zeta)} ds \, dt.
\end{equation}
This operator appears naturally in normal form constructions based on the averaging method \cite{Uribe85}. Notice that $\widehat{\mathcal{G}}_h$ solves the cohomological equation
\begin{equation}
\label{e:semiclassic_cohomological_equation}
[\widehat{H}_h, \widehat{\mathcal{G}}_h ] = \frac{1}{i} \left(  \langle w  \rangle_h - w \right).
\end{equation}
Now, given $\tau > 0$, we define the localization operator
$$
\widehat{\mathcal{E}}_{h,\tau}(\eta,\zeta) := \chi \left( \frac{  \widehat{H}_h(\eta,\zeta) - 1 }{\tau} \right).
$$
Let $[\cdot, \cdot]^\perp$ denote the anticommutator\footnote{The anticommutator of two operators is defined by $[A,B]^\perp := AB + BA$.}. By \eqref{e:semiclassic_cohomological_equation} and since the operator $\widehat{\mathcal{E}}_{h,\tau}$ is a function of $\widehat{H}_h$ (hence $[\widehat{H}_h, \mathcal{E}_{h,\tau}] = 0$), the operator $\widehat{\mathcal{G}}_{h,\tau} := [ \widehat{\mathcal{G}}_h , \widehat{\mathcal{E}}_{h,\tau}]^\perp$ solves the cohomological equation
\begin{equation}
\label{e:cohomological_non_commuting}
[\widehat{H}_h, \widehat{\mathcal{G}}_{h,\tau}] = \frac{1}{i} [\langle w \rangle_h - w , \widehat{\mathcal{E}}_{h,\tau}]^\perp.
\end{equation}
Let us denote
$$
\bm{\chi}_{h,R,\epsilon}(\eta,\zeta) := \breve{\chi} \left( \frac{h^2 \zeta}{\epsilon} \right) \breve{\chi} \left( \frac{\zeta}{R} \right),
$$
and define the operator valued symbol (depending on the scale $\tau$):
$$
\mathbf{a}_{h,R}^\epsilon(y,z,\eta,\zeta) = \bm{\chi}_{h,R,\epsilon}(\eta,\zeta) \Big(  \widehat{\mathcal{E}}_{h,\tau}(\eta,\zeta) + h \widehat{\mathcal{G}}_{h,\tau}(\eta,\zeta)  \partial_y \Big) a\left( y,z,\eta,h^2\zeta, \frac{\eta - \eta_\jmath}{h} \right).
$$
By the semiclassical pseudodifferential calculus and using the cohomological equation \eqref{e:cohomological_non_commuting}, we have the following commutator formula:
\begin{align*}
 \big[ h^2 \Delta_{\mathcal{M}}, \Op_h^{\T^2} \big( \mathbf{a}_{h,R}^\epsilon \big) \big] & \\[0.2cm]
 & \hspace*{-2cm}  = \Op_h^{\T^2} \big( [\widehat{H}_h, \mathbf{a}_{h,R}^\epsilon] \big) + \frac{h}{2i} \Op_h^{\T^2} \big( [ \partial_\eta \widehat{H}_h , \partial_y \mathbf{a}_{h,R}^\epsilon]^\perp \big) + \mathcal{O}(h^3) \\[0.2cm]
 & \hspace*{-2cm} = h \Op_h^{\T^2} \big( \bm{\chi}_{h,R,\epsilon} \, [\widehat{H}_h, \widehat{\mathcal{G}}_{h,\tau} ]  \partial_y a \big)  \\[0.2cm]
 & \hspace*{-2cm} \quad + \frac{h}{i} \Op_h^{\T^2} \big( \bm{\chi}_{h,R,\epsilon} \, [ w,  \widehat{\mathcal{E}}_{h,\tau}]^\perp \, \partial_y a \big) \\[0.2cm]
 & \hspace*{-2cm} \quad + \frac{h^2}{i} \Op_h^{\T^2} \big( \bm{\chi}_{h,R,\epsilon} \, [ w, \widehat{\mathcal{G}}_{h,\tau}]^\perp \partial_y^2 a \big) + \mathcal{O}(h^3) \\[0.2cm]
 & \hspace*{-2cm} =  \frac{h}{i} \Op_h^{\T^2} \big(  \bm{\chi}_{h,R,\epsilon} \,  [ \langle w \rangle_h,  \widehat{\mathcal{E}}_{h,\tau}]^\perp \partial_y a \big) \\[0.2cm]
 & \hspace*{-2cm} \quad + \frac{h^2}{i} \Op_h^{\T^2} \big(\bm{\chi}_{h,R,\epsilon} \, [ w, \widehat{\mathcal{G}}_{h,\tau}]^\perp \partial_y^2 a \big) + \mathcal{O}(h^3).
\end{align*}
Next, as done before in \eqref{e:Wigner_1}, we consider again the Wigner equation
$$
0 = \Big \langle \big[ -h^2 \Delta_{\mathcal{M}}, \Op_h^{\T^2} \big( \mathbf{a}_{h,R}^\epsilon \big) \big] \psi_h, \psi_h \Big \rangle_{L^2(M)}. 
$$
In view of \eqref{e:commutator_in_trace_1}, denoting $\widehat{\mathcal{G}}_\tau := \widehat{\mathcal{G}}_{1,\tau}$ and $\langle \cdot \rangle := \langle \cdot \rangle_1$, we have the trace formula:
\begin{align*}
\mathcal{O}(h^3) & = \frac{h}{i}  \int_{\R^6}  \operatorname{Tr}_{L^2(\R_x)} \Big[ [\langle w \rangle, \widehat{\mathcal{E}}_\tau]^\perp (\Theta_h) \mathcal{T}_\epsilon(\Theta_h)  \mathcal{K}_h(Y,Y') \Big]  \partial_y a\left( \frac{Y + Y'}{2},h\Theta, \frac{h\eta - \eta_\jmath}{h} \right)  \\[0.2cm]
 & \hspace*{10cm} \times e^{i \Theta \cdot (Y-Y')} \text{d}(Y',\Theta,Y) \\[0.2cm]
&   \quad + \frac{h^2}{i} \int_{\R^6} \operatorname{Tr}_{L^2(\R_x)} \Big[ [ w, \widehat{\mathcal{G}}_\tau]^\perp(\Theta_h) \mathcal{T}_\epsilon(\Theta_h)  \mathcal{K}_h(Y,Y') \Big] \partial_y^2 a\left( \frac{Y + Y'}{2},h\Theta, \frac{h\eta - \eta_\jmath}{h} \right) \\[0.2cm]
& \hspace*{10cm} \times  e^{i \Theta \cdot (Y - Y')} \text{d}(Y',\Theta,Y),
\end{align*}
where we now (re)define $\Theta_h = (h\eta, h^3 \zeta)$ and  $\mathcal{K}_h(Y,Y')$ as the operator with integral kernel given by the tensor product
$$
\psi_h \otimes \psi_h(x',Y',x,Y) = \psi_h(x',Y') \overline{\psi}_h(x,Y).
$$ 
We next calculate the trace
\begin{align}
\label{e:trace_1}
 \operatorname{Tr}_{L^2(\R_x)} \Big[ [\langle w \rangle, \widehat{\mathcal{E}}_\tau]^\perp(\Theta_h)  \mathcal{K}_h(Y,Y') \Big] & =  2 \sum_{k \in \mathbb{N}} \partial_\eta \lambda_k(\Theta_h) \chi\left( \frac{\lambda_k(\Theta_h) - 1}{\tau} \right)  \kappa_k^h(Y,Y'),
\end{align}
where $\kappa_k^h(Y,Y') = \langle \mathcal{K}_h(Y,Y') \varphi_k, \varphi_k \rangle_{L^2(\R_x)}$. On the other hand, we observe that $\big \langle [w, \mathcal{G}_\tau]^\perp \big \rangle = 0$. Indeed, deriving equation
$$
( \widehat{H}_{\eta,\zeta} - \lambda_k ) \varphi_k = 0
$$
with respect to $\eta$, we find
\begin{align}
\label{e:matrix_element_w}
2\langle w \rangle \varphi_k = \partial_\eta \lambda_k \varphi_k, \quad 2 ( w - \langle w \rangle) \, \varphi_k =  (\lambda_k - \widehat{H}_{\eta,\zeta} ) \partial_\eta \varphi_k.
\end{align}
Moreover  $\partial_\eta \varphi_k$ and $(\widehat{H}_{\eta,\zeta} - \lambda_k) \partial_\eta \varphi_k$ belong to the Schwartz class $\mathscr{S}(\R)$ (this is consequence of Agmon estimates (see \cite[Prop. 3.3.4]{Helffer88} and  \cite[\S 3]{CdV_Letrouit22}). Besides that, we have
\begin{align}
 \widehat{\mathcal{G}}_\tau \, \varphi_k & = \frac{1}{i(\lambda_k - \widehat{H}_{\eta,\zeta})} \left( \chi \left( \frac{\lambda_k - 1}{\tau} \right) + \chi\left( \frac{\widehat{H}_{\eta,\zeta} - 1}{\tau} \right) \right) \, w \, \varphi_k \notag \\[0.2cm]
 \label{e:matrix_elements_G}
  & = \frac{1}{i} \left( \chi \left( \frac{\lambda_k - 1}{\tau} \right) + \chi\left( \frac{\widehat{H}_{\eta,\zeta} - 1}{\tau} \right) \right) \partial_\eta \varphi_k.
\end{align}
Thus, combining \eqref{e:matrix_element_w} and \eqref{e:matrix_elements_G}, we get, for each $k \in \mathbb{N}$: 
$$
\big \langle [w, \mathcal{G}_\tau]^\perp \varphi_k, \varphi_k \big \rangle_{L^2(\R_x)} = 0
$$
and hence $\langle  [w, \mathcal{G}_\tau]^\perp \rangle = 0$. We next observe that the localization by the cut-off
$$
\chi \left( \frac{ h\eta - \eta_\jmath}{d} \right) \chi\left( \frac{\lambda_k(\Theta_h) - 1}{\tau} \right)
$$
implies, for $\tau > 0$ sufficiently small, that the only non-vanishing term coming from the sum in \eqref{e:trace_1} is
$$
2\partial_\eta \lambda_{\mathbf{k}}(\Theta_h) \chi\left( \frac{\lambda_{\mathbf{k}}(\Theta_h )- 1}{\tau} \right) \kappa_{\mathbf{k}}^h(Y,Y').
$$

We now proceed to the proof of \eqref{e:weak_1}.  Consider the localization \eqref{e:cut_2} and set $\tau = dhr$. Using Taylor expansion and the fact that $\partial_\eta \lambda_{\mathbf{k}}(\Theta_\jmath) = 0$, we have:
$$
\partial_\eta \lambda_{\mathbf{k}}(\Theta_h) = (h\eta - \eta_\jmath) \partial^2_\eta \lambda_{\mathbf{k}}(\Theta_\jmath) + \mathcal{O}(dhr) + \mathcal{O}(\vert h\eta - \eta_\jmath \vert^2).
$$
Thus, replacing $a$ by $\frac{a}{\eta - \eta_\jmath}$, and taking limits through succesive subsequences, we get:
$$
0 = \int_{\T^2_{y,z} \times \mathbb{S}^0_\theta } \partial_\eta^2 \lambda_{\mathbf{k}}(\eta_\jmath,\zeta_\jmath) \partial_y a_\infty (y,z,\eta_\jmath,\zeta_\jmath,\theta) d\bm{\nu}^\jmath_{1,\mathbf{k}}.
$$
We conclude the proof of \eqref{e:weak_1} by using that
\begin{equation}
\label{e:non_degenerate_critical_value}
\partial_\eta^2 \lambda_{\mathbf{k}}(\eta_\jmath,\zeta_\jmath) = \Lambda_{\mathbf{k}}''( \mu_*) \neq 0.
\end{equation}

Finally, we show \eqref{e:weak_2}. We consider in this case the cut-off \eqref{e:cut_3} and set $\tau = (hr)^2$. Then by Taylor expansion we have:
$$
\partial_\eta \lambda_{\mathbf{k}}(\Theta_h) = (h\eta - \eta_\jmath) \partial^2_\eta \lambda_{\mathbf{k}}(\Theta_\jmath) + \mathcal{O}((hr)^2) + \mathcal{O}(\vert h \eta - \eta_\jmath \vert^2).
$$
Then, taking limits through succesive subsequences, we get:
$$
0 = \partial_\eta^2 \lambda_{\mathbf{k}}(\eta_\jmath, \zeta_\jmath) [D_y^2, M_{1,\mathbf{k}}^\jmath].
$$
By using again \eqref{e:non_degenerate_critical_value}, we obtain \eqref{e:weak_2}.
\end{proof}

\appendix

\section{Spectral properties of $\Delta_{\mathcal{M}}$}
\label{a:spectral_properties}
In this appendix, we give the proof that the operator $\Delta_{\mathcal{M}}$ defined with domain
$$
\mathcal{D}(\Delta_{\mathcal{M}}) = \{ \psi \in L^2_0(M) \, : \, \Delta_{\mathcal{M}} \psi \in L_0^2(M) \}
$$ 
is self-adjoint and has compact resolvent. Notice first that, taking the Fourier transform $\mathcal{F}_{y,z}$ in the variables $(y,z)$, we can decompose the space $L^2_0(M)$ into the direct sum
$$
\mathcal{F}_{y,z} : L^2_0(M) \to \bigoplus_{n \in \mathbb{Z} \times \mathbb{Z}^*} \mathcal{L}_n^2(\R), \quad \mathcal{L}_n^2(\R) := \big \{ \varphi(x) \, \mathbf{e}_n(y,z) \, : \, \varphi \in L^2(\R) \big \},
$$
where  $\mathbf{e}_n(y,z) =e^{i n \cdot(y,z)}/2\pi$. Then the operator $\Delta_\mathcal{M}$ conjugates into a discrete superposition of operators
\begin{equation}
\label{e:unitary_martinet}
\mathcal{F}_{y,z} \, \Delta_\mathcal{M} \, \mathcal{F}_{y,z}^{-1} = \bigoplus_{n \in \mathbb{Z}^2} \mathcal{P}_n,
\end{equation}
where, for each $n = (n_1,n_2) \in \mathbb{Z} \times \mathbb{Z}^*$, the operator $\mathcal{P}_n$ is the quartic oscillator defined on the Hilbert space $\mathcal{L}_n^2(\R) \simeq L^2(\R)$ by
\begin{equation}
\label{e:Montgomery}
\mathcal{P}_k = - \partial_x^2 + ( n_1 + x^2 n_2)^2.
\end{equation}
It is well-known (see for instance \cite[Prop. B.1]{Bahouri23}) that for each $n \in \mathbb{Z} \times \mathbb{Z}^*$,  the quartic oscillator $\mathcal{P}_n$ with domain $\mathcal{D}(\mathcal{P}_n) = \{ \varphi \in L^2(\R) \, : \, \mathcal{P}_n \, \varphi \in L^2(\R) \}$ is self-adjoint and has compact resolvent. Then by \eqref{e:unitary_martinet} and \eqref{e:Montgomery} the operator $\Delta_\mathcal{M}$ is also self-adjoint. Moreover, the spectrum of $\Delta_\mathcal{M}$ satisfies
\begin{align*}
\operatorname{Sp}_{L_0^2(M)} (\Delta_{\mathcal{M}}) & = \bigcup_{n \in \mathbb{Z} \times \mathbb{Z}^*} \operatorname{Sp}_{\mathcal{L}_n^2(\R)} (\mathcal{P}_n).
\end{align*}
Finally, by \eqref{e:Martinet_Montgomery}, the eigenvalues $\lambda_k(n)$ of $\mathcal{P}_n$ are given by
$$
\lambda_k(n) = \vert n_1 \vert^{1/3} \Lambda_k\left( \frac{ n_2 \sgn(n_1)}{\vert n_1 \vert^{2/3}} \right),
$$
where the $\Lambda_k$ are the eigenvalues of the Montgomery family of operators \eqref{e:Montgomery_family}. Then, using the fact $0 < \lambda_k(\mu)$ for every $\mu \in \R$ and that $\lim_{\mu \to \pm\infty} \Lambda_k(\mu) = +\infty$ (see \cite{Helffer_Leautaud22}), we get:
$$
\lim_{\vert n \vert \to + \infty} \vert \lambda_k(n) \vert = + \infty.
$$ 

\section{Semiclassical quantization}
\label{a:appendix}

In this appendix we review some basic results on semiclassical pseudodifferential calculus. We denote by $(X,\Xi) \in T^*M$ the phase-space variables. Let us consider a symbol 
$$
\mathbf{a} \in \mathcal{C}^\infty( \R^3_{X} \times \R^3_\Xi \times \R^3_{X'}).
$$ 
Assume that $\mathbf{a}$ is $2\pi$-periodic in the variable $X$, and it is bounded together with all its derivatives in $\Xi$, and define the semiclassical quantization $\Op_h(\mathbf{a})$ by
\begin{align*}
\Op^{M}_h(\mathbf{a})\psi(X) & :=  \int_{\R^3 \times \R^3} \mathbf{a}(X,h\Xi,X') \psi(X') e^{i \Xi \cdot(X - X' )} \text{d}(X',\Xi),
\end{align*}
for any $\psi \in \mathscr{S}(M)$ in the Schwartz class, where the integral is taken with respect to the measure \eqref{e:measure}.  We can write, for any $\psi, \varphi \in \mathscr{S}(M)$:
$$
\big \langle \Op_h^M(\mathbf{a})\psi , \varphi \big \rangle_{L^2(M)} =  \int_{\R^3 \times \R^3 \times \R^3} \mathbf{a}(X,h\Xi,X') \psi(X') \overline{\varphi}(X) e^{i \Xi \cdot(X - X' )} \text{d}(X',\Xi,X),
$$
with respect to the measure \eqref{e:measure_2}.

\begin{definition}
\label{d:discret_derivatives}
Let $h > 0$ and let $\Xi \in \R \times h^{1/2} \mathbb{Z}^2$. Given $u : \R \times h^{1/2} \mathbb{Z}^2 \to \mathbb{C}$. We define:
\begin{align*}
\partial_\eta^h  u (\Xi) & : = \frac{u(\Xi + h^{1/2} e_\eta) - u(\Xi)}{h^{1/2}}, \quad \overline{\partial}^h_\eta  u(\Xi) : = \frac{u(\Xi) - u(\Xi - h^{1/2} e_\eta)}{h^{1/2}}, \\[0.2cm]
\partial_\zeta^h  u (\Xi) & : = \frac{u(\Xi + h^{1/2} e_\zeta) - u(\Xi)}{h^{1/2}}, \quad \overline{\partial}^h_\zeta  u(\Xi) : = \frac{u(\Xi) - u(\Xi - h^{1/2} e_\zeta)}{h^{1/2}}.
\end{align*}
Here, $e_\eta = (0,1,0) \in \R^3$ and $e_\zeta = (0,0,1) \in \R^3$ denote respcetively the second and third vectors of the canonical basis of $\R^3$.
\end{definition}

 The following is the standard discrete integration by parts Lemma obtained from Abel's transformation.

\begin{lemma} 
\label{l:ipp}
Let $u,v : \R \times h^{1/2} \mathbb{Z}^2 \to \mathbb{C}$. Then:
\begin{align*}
\sum_{\Xi \in \R \times h^{1/2} \mathbb{Z}^2}  \Big( u(\Xi) \partial^h_\eta v(\Xi)  -  \overline{\partial}_\eta^h u(\Xi) v(\Xi) \Big) & = 0, \\[0.2cm]
\sum_{\Xi \in \R \times h^{1/2} \mathbb{Z}^2}  \Big( u(\Xi) \partial^h_\zeta v(\Xi)  - \overline{\partial}_\zeta^h u(\Xi) v(\Xi) \Big) & = 0.
\end{align*}
\end{lemma}

We next show an elementary Lemma which we will use in the proof of Calderón-Vaillancourt theorem, as given in \cite{Hwang87}.
\begin{lemma}{\cite[Lemma 3.1]{Hwang87}}
\label{l:Hwang}
Let $\Gamma \in L^2(\R^3; \operatorname{d})$ and $\Psi \in L^2(M)$. Define:
$$
\mathbf{h}(X,\Xi) := \int_{\R^3} e^{i X \cdot \Lambda} \Gamma(\Lambda - \Xi) \widehat{\Psi}(\Lambda) \operatorname{d}(\Lambda), \quad (X,\Xi) \in M \times \R^3,
$$
where, denoting $\Lambda = (\lambda_1, \lambda_2,\lambda_3)$,
$$
\widehat{\Psi}(\Lambda) := \frac{1}{(2\pi)^{3/2}} \int_M \Psi(X) e^{-iX \cdot \Lambda} dX, \quad \operatorname{d}(\Lambda) := d\lambda_1 \otimes  \left(\sum_{k \in \mathbb{Z}^2} \delta\big( (\lambda_2, \lambda_3) - k \big) \right).
$$
Then:
$$
\Vert \mathbf{h} \Vert_{L^2(M \times \R^3; \operatorname{d})} = \Vert \Gamma \Vert_{L^2(\R^3 ; \operatorname{d})} \Vert \Psi \Vert_{L^2(M)},
$$
where $L^2(M \times \R^3; \operatorname{d})$  is considered with respect to the measure $\operatorname{d}$ defined by \eqref{e:measure}.
\end{lemma}

\begin{proof}
We notice that
\begin{align*}
\Vert \mathbf{h} \Vert_{L^2(M\times \R^3 ; \operatorname{d})}^2 & = \int_{\R^6} \vert \mathbf{h}(X,\Xi) \vert^2 \text{d}(X,\Xi) \\[0.2cm]
 & =   \int_{\R^6} \vert \mathcal{F}_\Xi \mathbf{h}(X,Y) \vert^2 \text{d}(X,Y),
\end{align*}
where $\mathcal{F}_\Xi : L^2(M \times \R^3, \text{d}) \mapsto L^2(M \times M)$ denotes the (inverse) Fourier transform in the variable $\Xi$.
Next, observe that
$$
\mathcal{F}_\Xi \mathbf{h}(X,Y) = \widehat{\Gamma}(Y) \Psi(X - Y).
$$
Thus, by using Fubini's theorem, we get the claim.
\end{proof}

\begin{lemma}
\label{l:Calderon_Vaillancourt}
Let $\mathbf{a} \in \mathcal{C}^\infty( R^3_X \times \R^3_\Xi \times \R^3_{X'})$ be such that $\mathbf{a}$ is $2\pi$-periodic in the variables $(y,z)$ and $(y',z')$, and bounded together with all its derivatives. Then there exists a universal constant $C > 0$ such that
$$
\Vert \Op^{M}_h(\mathbf{a}) \Vert_{\mathcal{L}(L^2(M))} \leq C \sum_{\alpha \in \mathcal{N}} h^{\gamma(\alpha)} \Vert \partial^\alpha \mathbf{a} \Vert_{L^\infty( \R^3_X \times \R^3_\Xi \times \R^3_{X'})},
$$
where the sum runs on 
$$
\alpha = ( \alpha_1,\alpha_2, \beta_1, \beta_2, \alpha'_1, \alpha'_2) \in \mathbb{N}_x \times \mathbb{N}^2_{y,z} \times \mathbb{N}_\xi \times \mathbb{N}^2_{\eta,\zeta} \times \mathbb{N}_{x'} \times \mathbb{N}^2_{y',z'},
$$ 
belonging to the set of indeces
$$
\mathcal{N} := \Big \{ \alpha \in \mathbb{N}^9 \, : \, \vert \alpha_1 \vert, \vert \alpha'_1 \vert, \vert \alpha_2 \vert, \vert \alpha'_2 \vert \leq 1, \quad \vert \beta_1 \vert, \vert \beta_2 \vert \leq 2 \Big \},
$$
and
$$
\gamma(\alpha) = \vert \alpha_1 \vert + \vert \alpha'_1 \vert + \frac{ \vert \alpha_2 \vert + \vert \alpha'_2 \vert + \vert \beta_2 \vert}{2}.
$$
\end{lemma}

\begin{proof}
The proof is an adapatation of the proof of Calderón-Vaillancourt Theorem given in \cite{Hwang87}. For any $\psi, \varphi \in L^2(M)$, let us define
$$
\Psi_h(X) := h \, \psi(hx,h^{1/2}y,h^{1/2}z), \quad \Phi_h(X) := h \, \varphi(hx,h^{1/2}y,h^{1/2}z).
$$
Then:
\begin{align*}
\big \langle \Op^{M}_h(\mathbf{a}) \psi, \varphi \big \rangle_{L^2(M)} & \\[0.2cm]
 & \hspace*{-3cm} = \int_{\R_X \times \R_\Xi \times \R_{X'}} \mathbf{a}_h( X, \Xi, X') \Psi_h(X') \overline{\Phi}_h(X) e^{i \Xi \cdot (X - X')} \text{d}_h(X',\Xi,X) \\[0.2cm]
 & \hspace*{-3cm} =  \int_{\R^3_X \times \R^3_\Xi \times \R^3_{X'} \times \R^6_\Lambda} \mathbf{a}_h( X, \Xi, X') \widehat{\Psi}_h(\Lambda_1) \widehat{\overline{\Phi}}_h(\Lambda_2) e^{i \Xi \cdot (X - X')} e^{i (X' \cdot \Lambda_1 + X \cdot \Lambda_2)}  \text{d}_h(X',\Xi,X) \text{d}(\Lambda_1, \Lambda_2),
\end{align*}
where $\mathbf{a}_h(X,\Xi,X') = \mathbf{a}( hx, h^{1/2}y, h^{1/2}z, \xi, h^{1/2}\eta, h^{1/2} \zeta)$, $\Lambda = (\Lambda_1, \Lambda_2) \in \R^6$, and
$$
\text{d}(\Lambda_1, \Lambda_2) = \text{d}(\Lambda_1) \otimes \text{d}(\Lambda_2).
$$ 
We next use integration by parts ti gain decayment in the variables $(X,\Xi,X')$. To this aim, define the $L^2$ functions:
\begin{align*}
\Gamma_h(X) & := \frac{1}{1 + ix} \cdot \frac{1}{1 + iy_h} \cdot \frac{1}{1 + iz_h}, \\[0.2cm]
\Gamma(\Xi ) & := \frac{1}{1 + i\xi} \cdot \frac{1}{1 + i \eta} \cdot \frac{1}{1 + i\zeta},
\end{align*}
where
\begin{equation}
\label{e:h_points}
 y_h := \frac{e^{i h^{1/2} y}-1}{h^{1/2}} , \quad z_h := \frac{e^{i h^{1/2} z}-1}{h^{1/2}},
\end{equation}
and introduce also the functions
\begin{align*}
\mathbf{g}_h(X,\Xi,X') & :=  \mathbf{a}_h(X,\Xi,X') \Gamma_h^2(X-X') \\[0.2cm]
 \mathbf{h}_1(X',\Xi) & := e^{-i X' \cdot \Xi} \int_{\R^3} e^{i X' \cdot \Lambda} \Gamma(\Lambda - \Xi) \widehat{\Psi}_h(\Lambda) \text{d}(\Lambda), \\[0.2cm]
\mathbf{h}_2(X,\Xi) & := e^{i X \cdot \Xi} \int_{\R^3} e^{i X \cdot \Lambda} \Gamma(\Xi + \Lambda) \widehat{\overline{\Phi}}_h(\Lambda) \text{d}(\Lambda).
\end{align*}
Hence, defining the partial differential operators
\begin{align*}
\mathfrak{D}_\Xi := (1 - \partial_\xi) (1 - \overline{\partial}^h_\eta)(1 - \overline{\partial}^h_\zeta), \\[0.2cm]
\mathfrak{D}_X := (1 - \partial_x) (1 - \partial_y)(1 - \partial_z),
\end{align*}
we obtain, by Lemma \ref{l:ipp} and Lemma \ref{l:Hwang}:
\begin{align*}
\left \vert \big \langle \Op^M_h(\mathbf{a}) \psi, \varphi \big \rangle_{L^2(M)} \right \vert & \\[0.2cm]
 & \hspace*{-2.5cm} = \left \vert \int_{\R^3_X \times \R^3_\Xi \times \R_{X'}^3} \mathfrak{D}_X \mathfrak{D}^2_\Xi \mathfrak{D}_{X'} \mathbf{g}_h(X,\Xi,X')  \, \mathbf{h}_1(X',\Xi) \mathbf{h}_2(X,\Xi)\text{d}_h(X',\Xi,X) \right \vert \\[0.2cm]
& \hspace*{-2.5cm} \leq C \sum_{\alpha \in \mathcal{N}} h^{\gamma(\alpha)} \big \Vert \partial^\alpha \mathbf{a}_h \big \Vert_{L^\infty( \R^3_X \times \R^3_\Xi \times \R^3_{X'})} \Vert \Gamma_h^2 \Vert_{L^1(\R_X^3)}  \left \Vert \mathbf{h}_1 \right \Vert_{L^2(M\times \R^3; \text{d})} \left \Vert \mathbf{h}_1 \right \Vert_{L^2(M \times \R^3; \text{d})} \\[0.2cm]
& \hspace*{-2.5cm} \leq C \sum_{\alpha \in \mathcal{N}} h^{\gamma(\alpha)} \big \Vert \partial^\alpha \mathbf{a}_h \big \Vert_{L^\infty( \R^3_X \times \R^3_\Xi \times \R^3_{X'})} \Vert \psi \Vert_{L^2(M)} \Vert \varphi \Vert_{L^2(M)}.
\end{align*}

\end{proof}

\bibliographystyle{plain}
\bibliography{Referencias}

\end{document}